\documentclass[11pt,reqno]{amsart}

\usepackage[dvipsnames]{xcolor}
\usepackage{amsmath,amssymb,amsthm,amsfonts,enumerate,tikz,bm}
\usepackage{mathrsfs}
\usetikzlibrary{matrix,arrows,positioning,calc}
\usepackage[all]{xy}
\usepackage[bookmarksnumbered,colorlinks]{hyperref}
\usepackage{url}
\numberwithin{equation}{section}
\usepackage{graphicx}

\usepackage[T1]{fontenc}

\theoremstyle{definition}
\newtheorem{definition}{Definition}[subsection]
\newtheorem{remark}[definition]{Remark}
\newtheorem{example}[definition]{Example}

\theoremstyle{plain}
\newtheorem{theorem}[definition]{Theorem}
\newtheorem{proposition}[definition]{Proposition}
\newtheorem{lemma}[definition]{Lemma}
\newtheorem{corollary}[definition]{Corollary}

\newtheorem{theorem*}{Theorem}

\theoremstyle{remark}

\usepackage{paralist}

\def\fd{\mathop{\rm fd}\nolimits}

\def\sup{\mathop{\rm sup}\nolimits}

\def\hom{\mathop{\rm \mathscr{H}om}\nolimits}

\def\sup{\mathop{\rm sup}\nolimits}
\def\lim{\mathop{\underrightarrow{\rm lim}}\nolimits}

\newcommand{\Tor}{\operatorname{Tor}}
\newcommand{\Mod}{\operatorname{Mod}}
\newcommand{\Hom}{\operatorname{Hom}}
\newcommand{\Ext}{\operatorname{Ext}}

\newcommand{\op}{{}^{\operatorname{op}}}

\newcommand{\id}{\operatorname{id}}

\newcommand{\ra}{\rightarrow}

\newcommand{\FP}{\mbox{FP}}
\newcommand{\Ch}{\operatorname{Ch}}
\def\hom{\mathop{\rm \mathscr{H}{\it om}}\nolimits}

\makeatletter
\def\@seccntformat#1{%
  \protect\textup{\protect\@secnumfont
    \ifnum\pdfstrcmp{section}{#1}=0 \scshape\bfseries\fi
    \ifnum\pdfstrcmp{subsection}{#1}=0 \bfseries\fi
    \csname the#1\endcsname
    \protect\@secnumpunct
  }%
}
\makeatother

\begin{document}

\title{Relative FP-injective and FP-flat complexes \\ and their Model Structures}

\author{Tiwei Zhao}
\address[T. Zhao]{Department of Mathematics. Nanjing University. Nanjing 210093. PEOPLE'S REPUBLIC OF CHINA}
\email{tiweizhao@hotmail.com}

\author{Marco A. P\'erez}
\address[M. A. P\'erez]{Instituto de Matem\'aticas. Universidad Nacional Aut\'onoma de M\'exico. Mexico City 04510. MEXICO}
\email{maperez@im.unam.mx}

\baselineskip=14pt
\maketitle

\begin{abstract}
In this paper, we introduce the notions of $\FP_n$-injective and $\FP_n$-flat complexes in terms of complexes of type $\FP_n$. We show that some characterizations analogous to that of injective, FP-injective and flat complexes exist for $\FP_n$-injective and $\FP_n$-flat complexes. We also introduce and study  $\FP_n$-injective and $\FP_n$-flat dimensions of modules and complexes, and give a relation between them in terms of Pontrjagin duality. The existence of pre-envelopes and covers in this setting is discussed, and we prove that any complex has an $\FP_n$-flat cover and an $\FP_n$-flat pre-envelope, and in the case $n\geq 2$ that any complex has an $\FP_n$-injective cover and an $\FP_n$-injective  pre-envelope. Finally, we construct model structures on the category of complexes from the classes of modules with bounded $\FP_n$-injective and $\FP_n$-flat dimensions, and analyze several conditions under which it is possible to connect these model structures via Quillen functors and Quillen equivalences.
\end{abstract}


\pagestyle{myheadings}
\markboth{\rightline {\scriptsize T. Zhao and M. A. P\'{e}rez}}
         {\leftline{\scriptsize Relative FP-injective and FP-flat complexes and their model structures}}


\section*{\textbf{Introduction}}

Throughout this paper, $R$ denotes an associative ring with unit, $\Mod(R)$ (resp., $\Mod(R\op)$) denotes the category of all left (resp., right) $R$-modules, and $\Ch(R)$ (resp., $\Ch(R\op)$) denotes the  category of all complexes of left (resp., right) $R$-modules. We denote by $(\bm{X},\delta)$, or simply by $\bm{X}$, a chain complex
\[
\bm{X} = \cdots \to X_2 \xrightarrow{\delta_{2}^{\bm{X}}} X_1 \xrightarrow{\delta_1^{\bm{X}}} X_0 \xrightarrow{\delta^{\bm{X}}_0}  X_{-1} \xrightarrow{\delta^{\bm{X}}_{-1}}  \cdots
\]
in $\Ch(R)$ (or $\Ch(R\op)$), and by $Z(\bm{X})$ and $B(\bm{X})$ the sub-complexes of cycles and boundaries of $\bm{X}$, respectively. For more background material, we refer the reader to \cite{EJ11,GR99,We94}.

The category $\Ch(R)$ plays an important role in homological algebra, and it has been studied by many authors (see, for example \cite{AEGO01, AF91, EG97,EG98,EJ11, GR99, WL11,Ya,YL10}), and many results in $\Mod(R)$ have been generalized to $\Ch(R)$. As we know, injective and flat complexes are key in the study of $\Ch(R)$, and they have a closed relation with injective and flat modules respectively. For example, a complex $\bm{X}$ in $\Ch(R)$ is injective (resp., flat) if, and only if, $\bm{X}$ is exact and $Z_m(\bm{X})$ is injective (resp., flat) as a left $R$-module for any $m \in \mathbb{Z}$. In 1970, Stenstr\"{o}m \cite{St70} introduced the notion of FP-injective modules, which generalizes that of injective modules, and using it, he gave the homological properties over coherent rings analogous to that of injective modules over Noetherian rings. In \cite{WL11, YL10}, Liu et al. introduced the notion of FP-injective complexes. They obtained many nice characterizations of them over coherent rings, and they showed that some properties of injective complexes have counterparts for FP-injective complexes. Recently, Gao and Wang \cite{GW15} introduced the notions of weak injective and weak flat modules, which are further generalizations of FP-injective modules and flat modules. Independently, from the viewpoint of model structures, D. Bravo, J. Gillespie and M. Hovey \cite{BGH14} also investigated these classes of modules, and in their paper, they called them absolutely clean (or $\FP_\infty$-injective) and level (or $\FP_\infty$-flat) modules, respectively. The counterpart of the corresponding complexes was also introduced and investigated in \cite{BG16,GH16}. It seems that there is  a gap between 1 and $\infty$, and it is in this gap where one can extend essential aspects from coherent rings to arbitrary rings. In fact, in some cases, the parameter `2' is enough to obtain a lot of information (See for example \cite[Section 3]{BP16}). Recently, Bravo and the second author introduced and investigated in \cite{BP16} $\FP_n$-injective and $\FP_n$-flat modules for each non-negative integer $n$, and generalized many results from coherent rings to  $n$-coherent rings by using them. In this process, finitely presented modules are replaced by finitely $n$-presented modules. As a summary to the above work, we outline a diagram to reflect the intrinsic relation between these concepts as follows:
\[
{\footnotesize
\xymatrix @C=0.3cm @R=1.15cm {\mbox{Injective module}\ar@{=>}[r]\ar@{:>}[d]&\mbox{FP-injective module}\ar@{=>}[r]\ar@{:>}[d]&\cdots\ar@{=>}[r]&\mbox{$\FP_n$-injective module}\ar@{=>}[r]\ar@{:>}[d]&\cdots\ar@{=>}[r]&\mbox{$\FP_\infty$-injective module}\ar@{:>}[d]\\
\mbox{Injective complex}\ar@{=>}[r]&\mbox{FP-injective complex}\ar@{=>}[r]&\cdots\ar@{=>}[r]&{\mbox{~~~~~ ?~~~~~ }}\ar@{=>}[r]&\cdots\ar@{=>}[r]&\mbox{$\FP_\infty$-injective complex}}
}
\]
Following the above philosophy, it is natural to extend the notions of $\FP_n$-injective and $\FP_n$-flat  modules to $\Ch(R)$, and then to establish a relation between the $\FP_n$-injectivity (resp., $\FP_n$-flatness) of a complex and that of its cycles.

The structure of this paper is as follows:
\begin{itemize}
\item[$\bullet$] In Section \ref{sec:prelim}, we  recall  some notions and terminologies needed in this article.

\item[$\bullet$] Section \ref{sec:FPninjflat} is devoted to introducing the notion of complexes of type $\FP_n$ for some non-negative integer $n$, and give some characterizations for $n$-coherent rings in terms of a stable condition of complexes of type $\FP_n$. Then, we  introduce the notions of $\FP_n$-injective and $\FP_n$-flat complexes in terms of complexes of type $\FP_n$. We will obtain a description of $\FP_n$-injective complexes (resp., $\FP_n$-flat complexes) in terms of their exactness and the injectivity (resp., flatness) of their cycles relative to the class of modules of type $\FP_n$, among other homological properties (See Theorems \ref{tfae1} and \ref{theo:char_flat}).

\item[$\bullet$] In Section \ref{sec:FPninjflatdimensions}, we present and characterize the $\FP_n$-injective and $\FP_n$-flat dimensions of (left and right) modules and complexes, denoted $\FP_n\mbox{-}\id_R(M)$ and $\FP_n\mbox{-}\fd_{R\op}(N)$ for $M$ in $\Mod(R)$ and $N$ in $\Mod(R\op)$, and by $\FP_n\mbox{-}\id(\bm{X})$ and $\FP_n\mbox{-}\fd(\bm{Y})$ for $\bm{X}$ in $\Ch(R)$ and $\bm{Y}$ in $\Ch(R\op)$. In the contexts of complexes, we prove that $\FP_n\mbox{-}\id(\bm{X}) \leq m$ if, and only if, $\bm{X}$ is exact and $\FP_n\mbox{-}\id_R (Z_i(\bm{X}))\leq m$ for any $i \in \mathbb{Z}$, along with a dual characterization for the $\FP_n$-flat dimension. As a consequence, we get that if $\bm{X}$ and $\bm{Y}$ are exact complexes in $\Ch(R)$ and $\Ch(R\op)$, respectively, then:
\begin{align*}
\FP_n\mbox{-}\id(\bm{X}) & = \sup\{\FP_n\mbox{-}\id_R(Z_m(\bm{X})) \mbox{ : } m \in \mathbb{Z}\}, \\
\FP_n\mbox{-}\fd(\bm{Y}) & = \sup\{\FP_n\mbox{-}\fd_{R\op}(Z_m(\bm{Y})) \mbox{ : } m \in \mathbb{Z}\}.
\end{align*}
Moreover, we prove that:
\begin{align*}
\FP_n\mbox{-}\fd(\bm{Y}) & = \FP_n\mbox{-}\id(\bm{Y}^+) \mbox{ for every $n \geq 0$}, \\
\FP_n\mbox{-}\id(\bm{X}) & = \FP_n\mbox{-}\fd(\bm{X}^+) \mbox{ for every $n \geq 2$},
\end{align*}
where $\bm{X}^+$ and $\bm{Y}^+$ denote the Pontrjagin dual of $\bm{X}$ and $\bm{Y}$ in $\Ch(R\op)$ and $\Ch(R)$, respectively.

\item[$\bullet$] Denote by $\mathcal{F}_{(n,k)}(R\op)$ (resp., $\mathcal{I}_{(n,k)}(R)$) the class of modules in $\Mod(R\op)$ (resp., in $\Mod(R)$) with $\FP_n$-flat (resp., $\FP_n$-injective) dimension at most $k$, and the corresponding classes in $\Ch(R\op)$ and $\Ch(R)$ by $\mathscr{F}_{(n,k)}(R\op)$ and $\mathscr{I}_{(n,k)}(R)$. In Section \ref{sec:InjFlatCom}, we show that the pair $(\mathcal{F}_{(n,k)}(R\op),\mathcal{I}_{(n,k)}(R))$ is a duality pair over $R$ for every $n \geq 0$; and that the same holds for the pair $(\mathcal{I}_{(n,k)}(R),\mathcal{F}_{(n,k)}(R\op))$ in the case $n \geq 2$. We later prove that these results carry over to $\Ch(R)$, by using a method to inducing three different dual pairs of complexes from a duality pair of modules (See Theorem \ref{theo:induced_duality}). After constructing these duality pairs, we use some results by H. Holm and P. J\o rgensen \cite{duality} and by X. Yang in \cite{Ya} on duality pairs to obtain covers and pre-envelopes associated to the previous classes.

\item[$\bullet$] The final Section \ref{sec:models} is devoted to constructing several model structures on $\Ch(R)$ associated to the classes $\mathcal{I}_{(n,k)}(R)$ and $\mathcal{F}_{(n,k)}(R\op)$. The method we apply is the so called Hovey's correspondence, along with several techniques developed by Gillespie to induce cotorsion pairs in the category of complexes from a cotorsion pair of modules. We also study the possibility to obtaining Quillen equivalences between these new model structures, from the identity, induction and restriction functors, by analyzing certain conditions on the ground rings, and using it, we can judge whether or not a ring $R$ and its opposite ring $R\op$ are derived equivalent.
\end{itemize}

Throughout this paper, the results stated in the categories $\Mod(R)$ and $\Ch(R)$ will be also valid in $\Mod(R\op)$ and $\Ch(R\op)$, and viceversa.


\section[Preliminaries]{\textbf{Preliminaries}}\label{sec:prelim}

In this paper, we mainly use the superscripts to distinguish complexes and the subscripts for a complex components. For example, if $\{\bm{X^i}\}_{i\in I}$ is a family of complexes in $\Ch(R)$, then $X^i_n$ denotes the degree-$n$ term of the complex $\bm{X^i}$. Given a left $R$-module $M$, we denote by $D^n(M)$ the \emph{$n$-disk on $M$}, that is, the complex
\[
\cdots \to 0 \to M \xrightarrow{\id} M \to 0 \to \cdots
\]
with $M$ in the $n$-th and $(n\mbox{-}1)$-st positions; and by $S^n(M)$ the \emph{$n$-sphere on $M$}, that is, a complex with $M$ in the $n$-th position and 0 everywhere else. Given a complex $\bm{X}$ in $\Ch(R)$ and an integer $m$, $\bm{X}[m]$ denotes the complex such that $X[m]_n = X_{n-m}$, and whose boundary operators are
\[
\delta^{\bm{X}[m]}_n := (-1)^m\delta_{n-m}^{\bm{X}}.
\]
The complex $\bm{X}[m]$ is usually referred as the \emph{$m$-th suspension} of $\bm{X}$.

For complexes $\bm{X}$ and $\bm{Y}$ in $\Ch(R)$, $\Hom_{\Ch}(\bm{X},\bm{Y})$ (or $\Hom(\bm{X},\bm{Y})$ for short) is the abelian group of morphisms from $\bm{X}$ to $\bm{Y}$ in the category of complexes, and $\Ext^i_{\Ch}(\bm{X},\bm{Y})$ (or $\Ext^i(\bm{X},\bm{Y})$ for short) for $i\geq 1$ will denote the extension groups we get from the right derived functors of $\Hom(-,-)$. We will frequently consider the sub-group $\Ext^1_{\rm dw}(\bm{X},\bm{Y})$ of $\Ext^1_{\Ch}(\bm{X},\bm{Y})$ formed by those short exact sequences $\bm{0} \to \bm{Y} \to \bm{Z} \to \bm{X} \to \bm{0}$ in $\Ch(R)$ which are split exact at the module level. Note that $\Ext^1_{\Ch}(\bm{X},\bm{Y}) = \Ext^1_{\rm dw}(\bm{X},\bm{Y})$ if $\Ext^1_R(X_m,Y_m) = 0$ for every $m \in \mathbb{Z}$.

We let $\hom(\bm{X},\bm{Y})$ be the complex of abelian groups
\[
\cdots \xrightarrow{\delta^{\hom(\bm{X},\bm{Y})}_{n+1}} \prod_{i\in \mathbb{Z}}\Hom_R(X_i,Y_{n+i}) \xrightarrow{\delta^{\hom(\bm{X},\bm{Y})}_n} \prod_{i\in \mathbb{Z}}\Hom_R(X_i,Y_{n-1+i}) \xrightarrow{\delta^{\hom(\bm{X},\bm{Y})}_{n-1}} \cdots
\]
(where $\mathbb{Z}$ is the additive group of integers) such that if $f\in \hom(\bm{X},\bm{Y})_n$, then
\[
(\delta^{\hom(\bm{X},\bm{Y})}_n(f))_m  := \delta_{n+m}^{\bm{Y}} \circ f_m - (-1)^n f_{m-1} \circ \delta_m^{\bm{X}}.
\]
This construction defines a bifunctor $\hom(-,-)$, which is the internal-hom of a closed monoidal structure $(\Ch(R),\otimes)$. Here, the tensor product $\bm{Z} \otimes \bm{Y}$ of two complexes $\bm{Z}$ in $\Ch(R\op)$ and $\bm{Y}$ in $\Ch(R)$ is defined as the complex
\[
\bm{Z} \otimes \bm{Y} := \cdots \xrightarrow{\delta^{\bm{Z} \otimes \bm{Y}}_{n+1}} \bigoplus_{k \in \mathbb{Z}} Z_k \displaystyle\operatorname*{\otimes}_R Y_{n-k} \xrightarrow{\delta^{\bm{Z} \otimes \bm{Y}}_n} \bigoplus_{k \in \mathbb{Z}} Z_k \displaystyle\operatorname*{\otimes}_R Y_{n-1-k} \xrightarrow{\delta^{\bm{Z}\otimes\bm{Y}}_{n-1}} \cdots
\]
where the boundary operators $\delta^{\bm{Z} \otimes \bm{Y}}_n \colon (\bm{Z} \otimes \bm{Y})_n \to (\bm{Z} \otimes \bm{Y})_{n-1}$ are defined, for every generator $z \otimes y$ with $z \in Z_k$ and $y \in Y_{n-k}$ as
\[
\delta^{\bm{Z} \otimes \bm{Y}}_n(z \otimes y) := \delta^{\bm{Z}}_k(z) \otimes y + (-1)^k z \otimes \delta^{\bm{Y}}_{n-k}(y).
\]
On the other hand, let $\underline{\Hom}(\bm{X},\bm{Y}) = Z(\hom(\bm{X},\bm{Y}))$. Then $\underline{\Hom}(\bm{X},\bm{Y})$ can be made into a complex with $\underline{\Hom}(\bm{X},\bm{Y})_n$ the abelian group of morphisms from $\bm{X}$ to $\bm{Y}[n]$ and with a boundary operator given by $\delta^{\underline{\Hom}(\bm{X},\bm{Y})}_n(f) \colon \bm{X} \ra \bm{Y}[n-1]$, where $f \in \underline{\Hom}(\bm{X},\bm{Y})_n$ and
\[
(\delta^{\underline{\Hom}(\bm{X},\bm{Y})}_n (f))_m := (-1)^n \delta^{\bm{Y}} \circ f_m
\]
for any $m\in \mathbb{Z}$. As it happens with $\hom(-,-)$, the previous construction defines a bifunctor $\underline{\Hom}(-,-)$ which turns out to be the internal-hom of a closed monoidal structure $(\Ch(R),\overline{\otimes})$, where the tensor product $\bm{Z} \overline{\otimes} \bm{Y}$ of $\bm{Z}$ in $\Ch(R\op)$ and $\bm{Y}$ in $\Ch(R)$ is defined as the chain complex
\[
\bm{Z} \overline{\otimes} \bm{Y} := \cdots \xrightarrow{\delta^{\bm{Z} \overline{\otimes} \bm{Y}}_{n+1}} \frac{(\bm{Z} \otimes \bm{Y})_n}{B_n(\bm{Z} \otimes \bm{Y})} \xrightarrow{\delta^{\bm{Z} \overline{\otimes} \bm{Y}}_n} \frac{(\bm{Z} \otimes \bm{Y})_{n-1}}{B_{n-1}(\bm{Z} \otimes \bm{Y})} \xrightarrow{\delta^{\bm{Z} \overline{\otimes} \bm{Y}}_{n-1}} \cdots
\]
where the boundary operators $\partial^{\bm{Z} \overline{\otimes} \bm{Y}}_n$ are given for every coset $z \otimes y + B_n(\bm{Z} \otimes \bm{Y})$ with $z \in Z_k$ and $y \in Y_{n-k}$ by
\[
\delta^{\bm{Z} \overline{\otimes} \bm{Y}}_n(x \otimes y + B_n(\bm{Z} \otimes \bm{Y})) := \partial^{\bm{Z}}_k(z) \otimes y + B_{n-1}(\bm{Z} \otimes \bm{Y}).
\]
The functor $- \overline{\otimes} \bm{Y}$ is right exact, for every $\bm{Y} \in \Ch(R)$, so we can construct the corresponding left derived functors $\overline{\Tor}_i(-,\bm{Y})$ with $i \geq 0$. Note also that the bifunctor $\underline{\Hom}(-,-)$ has right derived functors whose values will be complexes. These values are denoted by $\underline{\Ext}^i(\bm{X},\bm{Y})$. One can see that $\underline{\Ext}^i(\bm{X},\bm{Y})$ is the complex
\begin{align}\label{eqn:underext}
\underline{\Ext}^i(\bm{X},\bm{Y}) & = \cdots \ra \Ext^i(\bm{X},\bm{Y}[n+1]) \ra \Ext^i(\bm{X},\bm{Y}[n]) \ra \Ext^i(\bm{X},\bm{Y}[n-1]) \ra \cdots
\end{align}
with boundary operator induced by the boundary operator of $\bm{Y}$. For a detailed proof, we refer the reader to \cite[Proposition 4.4.7]{PerezBook}. For any complex $\bm{X}$ in $\Ch(R)$, the \emph{character} or \emph{Pontrjagin dual} complex is defined, according to \cite{GR99}, by
\[
\bm{X}^+ := \underline{\Hom}(\bm{X},D^1({\mathbb{Q}/\mathbb{Z}})),
\]
a complex in $\Ch(R\op)$, where $\mathbb{Q}$ is the additive group of rational numbers. There is an equivalent definition of $\bm{X}^+$ which will be used in the sequel. Namely, according to \cite[Proposition 4.4.10]{PerezBook}, we have that
\begin{align}\label{eqn:alter_Pontrjagin}
\bm{X}^+ & \simeq \cdots \to \Hom_\mathbb{Z}(X_{-n-2},\mathbb{Q} / \mathbb{Z}) \xrightarrow{\delta_n} \Hom_\mathbb{Z}(X_{-n-1},\mathbb{Q} / \mathbb{Z}) \to \cdots
\end{align}
where the boundary operators are given by the formula $\delta_n := (-1)^n\Hom_\mathbb{Z}(\delta^{\bm{X}}_{-n-1},\mathbb{Q} / \mathbb{Z})$. We denote the complex on the right side by $\bm{X^\ast}$.

Following \cite{En81}, for any sub-category $\mathcal{F}$ of an abelian category $\mathscr{A}$, a morphism $f \colon F \ra M$ in $\mathscr{A}$ with $F \in \mathcal{F}$ is called an $\mathcal{F}$-\emph{pre-cover} of $M$ if for any morphism $f' \colon F' \ra M$ in $\mathscr{A}$ with $F' \in \mathcal{F}$, there exists a morphism $h \colon F' \ra F$ such that the following diagram commutes:
\[
\begin{tikzpicture}[description/.style={fill=white,inner sep=2pt}]
\matrix (m) [matrix of math nodes, row sep=3em, column sep=3em, text height=1.5ex, text depth=0.5ex]
{
{} & F' \\
F & M \\
};
\path[->]
(m-1-2) edge node[right] {\footnotesize$f'$} (m-2-2)
(m-2-1) edge node[below] {\footnotesize$f$} (m-2-2)
;
\path[dotted,->]
(m-1-2) edge node[sloped,description] {\footnotesize$\exists \mbox{ } h$} (m-2-1)
;
\end{tikzpicture}
\]
The morphism $f \colon F \to M$ is called \emph{right minimal} if an endomorphism $h \colon F \to F$ is an automorphism whenever $f = f \circ h$. An $\mathcal{F}$-pre-cover $f \colon F \to M$ is called an $\mathcal{F}$-\emph{cover} if $f$ is right minimal. An $\mathcal{F}$-pre-cover $f \colon F \to M$ in $\mathscr{A}$ is called \emph{special} if it is epic and if $\Ext^1_{\mathscr{A}}(F',{\rm Ker}(f)) = 0$ for every $F' \in \mathcal{F}$. The sub-category $\mathcal{F}$ is called a \emph{(special) (pre-)covering} in $\mathscr{A}$ if every object in $\mathscr{A}$ has an (special) $\mathcal{F}$-(pre-)cover. Dually, one has the notions of \emph{(special)} $\mathcal{F}$-\textit{(pre-)envelopes}, \textit{left minimal morphisms} and \emph{(special)} \textit{(pre-)enveloping sub-categories}.

The previous notions are closely related to the concept of cotorsion pairs. Two classes $\mathcal{A}$ and $\mathcal{B}$ of objects in $\mathscr{A}$ form a \emph{cotorsion pair} $(\mathcal{A,B})$ if
\begin{align*}
\mathcal{A} & = {}^\perp\mathcal{B} := \{ A \in \mathscr{A} \mbox{ : } \Ext^1_{\mathscr{A}}(A,B) = 0 \mbox{ for every } B \in \mathcal{B} \}, \\
\mathcal{B} & = \mathcal{A}^\perp := \{ B \in \mathscr{A} \mbox{ : } \Ext^1_{\mathscr{A}}(A,B) = 0 \mbox{ for every } A \in \mathcal{A} \}.
\end{align*}
A cotorsion pair $(\mathcal{A,B})$ is \emph{complete} if every object in $\mathscr{A}$ has an special $\mathcal{A}$-pre-cover and a special $\mathcal{B}$-pre-envelope. All of the cotorsion pairs presented in this paper will be complete, and one way to showing this to provide a cogenerating set. A cotorsion pair $(\mathcal{A,B})$ in $\mathscr{A}$ is said to be \emph{cogenerated by a set} $\mathcal{S}$ if $\mathcal{B} = \mathcal{S}^\perp$. Due to the Eklof and Trlifaj's Theorem \cite{EkTr}, we know that every cotorsion pair cogenerated by a set is complete. As complete cotorsion pairs are related to special pre-cover and special pre-envelopes, the analogous type of cotorsion pair for covers and envelopes is known as \emph{perfect}, that is, a cotorsion pair $(\mathcal{A,B})$ in $\mathscr{A}$ such that every object in $\mathscr{A}$ has an $\mathcal{A}$-cover and a $\mathcal{B}$-envelope. In order to show that a cotorsion pair is perfect, it suffices to verify that it is complete and that $\mathcal{A}$ is closed under direct limits (See \cite[Corollary 2.3.7]{Gobel}).


\subsection{Finiteness of modules and chain complexes}

In order to generalize the homological properties from noetherian rings to coherent rings, Stenstr\"{o}m  \cite{St70} introduced the notion of FP-injective modules as follows.

\begin{definition}
A module $M$ in $\Mod(R)$ is called \emph{$\FP$-injective} if $\Ext_{R}^1(L,M) = 0$ for all finitely presented modules $L$ in $\Mod(R)$.
\end{definition}

To give an extension of homological algebra to arbitrary rings, one of the key problems is to increase the length of finitely generated projective resolutions of modules. So the following definition from \cite{BGH14} and \cite{GW15} suits this purpose.

\begin{definition}
A module $L$ in $\Mod(R)$  is said to be of \emph{type $\FP_\infty$} (or \emph{super finitely presented}) if there exists an exact sequence $\cdots \to P_n \to P_{n-1} \to \cdots \to P_1 \to P_0 \to L \to 0$ in $\Mod(R)$, where each $P_i$ is finitely generated projective.
\end{definition}

We denote by $\mathcal{FP}_\infty(R)$ the class of modules in $\Mod(R)$ of type $\FP_\infty$.

Using the previous concept, Bravo, Gillespie and Hovey \cite{BGH14} and independently, Z. Gao and F. Wang \cite{GW15} introduced the following extension of the notions of FP-injective and flat modules.

\begin{definition} Let $M$ be a module in $\Mod(R)$ and $N$ be a module in $\Mod(R\op)$.
\begin{itemize}
\item[(a)] $M$ is called \emph{absolutely clean} (or \emph{weak injective}) if $\Ext_{R}^1(L,M) = 0$ for all $L \in \FP_\infty(R)$.

\item[(b)] $N$ is called \emph{level} (or \emph{weak flat}) if $\Tor^R_1(N,L) = 0$ for all $L \in \FP_\infty(R)$.
\end{itemize}
\end{definition}

We will denote by $\mathcal{I}_\infty(R)$ and $\mathcal{F}_\infty(R\op)$ the classes of absolutely clean and level modules in $\Mod(R)$ and $\Mod(R\op)$, respectively.

To extend the homological properties related to finiteness from modules to complexes, the key step is to give the counterpart for complexes of the above definitions, and we will list them below. We begin recalling from \cite[Definition 2.1]{EG98} and \cite[Definitions 1.3.1 and 1.3.2]{EJ11} the definitions of finitely generated and finitely presented complexes.

\begin{definition} \
\begin{itemize}
\item[{\rm (a)}] A \emph{graded set} $G$ is a family of sets $\{ G_m \mbox{ : } m \in \mathbb{Z} \}$ such that $G_m \cap G_n = \emptyset$ whenever $m \neq n$. If $G$ and $H$ are graded sets, a \emph{morphism $f \colon G \to H$ of degree $p$} is a family of functions of the form $f_m \colon G_m \to H_{m+p}$ with $m \in \mathbb{Z}$. Given a graded set $G$ and a complex $\bm{X}$ in $\Ch(R)$, the notation $G \subseteq \bm{X}$ means $G_n \subseteq X_n$ for every $n \in \mathbb{Z}$. In this case, a sub-complex $\bm{Y} \subseteq \bm{X}$ is the \emph{sub-complex generated by $G$} if $\bm{Y}$ is the intersection of all sub-complexes of $\bm{X}$ containing $G$. A complex $\bm{X}$ is said to be \emph{finitely generated} if one of the following equivalent conditions is satisfied:
	\begin{itemize}
	\item[(a.1)] There exists a finite graded set $G \subseteq \bm{X}$ that generates $\bm{X}$.
	
	\item[(a.2)] Whenever $\bm{X} = \sum_{i \in I} \bm{S}^i$ for some collection $\{ \bm{S}^i \}_{i \in I}$ of sub-complexes of $\bm{X}$, then there exists a finite subset $J \subseteq I$ for which $\bm{X} = \sum_{i \in J} \bm{S}^i$
	\end{itemize}

\item[(b)] A complex $\bm{X}$ is called \emph{finitely presented} if $\bm{X}$ is finitely generated and for each short exact sequence $\bm{0} \to \bm{K} \to \bm{Y} \to \bm{X} \to \bm{0}$ in $\Ch(R)$ with $\bm{Y}$ finitely generated projective, $\bm{K}$ is also finitely generated; or equivalently, if there is an exact sequence $\bm{P_1} \to \bm{P_0} \to \bm{X} \to \bm{0}$ in $\Ch(R)$ such that $\bm{P_0}$ and $\bm{P_1}$ are finitely generated and projective.
\end{itemize}
\end{definition}

We have the following characterization from \cite[Lemma 2.2]{EG98} of finitely generated and finitely presented complexes.

\begin{lemma}\label{fg-fp} The following equivalences hold for any complex $\bm{X}$ in $\Ch(R)$:
\begin{itemize}
\item[{\rm (a)}] $\bm{X}$ is finitely generated if, and only if, it is bounded and each term $X_m$ is finitely generated in $\Mod(R)$.

\item[{\rm (b)}] $\bm{X}$ is finitely presented if, and only if, it is bounded and each term $X_m$ is finitely presented in $\Mod(R)$.
\end{itemize}
\end{lemma}

Now recall from \cite[Definition 2.1]{BG16} and \cite[Definition 3.1]{GH16} the following.

\begin{definition}
A complex $\bm{X}$ is said to be of \emph{type $\FP_\infty$} (or, \emph{super finitely presented}) if there exists an exact sequence $\cdots \to \bm{P_n} \to \bm{P_{n-1}} \to \cdots \to \bm{P_1} \to \bm{P_0} \to \bm{X} \to \bm{0}$ in $\Ch(R)$, where each $\bm{P_i}$ is finitely generated projective.
\end{definition}

We denote by $\mathscr{FP}_\infty(R)$ the class of all complexes in $\Ch(R)$ of type $\FP_\infty$.

Recall from \cite{YL10} and \cite{BG16,GH16} the following.

\begin{definition}
Let $\bm{X}$ be a complex in $\Ch(R)$ and $\bm{Y}$ a complex in $\Ch(R\op)$.
\begin{itemize}
\item[(a)] $\bm{X}$ is called  \emph{$\FP$-injective} if $\underline{\Ext}^1(\bm{L},\bm{X}) = \bm{0}$ for all finitely presented $\bm{L}$ in $\Ch(R)$.

\item[(b)] $\bm{X}$ is called \emph{absolutely clean} (or \emph{weak injective}) if $\underline{\Ext}^1(\bm{L},\bm{X}) = \bm{0}$ for all $\bm{L} \in \mathscr{FP}_\infty(R)$.

\item[(c)] $\bm{Y}$ is called \emph{level} (or \emph{weak flat}) if $\overline{\Tor}_1(\bm{Y},\bm{L}) = \bm{0}$ for all $\bm{L} \in \mathscr{FP}_\infty(R)$.
\end{itemize}
\end{definition}

To investigate the homological nature of finiteness of modules more precisely, Bravo and the second author \cite{BP16} studied the following class of modules.

\begin{definition}
A module $M$ in $\Mod(R)$ is called of \emph{type $\FP_n$} (or \emph{finitely $n$-presented}) if there exists an exact sequence
\begin{align} \label{eqn:nprepmodule}
P_n & \to P_{n-1} \to \cdots \to P_1 \to P_0 \to M \to 0
\end{align}
in $\Mod(R)$ where each $P_i$ is finitely generated and projective.
\end{definition}

We denote by $\mathcal{FP}_n(R)$ the class of all left $R$-modules of type $\FP_n$.

The injectivity and flatness associated to finitely $n$-presented modules were defined in \cite[Definitions 3.1 and 3.2]{BP16} as follows.

\begin{definition} Let $M$ be a module in $\Mod(R)$ and $N$ be a module in $\Mod(R\op)$.
\begin{itemize}
\item[(a)] $M$ is called \emph{$\FP_n$-injective} if $\Ext_R^1(L,M) = 0$ for all $L \in \mathcal{FP}_n(R)$.

\item[(b)] $N$ is called \emph{$\FP_n$-flat} if ${\Tor}_1^R(N,L) = 0$ for all $L \in \mathcal{FP}_n(R)$.
\end{itemize}
\end{definition}

We denote by $\mathcal{I}_n(R)$ the class of $\FP_n$-injective modules  in $\Mod(R)$, and by $\mathcal{F}_n(R\op)$ the class of $\FP_n$-flat modules in $\Mod(R\op)$. Note that the $\FP_0$-injective modules coincide with the  injective modules, the $\FP_1$-injective modules coincide with the  $\FP$-injective or absolutely pure modules, and the $\FP_i$-flat modules are the flat modules for $i=0,1$.


\section{\textbf{$\FP_{n}$-injective and $\FP_{n}$-flat complexes}}\label{sec:FPninjflat}

In this section, we first  introduce the notion of complexes of type $\FP_n$ for some non-negative integer $n$, and give some characterizations for $n$-coherent rings in terms of the stable condition of complexes of type $\FP_n$. Later, we  introduce the notions of $\FP_n$-injective and $\FP_n$-flat complexes in terms of complexes of type $\FP_n$, and study the relation between $\FP_n$-injective (resp., $\FP_n$-flat) complexes and $\FP_n$-injective (resp., $\FP_n$-flat) modules.


\subsection{Complexes of type $\FP_n$}

Let $n\geq 0$ be an integer.

\begin{definition}
A complex $\bm{X}$ in $\Ch(R)$ is said to be of \emph{type $\FP_n$}  if there is an exact sequence
\begin{align}\label{eqn:partial_presentation}
\bm{P_{n}} \to \bm{P_{n-1}} \to \cdots \to \bm{P_1} \to \bm{P_0} \to \bm{X} \to \bm{0}
\end{align}
in $\Ch(R)$, where each $\bm{P_i}$ is finitely generated projective.
\end{definition}

The exact sequence \eqref{eqn:partial_presentation} will be referred as a partial presentation of $\bm{X}$ of length $n$ (by finitely generated projective complexes). In some references in the literature, some authors rather define finitely presented modules (and more generally, modules of type $\FP_n$) by considering presentations by finitely generated free modules, such as in \cite{Bourbaki,Glaz}. Actually, the two approaches to this definition are equivalent in the sense that a left $R$-module $M$ has a partial presentation as \eqref{eqn:nprepmodule} if, and only if, there exists an exact sequence
\begin{align}\label{eqn:presentation_mod}
F_n \to F_{n-1} \to \cdots \to F_1 \to F_0 \to M \to 0
\end{align}
where each $F_k$ is a finitely generated free left $R$-module, that is, $M$ has a so called partial presentation of length $n$ by finitely generated free modules (or just a \emph{finite $n$-presentation}, for short). Modules of type $\FP_\infty$ have also a similar description. We will show that this equivalence is also valid for chain complexes, but in order to do that, we first recall from \cite[Definition 1.3.3]{EJ11} the definition of free complexes.

\begin{definition}
A complex $\bm{F}$ is called \emph{free} if there exists a graded set $B \subseteq \bm{F}$ such that for any complex $\bm{X}$ and any morphism $B \to \bm{X}$ of degree $0$, there exists a unique morphism $\bm{F} \to \bm{X}$ of complexes that agrees with $B \to \bm{X}$.
\end{definition}

\begin{proposition}\label{prop:free_complexes}
The following conditions hold true:
\begin{itemize}
\item[{\rm (a)}] $D^n(F)$ is a free complex in $\Ch(R)$ for any free module $F$ in $\Mod(R)$ and any $n \in \mathbb{Z}$.

\item[{\rm (b)}] \underline{Eilenberg Swindle}: For any finitely generated projective complex $\bm{P}$ in $\Ch(R)$, there exists a finitely generated free complex $\bm{F}$ in $\Ch(R)$ such that $\bm{F} \oplus \bm{P} \simeq \bm{F}$.
\end{itemize}
\end{proposition}

\begin{proof} \
\begin{itemize}
\item[(a)] Let $B \subseteq F$ be a finite set generating $F$. Consider the graded set $S^n(B) \subseteq D^n(F)$, and a morphism $S^n(B) \to \bm{X}$ of degree $0$, that is, we have a function $f \colon B \to X_n$ and zero morphisms $0 \to X_k$ for $k \neq n$. Since $F$ is free with generating set $B$, there is a unique homomorphism $\overline{f} \colon F \to X_n$ that agrees with $f$. From $\overline{f}$, we can define a morphism of complexes $\bm{f} \colon D^n(F) \to \bm{X}$ with $f_n := \overline{f}$, $f_{n-1} := \delta^{\bm{X}}_n \circ \overline{f}$, and $f_k := 0$ for every $k \neq n, n-1$. It is easy to verify that $\bm{f}$ is the only morphism of complexes that agrees with $S^n(B) \to \bm{X}$. Hence, $D^n(F)$ is a free complex.

\item[(b)] First note that $\bm{P} \simeq \bigoplus^m_{i = 1} D^{n_i}(Q_{i})$ with each $Q_{i}$ finitely generated and projective. From module theory, we can choose for each $i$ a finitely generated free module $F_i$ and an epimorphism $F_i \to Q_i$. This family of epimorphisms gives rise to an epimorphism $\bigoplus^m_{i = 1} D^{n_i}(F_i) \to \bm{P}$, where $\bigoplus^m_{i = 1} D^{n_i}(F_i)$ is finitely generated and free by part (a) and \cite[Section 1.3]{EJ11}. This epimorphism splits and so $\bigoplus^m_{i = 1} D^{n_i}(F_i) \simeq \bm{P} \oplus \bm{X}$ for some complex $\bm{X}$. The rest follows by \cite[Proposition 9.2.2]{PerezBook}.
\end{itemize}
\end{proof}

We denote by $\mathscr{FP}_n(R)$ the class of all complexes of type $\FP_n$ in $\Ch(R)$. Obviously, $\mathscr{FP}_0(R)$ consists of all finitely generated complexes in $\Ch(R)$, and $\mathscr{FP}_1(R)$ consists of all finitely presented complexes in $\Ch(R)$. For $n > 1$, and following the spirit of \cite[Proposition 2.2]{BG16}, we have the following characterization for complexes in $\mathscr{FP}_n(R)$.

\begin{proposition}\label{type fp-n}
The following statements are equivalent for a complex $\bm{X}$ in $\Ch(R)$.
\begin{itemize}
\item[{\rm (1)}] $\bm{X}$ is of type $\FP_n$.

\item[{\rm (2)}] There exists an exact sequence
\begin{align}\label{eqn:presentation}
\bm{F_n} \to \bm{F_{n-1}} \to \cdots \to \bm{F_1} \to \bm{F_0} \to \bm{X} \to \bm{0}
\end{align}
in $\Ch(R)$, where each $\bm{F_i}$ is finitely generated free.

\item[{\rm (3)}] $\bm{X}$ is bounded and each term $X_m$ is of type $\FP_n$ in $\Mod(R)$.

\item[{\rm (4)}] There exists an exact sequence $\bm{0} \to \bm{K_n} \to \bm{P_{n-1}} \to \cdots \to \bm{P_1} \to \bm{P_0} \to \bm{X} \to \bm{0}$ in $\Ch(R)$, where each $\bm{P_i}$ is finitely generated projective and $\bm{K_n}$ is finitely generated.

\item[{\rm (5)}] For each exact sequence $\bm{0} \to \bm{E_n} \to \bm{Q_{n-1}} \to \cdots \to \bm{Q_1} \to \bm{Q_0} \to \bm{X} \to \bm{0}$ in $\Ch(R)$ with each $\bm{Q_i}$ finitely generated projective, one has that $\bm{E_n}$ is finitely generated.
\end{itemize}
\end{proposition}

\begin{proof}
The equivalence (1) $\Leftrightarrow$ (2) is a consequence of Proposition \ref{prop:free_complexes} (b), while the equivalences (1) $\Leftrightarrow$ (4) $\Leftrightarrow$ (5) are trivial. We only focus on proving (1) $\Leftrightarrow$ (3).
\begin{itemize}
\item[$\bullet$] (1) $\Rightarrow$ (3): Let $\bm{X}$ be a complex  of type $\FP_n$ in $\Ch(R)$. By definition, $\bm{X}$ must be finitely generated, and so is bounded by Lemma \ref{fg-fp}. Moreover, for each term $X_m$ of $\bm{X}$, there is an exact sequence $(P_n)_m \to (P_{n-1})_m \to \cdots \to (P_1)_m \to (P_0)_m \to X_m \to 0$ in $\Mod(R)$, where each $(P_i)_m$ is finitely generated projective. Thus $X_m$ is  of type $\FP_n$ in $\Mod(R)$.

\item[$\bullet$] (3) $\Rightarrow$ (1): Since the complex $\bm{X}$ is bounded, we may assume that it is of the form
\[
\bm{X} = \cdots \to 0 \to X_s \to X_{s-1} \to \cdots \to X_{t+1} \to X_t \to 0 \to \cdots
\]
with each $X_i \in \mathcal{FP}_n(R)$. In particular, each $X_i$ is finitely generated, and then we can take an exact sequence $P^0_i\rightarrow X_i\rightarrow 0$ with $P^0_i$ finitely generated projective. Then we get a finitely generated projective complex $\bm{P^0}$ defined as
\[
\bm{P^0} := \cdots \to 0 \to P^0_s \to P_s^0 \oplus P^0_{s-1} \to \cdots \to P^0_{t+1} \oplus P^0_t \to P^0_t \to 0 \to \cdots
\]
and an exact sequence $\bm{P^0} \rightarrow \bm{X} \rightarrow \bm{0}$ in $\Ch(R)$. Set $\bm{K^1} = \mbox{Ker}(\bm{P^0} \rightarrow \bm{X})$. Then clearly $\bm{K^1}$ is bounded, and since each $X_i$ is of type $\FP_n$, each $K^1_i$ is of type $\FP_{n-1}$. By an argument similar to the previous process, we can get an exact sequence $\bm{P^1} \rightarrow \bm{K^1} \rightarrow \bm{0}$ in $\Ch(R)$ with $\bm{P^1}$ a finitely generated projective complex. Repeating this, we obtain an exact sequence in $\Ch(R)$ as \eqref{eqn:partial_presentation} where each $\bm{P^i}$ is finitely generated projective, that is, $\bm{X} \in \mathscr{FP}_n(R)$.
\end{itemize}
\end{proof}

We can use the previous proposition to note a couple of facts about the classes $\mathscr{FP}_n(R)$ and the interplay between them. First, we note that $\mathscr{FP}_\infty(R) = \bigcap_{n\geq 0} \mathscr{FP}_n(R)$. On the one hand, the inclusion $``\subseteq"$ follows from the descending chain of inclusions:
\begin{equation}\label{1}
\mathscr{FP}_0(R) \supseteq \mathscr{FP}_1(R) \supseteq \cdots \supseteq \mathscr{FP}_n(R) \supseteq \mathscr{FP}_{n+1}(R) \supseteq \cdots\supseteq \mathscr{FP}_\infty(R),
\end{equation}
while on the other hand, the remaining inclusion follows from the fact that any truncated finitely generated projective resolution of length $n$ of a complex in $\bigcap_{n\geq 0}\mathscr{FP}_n(R)$ can be extended to a truncated finitely generated projective resolution of length $n+1$. Note that the inclusions in (\ref{1}) may be strict, as shown by the following examples.

\begin{example}\label{ex1}
Let $R=k[x_1,x_2,\cdots]/{\mathfrak{m}^2}$ with $k$ a field and the ideal $\mathfrak{m}=\langle x_1, x_2,\cdots\rangle$.
\begin{itemize}
\item[(a)] By \cite[Example 1.3]{BP16}, for each $i\geq 1$, $\langle x_i \rangle \in \mathcal{FP}_0(R) \backslash \mathcal{FP}_1(R)$. Then by Lemma \ref{fg-fp}, one has $D^n(\langle x_i\rangle) \in \mathscr{FP}_0(R) \backslash \mathscr{FP}_1(R)$ and $S^n(\langle x_i\rangle) \in \mathscr{FP}(R)_0 \backslash \mathscr{FP}_1(R)$.

\item[(b)] By \cite[Example 1.3]{BP16}, for each $i\geq 1$, $R/\langle x_i\rangle\in \mathcal{FP}_1(R) \backslash \mathcal{FP}_2(R)$. Then by Proposition \ref{type fp-n}, one has $D^n(R/\langle x_i\rangle) \in \mathscr{FP}_1(R) \backslash \mathscr{FP}_2(R)$ and $S^n(R/\langle x_i\rangle) \in \mathscr{FP}_1(R) \backslash \mathscr{FP}_2(R)$.

\item[(c)] By \cite[Example 1.3]{BP16} or \cite[Proposition 2.5]{BGH14}, one has that $\mathcal{FP}_n(R) = \mathcal{FP}_\infty(R)$ for every $n \geq 2$, which consists of the class of finitely generated free $R$-modules. Then by Proposition \ref{type fp-n}, we obtain $\mathscr{FP}_2(R) = \mathscr{FP}_3(R) = \cdots = \mathscr{FP}_\infty(R)$. Therefore, in this case, the chain (\ref{1}) is just as follows:
\begin{equation}\label{aa}
\mathscr{FP}_0(R) \supsetneq \mathscr{FP}_1(R) \supsetneq \mathscr{FP}_2(R) = \mathscr{FP}_3(R) = \cdots = \mathscr{FP}_\infty(R),
\end{equation}
which is stable after $n = 2$.
\end{itemize}
\end{example}

The inclusions in (\ref{1}) may all be strict.

\begin{example}\label{ex2}
Let $R = k[x_1,x_2,\cdots,y_1,y_2,\cdots]/{\langle x_{i+1}x_i,x_1y_1,y_1y_j\rangle}_{i,j\geq 1}$ with $k$ a field. By \cite[Example 1.4]{BP16}, $\langle y_1\rangle\in \mathcal{FP}_0(R) \backslash \mathcal{FP}_1(R)$, and hence $D^n(\langle y_1\rangle)\in \mathscr{FP}_0(R) \backslash \mathscr{FP}_1(R)$ and $S^n(\langle y_1\rangle)\in \mathscr{FP}_0(R) \backslash \mathscr{FP}_1(R)$. Moreover, for each $i\geq 1$, $\langle x_i\rangle\in \mathcal{FP}_i(R) \backslash \mathcal{FP}_{i+1}(R)$, and hence $D^n(\langle x_i\rangle)\in \mathscr{FP}_i(R) \backslash \mathscr{FP}_{i+1}(R)$ and $S^n(\langle x_i\rangle)\in \mathscr{FP}_i(R) \backslash \mathscr{FP}_{i+1}(R)$. Therefore, in this case, the chain (\ref{1}) is just as $\mathscr{FP}_0(R) \supsetneq \mathscr{FP}_1(R) \supsetneq \cdots \supsetneq \mathscr{FP}_n(R) \supsetneq \mathscr{FP}_{n+1}(R) \supsetneq \cdots$ which is not stable at any level.
\end{example}

The second fact to note about $\mathscr{FP}_n(R)$ is a series of closure properties. Recall that a class $\mathscr{X}$ of complexes in $\Ch(R)$ is:
\begin{itemize}
\item[(a)] \emph{closed under direct summands} if for every $\bm{X} \in \mathscr{X}$ and every complex $\bm{X}'$ that is a direct summand of $\bm{X}$, one has $\bm{X}' \in \mathscr{X}$;

\item[(b)] \emph{closed under extensions} if for every short exact sequence $\bm{\eta} \colon \bm{0} \to \bm{A} \to \bm{B} \to \bm{C} \to \bm{0}$ with $\bm{A}, \bm{C} \in \mathscr{X}$, one has $\bm{B} \in \mathscr{X}$;

\item[(c)] \emph{closed under epi-kernels} if for every short exact sequence as $\bm{\eta}$ with $\bm{B}, \bm{C} \in \mathscr{X}$, one has $\bm{A} \in \mathscr{X}$; and \emph{closed under mono-cokernels} is the dual property is satisfied.
\end{itemize}
These definitions are analogous in the category $\Mod(R)$.

The following result is a consequence of Proposition \ref{type fp-n} and the closure properties of $\mathcal{FP}_n(R)$ proved by Bravo and the second author in \cite[Proposition 1.7]{BP16}.

\begin{corollary}\label{coro:properties_FPn}
For every $n \geq 0$, the class $\mathscr{FP}_n(R)$ of complexes of type $\FP_n$ in $\Ch(R)$ is closed under extensions, direct summands and mono-cokernels.
\end{corollary}

Recall that a class $\mathscr{X}$ of complexes in $\Ch(R)$ is called \emph{thick} if (a), (b) and (c) above are satisfied. Thick classes of modules are defined in a similar way. For example, in the category $\Ch(R)$, the class of exact complexes and the class of bounded complexes are both thick. But in general, we cannot assert that $\mathscr{FP}_n(R)$ is thick (or equivalently in this case, closed under epi-kernels). This missing closure property for $\mathscr{FP}_n(R)$ is related to the stable condition of the chain (\ref{1}), which in turn can characterize  classes of special rings. For example, Bravo and Gillespie proved in \cite[Corollary 2.3]{BG16} that:
\begin{itemize}
\item[(a)] A ring $R$ is left Noetherian if, and only if, $\mathscr{FP}_0(R) = \mathscr{FP}_\infty(R)$.

\item[(b)] A ring $R$ is left coherent if, and only if, $\mathscr{FP}_1(R) = \mathscr{FP}_\infty(R)$.
\end{itemize}
For an analogous equivalence involving $\mathscr{FP}_n(R)$, one needs a more general class of rings, introduced by D. L. Costa in \cite{Co94}.

\begin{definition}
A ring $R$ is called \emph{left $n$-coherent} if each module of type $\FP_n$ in $\Mod(R)$ is of type $\FP_{n+1}$, that is, $\mathcal{FP}_n(R) \subseteq \mathcal{FP}_{n+1}(R)$.
\end{definition}

By definition, left $0$-coherent rings are just left noetherian rings, and left $1$-coherent rings are just left coherent rings. The family of $n$-coherent rings can be characterized in terms of thick classes of modules, as in \cite[Theorem 2.4]{BP16}. Namely, a ring $R$ is left $n$-coherent if, and only if, $\mathcal{FP}_n(R)$ is closed under epi-kernels. The analogous for $\mathscr{FP}_n(R)$ is specified below, which follows by \cite[Theorem 2.4]{BP16}, by the characterization of $\mathscr{FP}_n(R)$ proved in Proposition \ref{type fp-n}, and by Corollary \ref{coro:properties_FPn}.

\begin{proposition}
The following conditions are equivalent:
\begin{itemize}
\item[{\rm (1)}] $R$ is left $n$-coherent.

\item[{\rm (2)}] The class $\mathscr{FP}_n(R)$ is thick.

\item[{\rm (3)}] $\mathscr{FP}_n(R) = \mathscr{FP}_{n+1}(R)$.

\item[{\rm (4)}] $\mathscr{FP}_n(R) = \mathscr{FP}_\infty(R)$.

\item[{\rm (5)}] The chain \eqref{1} stabilizes at $n$, that is,
\[
\mathscr{FP}_0(R) \supseteq \mathscr{FP}_1(R) \supseteq \cdots \supseteq \mathscr{FP}_n(R) = \mathscr{FP}_{n+1}(R) = \cdots = \mathscr{FP}_\infty(R).
\]
\end{itemize}
\end{proposition}

It follows that the ring of Example \ref{ex1} is $2$-coherent, and the ring of Example \ref{ex2} is not $n$-coherent for any $n\geq 0$.

So far we have given several descriptions of the classes $\mathscr{FP}_n(R)$, but they can also be interpreted in terms of a certain resolution dimension, that is going to be presented and studied next.


\subsection{Presentation dimension}

In \cite[Section 1 of Chapter 2]{Glaz}, S. Glaz defined the following value for every module $M$ in $\Mod(R)$:
\[
\lambda_R(M) := \left\{
\begin{array}{ll}
{\rm sup} \{ n \geq 0 \mbox{ : $\exists$ a finite $n$-presentation (as \eqref{eqn:presentation_mod}) of $M$} \}, & \mbox{if $M$ is finitely generated}; \\
-1, & \mbox{otherwise}.
\end{array}
\right.
\]
We will refer to the value $\lambda_R(M)$ as the \emph{presentation dimension of $M$}. Motivated by this, and by Proposition \ref{type fp-n} (b), we define the \emph{presentation dimension of a complex $\bm{X}$ in $\Ch(R)$} as:
\[
\lambda(\bm{X}) := \left\{
\begin{array}{ll}
{\rm sup} \{ n \geq 0 \mbox{ : $\exists$ a finite $n$-presentation (as \eqref{eqn:presentation}) of $\bm{X}$} \}, & \mbox{if $\bm{X}$ is finitely generated}; \\
-1, & \mbox{otherwise}.
\end{array}
\right.
\]
Note that, by Proposition \ref{type fp-n}, for every finitely generated complex $\bm{X}$ in $\Ch(R)$, one has that $\lambda(\bm{X}) = n$ if, and only if, there exists a finite $n$-presentation of $\bm{X}$ with non-finitely generated $(n+1)$-st syzygy.

There is a relation between the presentation dimension of complexes and that of modules, specified in the following result.

\begin{theorem}\label{theo:presentation}
For every bounded complex $\bm{X}$ in $\Ch(R)$, the following equality holds:
\begin{align}\label{eqn:formula}
\lambda(\bm{X}) = {\rm inf} \{ \lambda_R(X_m) \mbox{ {\rm : }} m \in \mathbb{Z} \}.
\end{align}
\end{theorem}

\begin{proof}
Suppose first that $\bm{X}$ is a bounded complex which is not finitely generated, and so $\lambda(\bm{X}) = -1$. Then, by Lemma \ref{fg-fp} there exists $m_0 \in \mathbb{Z}$ such that $X_{m_0}$ is not finitely generated, and so $\lambda_R(X_{m_0}) = -1$. Hence, the formula \eqref{eqn:formula} holds.

Now we may assume that $\bm{X}$ is finitely generated, and so bounded with finitely generated terms. This implies ${\rm inf} \{ \lambda_R(X_m) \mbox{ {\rm : }} m \in \mathbb{Z} \} \geq 0$. Suppose $\bm{X} \in \mathscr{FP}_\infty(R)$. Then $X_m \in \mathcal{FP}_\infty(R)$ for every $m \in \mathbb{Z}$ by \cite[Proposition 2.2]{BG16}. Hence, in this case, the formula \eqref{eqn:formula} is also true.

Finally, suppose that the presentation dimension of $\bm{X}$ is finite, say $\lambda(\bm{X}) = n$. Then, $X_m \in \mathcal{FP}_n(R)$ for every $m \in \mathbb{Z}$ by Proposition \ref{type fp-n}. It follows $\lambda_R(X_m) \geq n$ for every $m \in \mathbb{Z}$, and so ${\rm inf} \{ \lambda_R(X_m) \mbox{ : } m \in \mathbb{Z} \} \geq \lambda(\bm{X})$. On the other hand, since the presentation dimension of $\bm{X}$ is finite, there exists $m_0 \in \mathbb{Z}$ such that $\lambda_R(X_{m_0}) = k < \infty$. Without loss of generality, we may assume ${\rm inf} \{ \lambda_R(X_m) \mbox{ : } m \in \mathbb{Z} \} = \lambda_R(X_{m_0})$. Then, $X_m \in \mathcal{FP}_k(R)$ for every $m \in \mathbb{Z}$. Since $\bm{X}$ is bounded, we can use the arguments applied in the proof of Proposition \ref{type fp-n} (3) $\Rightarrow$ (1) to showing that $\bm{X} \in \mathscr{FP}_k(R)$. Hence, ${\rm inf} \{ \lambda_R(X_m) \mbox{ : } m \in \mathbb{Z} \} = k \leq \lambda(\bm{X})$.
\end{proof}

\begin{remark}
Note that Theorem \ref{theo:presentation} holds for every finitely generated complex. We have not included in the statement the case where $\bm{X}$ is not finitely generated, since for such complexes the formula \eqref{eqn:formula} may not hold. For instance, the complex $\bm{S} = \bigoplus_{m \in \mathbb{Z}} S^m(R)$ is not finitely generated by Lemma \ref{fg-fp}, since it is unbounded, and so $\lambda(\bm{S}) = -1$. On the other hand, ${\rm inf} \{ \lambda_R(S_n) \mbox{ {\rm :} } n \in \mathbb{Z} \} = {\rm inf} \{ \lambda_R(R) \mbox{ {\rm :} } n \in \mathbb{Z} \} = \infty$. Notice that $\lambda_R(R) = \infty$ since $R$ is free and so of type $\FP_\infty$.
\end{remark}

Using Theorem \ref{theo:presentation}, we can extend \cite[Theorem 2.1.2]{Glaz} to the category of complexes in $\Ch(R)$, as a way to compare the presentation dimension of complexes appearing in short exact sequences.

\begin{proposition}
Let $\bm{\eta} \colon \bm{0} \to \bm{A} \to \bm{B} \to \bm{C} \to \bm{0}$ be a short exact sequence in $\Ch(R)$. The following relations hold:
\begin{itemize}
\item[{\rm (a)}] $\lambda(\bm{A}) \geq {\rm min} \{ \lambda(\bm{B}), \lambda(\bm{C}) - 1 \}$.

\item[{\rm (b)}] $\lambda(\bm{B}) \geq {\rm min} \{ \lambda(\bm{A}), \lambda(\bm{C}) \}$.

\item[{\rm (c)}] $\lambda(\bm{C}) \geq {\rm min} \{ \lambda(\bm{B}), \lambda(\bm{A}) + 1 \}$.

\item[{\rm (d)}] If $\bm{B} = \bm{A} \oplus \bm{C}$, then $\lambda(\bm{B}) = {\rm min} \{ \lambda(\bm{A}), \lambda(\bm{C}) \}$.
\end{itemize}
\end{proposition}

\begin{proof}
The first lines of this proof will be devoted to show that we may assume that the sequence $\bm{\eta}$ is formed by finitely generated complexes. We study the finiteness possibilities for each term in several cases:
\begin{itemize}
\item[$\bullet$] \underline{$\bm{A}$ is not finitely generated}: Suppose that $\bm{B}$ is finitely generated. Then, $\bm{C}$ must be finitely generated. We will see that, in this case, $\lambda(\bm{C}) = 0$. Suppose that $\lambda(\bm{C}) = k > 0$. Then $\bm{\eta}$ is a short exact sequence of bounded complexes, since the class of such complexes is thick. Since $\bm{A}$ is bounded and not finitely generated, there exists $m_0 \in \mathbb{Z}$ such that $\lambda_R(A_{m_0}) = -1$. On the other hand, by \cite[Theorem 2.1.2 (3)]{Glaz} we have $\lambda_R(A_{m_0}) \geq {\rm min}\{ \lambda_R(B_{m_0}), \lambda_R(C_{m_0}) - 1 \}$. Using the hypothesis that $\bm{B}$ is finitely generated, and by the previous inequality, we have that ${\rm min}\{ \lambda_R(B_{m_0}), \lambda_R(C_{m_0}) - 1 \} = \lambda_R(C_{m_0}) - 1$. Then, $\lambda_R(C_{m_0}) \leq 0$, and thus we get a contradiction with the assumption $\lambda(\bm{C}) > 0$. Hence, we have $\lambda(\bm{C}) = 0$, and so (a) holds. The inequalities (b) and (c) are clearly satisfied in this case. Note that (d) cannot be covered under the assumption that $\bm{A}$ is not finitely generated and $\bm{B}$ is finitely generated.

In the case $\bm{B}$ is not finitely generated, items from (a) to (d) clearly hold true.

\item[$\bullet$] \underline{$\bm{B}$ is not finitely generated}: It follows either $\bm{A}$ or $\bm{C}$ must not be finitely generated. Otherwise, we would contradict the fact that finitely generated complexes are closed under extensions. It follows that the inequalities (a), (b) and (c) hold. In fact, (b) is actually an equality, and so (d) is also true.

\item[$\bullet$] \underline{$\bm{C}$ is not finitely generated}: Then, $\bm{B}$ must not be finitely generated, and hence items from (a) to (d) are clearly satisfied.
\end{itemize}
For the rest of the proof, we may assume that $\bm{\eta}$ is a short exact sequence of finitely generated (and so bounded) complexes. We only prove (2) and (4), and the remaining inequalities will follow similarly. Without loss of generality, suppose ${\rm min}\{ \lambda(\bm{A}), \lambda(\bm{C}) \} = \lambda(\bm{A})$. If $\lambda(\bm{A}) = \infty$, then $\lambda(\bm{C}) = \infty$, and so $\lambda(\bm{B}) = \infty$ since the class of complexes of type $\FP_\infty$ is closed under extensions. In this case, (2) follows immediately.

Now suppose $\lambda(\bm{A}) < \infty$. Since $\bm{A}$ is bounded, there exists $m_0 \in \mathbb{Z}$ such that $\lambda(\bm{A}) = \lambda_R(A_{m_0})$. From the assumption that ${\rm min}\{ \lambda(\bm{A}), \lambda(\bm{C}) \} = \lambda(\bm{A})$, note that $\lambda_R(A_{m_0}) \leq \lambda_R(C_{m_0})$. Let $m \in \mathbb{Z}$. By \cite[Theorem 2.1.2 (1)]{Glaz}, we have $\lambda_R(B_m) \geq {\rm min}\{ \lambda_R(A_m), \lambda_R(C_m) \}$. On the one hand, by Theorem \ref{theo:presentation}, if ${\rm min}\{ \lambda_R(A_m), \lambda(C_m) \} = \lambda(A_m)$, then
\[
\lambda_R(B_m) \geq \lambda_R(A_m) \geq {\rm inf}\{ \lambda_R(A_m) \mbox{ : } m \in \mathbb{Z} \} \geq \lambda_R(A_{m_0}) = {\rm min}\{ \lambda(\bm{A}), \lambda(\bm{C}) \}.
\]
On the other hand, using Theorem \ref{theo:presentation} again, if ${\rm min}\{ \lambda_R(A_m), \lambda_R(C_m) \} = \lambda_R(C_m)$, then
\[
\lambda_R(B_m) \geq \lambda_R(C_m) \geq {\rm inf}\{ \lambda_R(C_m) \mbox{ : } m \in \mathbb{Z} \} =  \lambda(\bm{C}) \geq \lambda(\bm{A}) = {\rm min}\{ \lambda(\bm{A}), \lambda(\bm{C}) \}.
\]
It follows that $\lambda(\bm{B}) \geq {\rm min}\{ \lambda(\bm{A}), \lambda(\bm{C}) \}$.

For the case $\bm{B} = \bm{A} \oplus \bm{C}$. On the one hand, we already know that $\lambda(\bm{B}) \geq {\rm min}\{ \lambda(\bm{A}), \lambda(\bm{C}) \}$. On the other hand, by Theorem \ref{theo:presentation} and \cite[Theorem 2.1.2 (4)]{Glaz}, we have
\begin{align*}
\lambda(\bm{B}) & \leq \lambda_R(B_m) = {\rm min}\{ \lambda_R(A_m), \lambda_R(C_m) \} \leq \lambda_R(A_m), \\
\lambda(\bm{B}) & \leq \lambda_R(B_m) = {\rm min}\{ \lambda_R(A_m), \lambda_R(C_m) \} \leq \lambda_R(C_m),
\end{align*}
for every $m \in \mathbb{Z}$. Then, $\lambda(\bm{B}) \leq \lambda(\bm{A})$ and $\lambda(\bm{B}) \leq \lambda(\bm{C})$, and hence $\lambda(\bm{B}) \leq {\rm min}\{ \lambda(\bm{A}), \lambda(\bm{C}) \}$ follows.
\end{proof}


\subsection{Injective and flat complexes relative to complexes of type $\FP_n$}

We now give the definitions of $\FP_n$-injective and $\FP_n$-flat complexes as follows.

\begin{definition}
Let $\bm{X}$ be a complex in $\Ch(R)$ and $\bm{Y}$ be a complex in $\Ch(R\op)$. We say that:
\begin{itemize}
\item[(a)] $\bm{X}$ is \emph{$\FP_n$-injective} if $\underline{\Ext}^1(\bm{L},\bm{X}) = \bm{0}$ for every $\bm{L} \in \mathscr{FP}_n(R)$.

\item[(b)] $\bm{Y}$ is \emph{$\FP_n$-flat} if $\overline{\Tor}_1(\bm{Y},\bm{L}) = \bm{0}$ for every $\bm{L} \in \mathscr{FP}_n(R)$.
\end{itemize}
\end{definition}

We denote by $\mathscr{I}_n(R)$ the class of $\FP_n$-injective complexes in $\Ch(R)$, and by $\mathscr{F}_n(R\op)$ the class of $\FP_n$-flat complexes in $\Ch(R\op)$. Note that the $\FP_0$-injective complexes coincide with the  injective complexes, the $\FP_1$-injective complexes coincide with the  $\FP$-injective or absolutely pure complexes, and the $\FP_i$-flat complexes are the flat complexes for $i=0,1$. Moreover, we immediately obtain the following ascending chains:
\begin{align}\label{2}
\mathscr{I}_0(R) & \subseteq \mathscr{I}_1(R) \subseteq \cdots \subseteq \mathscr{I}_n(R) \subseteq \mathscr{I}_{n+1}(R) \subseteq \cdots \subseteq \mathscr{I}_\infty(R),\\
\mathscr{F}_0(R\op) & = \mathscr{F}_1(R\op) \subseteq \cdots \subseteq \mathscr{F}_n(R\op) \subseteq \mathscr{F}_{n+1}(R\op) \subseteq \cdots \subseteq \mathscr{F}_\infty(R\op).
\end{align}

\begin{remark}\label{rem}
By definition, one easily checks that the class of  $\FP_n$-injective complexes is closed under extensions,  products and direct summands; and the category of $\FP_n$-flat complexes is closed under extensions, direct limits (and so under coproducts) and direct summands. We can add a couple of more properties for these two classes, after showing the following characterization.
\end{remark}

\begin{theorem}\label{tfae1}
The following are equivalent for every complex $\bm{X}$ in $\Ch(R)$ and every $n \geq 0$:
\begin{itemize}
\item[{\rm (1)}] $\bm{X} \in \mathscr{I}_n(R)$.

\item[{\rm (2)}] $\Ext^1_{\Ch}(\bm{L},\bm{X}) = 0$ for every $\bm{L} \in \mathscr{FP}_n(R)$.

\item[{\rm (3)}] $\Ext^1_{\Ch}(S^m(L),\bm{X}) = 0$ for every $m \in \mathbb{Z}$ and every $L \in \mathcal{FP}_n(R)$.

\item[{\rm (4)}] $\bm{X}$ is exact and $Z_m(\bm{X}) \in \mathcal{I}_n(R)$ for every $m \in \mathbb{Z}$.

\item[{\rm (5)}] $X_m \in \mathcal{I}_n(R)$ for every $m \in \mathbb{Z}$, and $\hom(\bm{L},\bm{X})$ is exact for every $\bm{L} \in \mathscr{FP}_n(R)$.

\item[{\rm (6)}] For any exact sequence $\bm{\eta} \colon \bm{0} \to \bm{Q} \to \bm{W} \to \bm{L} \to \bm{0}$ in $\Ch(R)$ with $\bm{L} \in \mathscr{FP}_n(R)$, the induced sequence $\underline{\Hom}(\bm{\eta},\bm{X})$ is exact.
\end{itemize}
\end{theorem}

\begin{proof}
The equivalence (1) $\Leftrightarrow$ (2) is clear by \eqref{eqn:underext}. On the other hand, (2) $\Leftrightarrow$ (3) $\Leftrightarrow$ (4) follows as in \cite[Lemma 2.5 and Proposition 2.6]{BG16}.
\begin{itemize}
\item[$\bullet$] (4) $\Rightarrow$ (5): Suppose $\bm{X}$ is an exact complex with $\FP_n$-injective cycles. For each $m \in \mathbb{Z}$, we have a short exact sequence $0 \to Z_m(\bm{X}) \to X_m \to Z_{m-1}(\bm{X}) \to 0$ in $\Mod(R)$. Since $\mathcal{I}_n(R)$ is closed under extensions by \cite[Proposition 3.10]{BP16}, we have that $X_m \in \mathcal{I}_n(R)$.

Now let $\bm{L} \in \mathscr{FP}_n(R)$. Using \cite[Lemma 2.1]{GillespieFlat}, we have that:
\begin{align*}
H_k(\hom(\bm{L},\bm{X})) & \cong \Ext^1_{\rm dw}(\bm{L},\bm{X}[-k-1]) = \Ext^1_{\Ch}(\bm{L},\bm{X}[-k-1]),
\end{align*}
where the right-hand side equality is valid since $\Ext^1_R(L_m,X_{m+k+1}) = 0$, being $L_m$ of type $FP_n$ by Proposition \ref{type fp-n} and $X_{m+k+1} \in \mathcal{I}_n(R)$. On the other hand, using the equivalence (2) $\Leftrightarrow$ (4) we have that $H_k(\hom(\bm{L},\bm{X})) \cong \Ext^1_{\Ch}(\bm{L}[k+1],\bm{X}) = 0$ for every $k \in \mathbb{Z}$. Hence, the complex $\hom(\bm{L},\bm{X})$ is exact.

\item[$\bullet$] (5) $\Rightarrow$ (6): Let us show $\underline{\Ext}^1(\bm{L},\bm{X}) = \bm{0}$ for every $\bm{L} \in \mathscr{FP}_n(R)$. Suppose we are given a short exact sequence $\bm{0} \ra \bm{X} \to \bm{H} \ra \bm{L} \ra \bm{0}$ in $\Ch(R)$. Since $X_m \in \mathcal{I}_n(R)$ by (5) for every $m \in \mathbb{Z}$, this exact sequence splits at the module level, and so it is isomorphic to $\bm{0} \ra \bm{X} \ra M(\bm{f}) \ra \bm{L} \ra \bm{0}$, where  $\bm{f} \colon \bm{L}[1] \ra \bm{X}$ is a morphism of complexes and $M(\bm{f})$ denotes its mapping cone. Since $\hom(\bm{L}[1],\bm{X})$ is exact by (5), $\bm{f}$ is homotopic to 0. It follows that $\bm{0} \ra \bm{X} \ra M(\bm{f}) \ra \bm{L} \ra \bm{0}$ is a split exact sequence in $\Ch(R)$ by \cite[Lemma 2.3.2]{GR99}. Therefore $\Ext^1_{\Ch}(\bm{L},\bm{X}) = 0$, and so $\underline{\Ext}^1(\bm{L},\bm{X}) = \bm{0}$ by \eqref{eqn:underext}. Thus, (6) follows.

\item[$\bullet$] (6) $\Rightarrow$ (1): Let $\bm{L} \in \mathscr{FP}_n(R)$. There is an exact sequence $\bm{\eta} \colon \bm{0} \ra \bm{Q} \ra \bm{P} \ra \bm{L} \ra \bm{0}$ in $\Ch(R)$ with $\bm{P}$ finitely generated projective. Applying $\underline{\Hom}(-,\bm{X})$ to $\bm{\eta}$, we get the exact sequence $\underline{\Hom}(\bm{P},\bm{X}) \ra \underline{\Hom}(\bm{Q},\bm{X}) \ra \underline{\Ext}^1(\bm{L},\bm{X}) \ra \bm{0}$ where the morphism $\underline{\Hom}(\bm{P},\bm{X}) \ra \underline{\Hom}(\bm{Q},\bm{X})$ is epic by (6). It follows that $\underline{\Ext}^1(\bm{L},\bm{X}) = \bm{0}$, and hence $\bm{X}$ is $\FP_n$-injective.
\end{itemize}
\end{proof}

Recall that a short exact sequence $\bm{\eta} \colon \bm{0} \to \bm{A} \to \bm{B} \to \bm{C} \to \bm{0}$ in $\Ch(R)$ is \emph{pure} if $\bm{Y} \overline{\otimes} \bm{\eta}$ for every complex $\bm{Y}$ in $\Ch(R\op)$. A class $\mathscr{X}$ of complexes in $\Ch(R)$ is \emph{closed under pure sub-complexes} (resp., \emph{closed under pure quotients}) if for every pure exact sequence as $\bm{\eta}$, one has that $\bm{B} \in \mathscr{X}$ implies $\bm{A} \in \mathscr{X}$ (resp., $\bm{C} \in \mathscr{X}$). Purity for modules and the corresponding closure properties are analogous, where one considers the usual tensor product $- \otimes_R -$ on $\Mod(R\op) \times \Mod(R)$ instead.

\begin{proposition}\label{coproduct-closed}
The sub-category $\mathscr{I}_n(R)$ is closed under coproducts and pure sub-complexes for any $n \geq 1$, and under direct limits and and pure quotients for any $n \geq 2$.
\end{proposition}

\begin{proof}
Let $\{ \bm{X^i} \}_{i \in I}$ be a directed family of $\FP_n$-injective complexes in $\Ch(R)$, and let $\bm{X} := \varinjlim_{i \in I} \bm{X^i}$ denote its direct limit (this covers the case where $\bm{X}$ is a coproduct of $\FP_n$-injective complexes). By Theorem \ref{tfae1}, each $\bm{X^i}$ is exact with $\FP_n$-injective cycles. Since $\Mod(R)$ is a Grothendieck category, we have that $\bm{X}$ is exact. On the other hand, direct limits preserve kernels and so $Z_m(\bm{X}) \cong \varinjlim_{i \in I} Z_m(\bm{X^i})$. Since the class $\mathcal{I}_n(R)$ is closed under direct limits (and so under coproducts) if $n > 1$, we have that $Z_m(\bm{X}) \in \mathcal{I}_n(R)$ if $n > 1$. It remains to cover the case where $n = 1$, in which we will only consider closure under coproducts\footnote{Recall that $\FP_1$-injective modules in $\Mod(R)$ are closed under direct limits if, and only if, the ground ring $R$ is left coherent. See \cite[Theorem 3.1]{DingChen}, for instance.}. In this case, it is know that $\FP$-injective modules are closed under coproducts \cite[Exercise 19 (ii), page 3.11]{StBook}. Hence, $\bm{X}$ is a $\FP$-injective complex by Theorem \ref{tfae1}.

Now suppose we are given a pure exact sequence $\bm{\eta} \colon \bm{0} \to \bm{A} \to \bm{B} \to \bm{C} \to \bm{0}$ with $\bm{B} \in \mathscr{I}_n(R)$ and $n \geq 1$. Let $\bm{L} \in \mathscr{FP}_n(R)$, and so $\bm{L}$ finitely presented. By \cite[Lemma 5.1.1 and Theorem 5.1.3]{GR99}, we have that $\underline{\Hom}(\bm{L},\bm{\eta})$ is exact. On the other hand, we have a long exact sequence $\bm{0} \to \underline{\Hom}(\bm{L},\bm{A}) \to \underline{\Hom}(\bm{L},\bm{B}) \to \underline{\Hom}(\bm{L},\bm{C}) \to \underline{\Ext}^1(\bm{L},\bm{A}) \to \underline{\Ext}^1(\bm{L},\bm{B})$ where $\underline{\Ext}^1(\bm{L},\bm{B}) = \bm{0}$ since $\bm{B} \in \mathscr{I}_n(R)$, and $\underline{\Hom}(\bm{L},\bm{B}) \to \underline{\Hom}(\bm{L},\bm{C})$ is an epimorphism. It follows that $\underline{\Ext}^1(\bm{L},\bm{A}) = \bm{0}$, that is, $\bm{A} \in \mathscr{I}_n(R)$. In the case $n > 1$, we use the characterization proved in Theorem \ref{tfae1} to show that $\bm{C} \in \mathscr{I}_n(R)$. First, we note that $\bm{C}$ is exact. Now, by \cite[Lemma 3.3.8]{PerezBook} there is an exact sequence $\zeta_m \colon 0 \to Z_m(\bm{A}) \to Z_m(\bm{B}) \to Z_m(\bm{C}) \to 0$ in $\Mod(R)$ for each $m \in \mathbb{Z}$, where the connecting morphisms are induced by the universal property of kernels. We show that $\zeta_m$ is pure exact, that is, $\Hom_R(L,\zeta_m)$ is exact for every  finitely presented module $L$ in $\Mod(R)$. Since for every $m \in \mathbb{Z}$ the complex $S^m(L)$ is finitely presented by Lemma \ref{fg-fp}, we have an exact sequence $\bm{0} \to \underline{\Hom}(S^m(L),\bm{A}) \to \underline{\Hom}(S^m(L),\bm{B}) \to \underline{\Hom}(S^m(L),\bm{C}) \to \bm{0}$. Then, by \cite[Proposition 4.4.7]{PerezBook} we have a short exact sequence of abelian groups of the form $0 \to \Hom_R(S^m(L),\bm{A}) \to \Hom_R(S^m(L),\bm{B}) \to \Hom_R(S^m(L),\bm{C}) \to 0$, which is isomorphic to $\Hom_R(L,\zeta_m)$ by \cite[Lemma 3.1 (2)]{GillespieFlat}. It follows that $\Hom_R(L,\zeta_m)$ is exact, that is, $\zeta_m$ is pure exact. Now, since $Z_m(\bm{A}) \in \mathcal{I}_n(R)$, we have by \cite[Part 4. of Proposition 3.10]{BP16} that $Z_m(\bm{C}) \in \mathcal{I}_n(R)$ for every $m \in \mathbb{Z}$. Therefore, $\bm{C}$ is $\FP_n$-injective.
\end{proof}

In \cite[Propositions 3.5 and 3.6]{BP16}, Bravo and the second author studied the relation between $\FP_n$-injective and $\FP_n$-flat modules via the Pontrjagin duality functor
\[
(-)^+ \colon \Mod(R) \longrightarrow \Mod(R\op).
\]
Specifically, for every $n > 1$ and every module $N$ in $\Mod(R\op)$ and $M \in \Mod(R)$, one has that:
\begin{itemize}
\item[(a)] $N \in \mathcal{F}_n(R\op)$ if, and only if, $N^+ \in \mathcal{I}_n(R)$.

\item[(b)] $M \in \mathcal{I}_n(R)$ if, and only if, $M^+ \in \mathcal{F}_n(R\op)$.
\end{itemize}
In the category of complexes, one can obtain a similar duality between $\FP_n$-injective and $\FP_n$-flat complexes, as specified in the following result.

\begin{proposition}\label{fi-if} The following equivalences hold for any complex $\bm{X}$ in $\Ch(R)$ and any complex $\bm{Y}$ in $\Ch(R\op)$.
\begin{itemize}
\item[{\rm (a)}] For every $n \geq 0$, $\bm{Y} \in \mathscr{F}_n(R\op)$ if, and only if, $\bm{Y}^+ \in \mathscr{I}_n(R)$.

\item[{\rm (b)}] For every $n \geq 2$, $\bm{X} \in \mathscr{I}_n(R)$ if, and only if, $\bm{X}^+ \in \mathscr{F}_n(R\op)$.
\end{itemize}
\end{proposition}

\begin{proof} \
\begin{itemize}
\item[(a)] By \cite[Lemma 5.4.2]{GR99}, we have that $\underline{\Ext}^1(\bm{X},\bm{Y}^+)\cong \overline{\Tor}_1(\bm{Y}, \bm{X})^+$ for any complex $\bm{Y}$ in $\Ch(R\op)$ and any complex $\bm{X}$ in $\Ch(R)$. So the assertion follows since $D^0(\mathbb{Q / Z})$ is an injective cogenerator in the category of complexes of abelian groups.

\item[(b)] Let $\bm{L} \in \mathscr{FP}_n(R)$. Then there exists an exact sequence $\bm{0} \ra \bm{K} \ra \bm{P} \ra \bm{L} \ra \bm{0}$ in $\Ch(R)$ with $\bm{P}$ finitely generated projective and $\bm{K} \in \mathscr{FP}_{n-1}(R)$.  Note that since $n \geq 2$, $\bm{K}$ must be a finitely presented complex. It follows that we can consider the following commutative diagram with exact rows:
\[
\xymatrix@R=25pt@C=15pt{
\bm{0} \ar[r] & \overline{\Tor}_1(\bm{X}^+,\bm{L}) \ar[r] \ar[d] & \bm{X}^+\overline{\otimes} \bm{K} \ar[r] \ar[d]^{\theta_{\bm{K}}} & \bm{X}^+\overline{\otimes} \bm{P} \ar[d]^{\theta_{\bm{P}}} \\
\bm{0} \ar[r] & \underline{\Ext}^1(\bm{L},\bm{X})^+ \ar[r] & \underline{\Hom}(\bm{K},\bm{X})^+\ar[r] & \underline{\Hom}(\bm{P},\bm{X})^+
}
\]
where $\theta_{\bm{K}}$ and $\theta_{\bm{P}}$ are the isomorphisms described in \cite[Lemma 2.3]{EG97}, and the left-hand side arrow is induced by the universal property of kernels. We have that $\underline{\Ext}^1(\bm{L},\bm{X})^+ \cong \overline{\Tor}_1(\bm{X}^+,\bm{L})$. Thus the result follows.
\end{itemize}
\end{proof}

Proposition \ref{fi-if} turns out to be an important tool that allows us to establish the following characterization of $\FP_n$-flat complexes, similar to that proved in Theorem \ref{tfae1} for $\FP_n$-injective complexes.

\begin{theorem}\label{theo:char_flat}
The following statements are equivalent for any complex $\bm{Y}$ in $\Ch(R\op)$.
\begin{itemize}
\item[{\rm (1)}] $\bm{Y} \in \mathscr{F}_n(R\op)$.

\item[{\rm (2)}] $\overline{\Tor}_1(\bm{Y},S^m(L)) = 0$ for every $m \in \mathbb{Z}$ and every $L \in \mathcal{FP}_n(R)$.

\item[{\rm (3)}] $\bm{Y}$ is exact and $Z_m(\bm{Y}) \in \mathcal{F}_n(R\op)$ for every $m \in \mathbb{Z}$.

\item[{\rm (4)}] The complex $\Hom_R(\bm{Y},\mathbb{Q / Z}) := \cdots \to (Y_{m-2})^+ \to (Y_{m-1})^+ \to (Y_m)^+ \to \cdots$ is $\FP_n$-injective in $\Ch(R)$.
\end{itemize}
\end{theorem}

\begin{proof}
The equivalences (1) $\Leftrightarrow$ (2) and (1) $\Leftrightarrow$ (3) follow as in \cite[Lemma 4.5 and Proposition 4.6]{BG16}. On the other hand, (1) $\Leftrightarrow$ (3) $\Leftrightarrow$ (4) follows using an argument similar to that of \cite[Theorem 2.4]{EG98}.
\end{proof}

Apart from those mentioned in Remark \ref{rem}, the previous result allows us to deduce the following properties of $\FP_n$-flat complexes.

\begin{proposition}\label{prop:flatness_dl}
The sub-category $\mathscr{F}_n(R\op)$ is closed under direct limits for every $n \geq 0$, under direct products and pure quotients for every $n \geq 1$, and under pure sub-complexes for every $n \geq 2$.
\end{proposition}

\begin{proof}
Being $\Mod(R\op)$ a Grothendieck category, we have that exact complexes are closed under direct limits and direct products. On the other hand, $\Tor^R_1(-,M)$ commutes with direct limits for any $M$ in $\Mod(R)$, and $\Tor^R_1(-,L)$ commutes with direct products for any $L \in \mathcal{FP}_n(R)$ with $n \geq 1$ by \cite[Theorem 2]{Brown}. It follows that $\mathcal{F}_n(R\op)$ is closed under direct limits for any $n \geq 0$, and under direct products for any $n \geq 1$. By Theorem \ref{theo:char_flat}, we obtain the same closure properties in the context of $\Ch(R\op)$.

Now suppose we are given a pure exact sequence $\bm{\eta} \colon \bm{0} \to \bm{A} \to \bm{B} \to \bm{C} \to \bm{0}$ in $\Ch(R\op)$ with $\bm{B} \in \mathscr{F}_n(R\op)$ and $n \geq 1$. Let $\bm{L} \in \mathscr{FP}_n(R)$. Then, we have an exact sequence of the form $\overline{\Tor}_1(\bm{B},\bm{L}) \to \overline{\Tor}_1(\bm{C},\bm{L}) \to \bm{A} \overline{\otimes} \bm{L} \to \bm{B} \overline{\otimes} \bm{L} \to \bm{C} \overline{\otimes} \bm{L} \to \bm{0}$ where $\overline{\Tor}_1(\bm{B},\bm{L}) = \bm{0}$ since $\bm{B} \in \mathscr{F}_n(R\op)$, and $\bm{A} \overline{\otimes} \bm{L} \to \bm{B} \overline{\otimes} \bm{L}$ is a monomorphism since $\bm{\eta}$ is pure exact. It follows that $\overline{\Tor}_1(\bm{C},\bm{L}) = \bm{0}$, and hence $\bm{C} \in \mathscr{F}_n(R\op)$. In the case $n \geq 2$, it suffices to apply Propositions \ref{fi-if} and \ref{coproduct-closed} to show that $\bm{A}$ is also $\FP_n$-flat.
\end{proof}

\begin{example} \
\begin{itemize}
\item[(a)] If $M \in \mathcal{I}_n(R)$ then $D^m(M) \in \mathscr{I}_n(R)$ for any $m \in \mathbb{Z}$, by Theorem \ref{tfae1}. Similarly, if $N \in \mathcal{F}_n(R\op)$ then $D^m(N) \in \mathscr{F}_n(R\op)$ for any $m \in \mathbb{Z}$, by Theorem \ref{theo:char_flat}.

\item[(b)] Consider the ring $R = k[x_1,x_2,\cdots]/{\mathfrak{m}^2}$ as in Example \ref{ex1}. By the chain (\ref{aa}), we immediately obtain the following ascending chains:
\begin{align*}
\mathscr{I}_0(R) & \subseteq \mathscr{I}_1(R) \subseteq \mathscr{I}_2(R) = \mathscr{I}_3(R) = \cdots = \mathscr{I}_\infty(R),\\
\mathscr{F}_0(R\op) & = \mathscr{F}_1(R\op) \subseteq \mathscr{F}_2(R\op) = \mathscr{F}_3(R\op) = \cdots = \mathscr{F}_\infty(R\op).
\end{align*}
Moreover, by \cite[Example 5.7]{BP16} and Proposition \ref{fi-if}, $D^m(\langle x_1\rangle) \in \mathscr{I}_2(R) \backslash\mathscr{I}_1(R)$, and $D^m(\langle x_1\rangle)^+ \in \mathscr{F}_2(R\op) \backslash \mathscr{F}_1(R\op)$, that is, there are strict inclusions $\mathscr{I}_1(R) \subsetneq \mathscr{I}_2(R)$ and $\mathscr{F}_1(R\op) \subsetneq \mathscr{F}_2(R\op)$.
\end{itemize}
\end{example}


\section{\textbf{$\FP_n$-injective and $\FP_n$-flat dimensions}}\label{sec:FPninjflatdimensions}

In this section, we introduce and investigate $\FP_n$-injective and $\FP_n$-flat dimensions of modules and complexes. We also show that there exists a close link between these relative homological dimensions via Pontrjagin duality.


\subsection{$\FP_n$-injective and $\FP_n$-flat dimensions of modules}

First of all,  every module $M$ in $\Mod(R)$ has a coresolution by $\FP_n$-injective modules, that is, there exists an exact sequence
\begin{align}\label{eqn:cores}
\bm{\varepsilon} \colon & 0 \to M \to E^0 \to E^1 \to \cdots \to E^{k-1} \to E^k \to \cdots
\end{align}
in $\Mod(R)$ where $E^k \in \mathcal{I}_n(R)$ for every $k \geq 0$. This is due to the fact that for any ring $R$ and any $n \geq 0$, the class $\mathcal{I}_n(R)$ is the right half of a complete cotorsion pair (and so a special pre-enveloping class), proved by D. Bravo and the second author in \cite[Corollary 4.2]{BP16}. Whenever we are given a $\FP_n$-injective coresolution $\bm{\varepsilon}$, the module $\Omega^{-i}_{\bm{\varepsilon}}(M) := {\rm Ker}(E^i \to E^{i+1})$ is called the \emph{$\FP_n$-injective $i$-th cosyzygy of $M$ in $\bm{\varepsilon}$}, for any $i \geq 0$.

Dually, by \cite[Theorem 4.5]{BP16}, for any ring $R$ and any $n \geq 0$, the class $\mathcal{F}_n(R\op)$ is the left half of a complete cotorsion pair in $\Mod(R\op)$. It follows that for every module $N$ in $\Mod(R\op)$, there exists an exact sequence
\begin{align}\label{eqn:res}
\bm{\rho} \colon & \cdots \to Q_t \to Q_{t-1} \to \cdots \to Q_1 \to Q_0 \to N \to 0
\end{align}
in $\Mod(R\op)$ where $Q_t \in \mathcal{F}_n(R\op)$ for every $t \geq 0$. Whenever we are given an $\FP_n$-flat resolution $\bm{\rho}$, the module $\Omega^i_{\bm{\rho}}(N) := {\rm Im}(Q_{i} \to Q_{i-1})$ is called the \emph{$\FP_n$-flat $i$-th syzygy of $N$ in $\bm{\rho}$}, for any $i \geq 0$, where $Q_{-1} := N$. Based on the above, we now present the following.

\begin{definition} \
\begin{itemize}
\item[(a)] The \emph{$\FP_n$-injective dimension of a module $M$ in $\Mod(R)$}, denoted $\FP_n\mbox{-}\id_R(M)$, is defined as the smallest non-negative integer $k \geq 0$ such that $M$ has a coresolution by $\FP_n$-injective modules, as \eqref{eqn:cores}, with $E^i = 0$ for every $i > k$. If such $k$ does not exist, we set $\FP_n\mbox{-}\id_R(M) := \infty$.

\item[(b)] The \emph{$\FP_n$-flat dimension of a module $N$ in $\Mod(R\op)$}, denoted $\FP_n\mbox{-}\fd_{R\op}(N)$, is defined as the smallest non-negative integer $t \geq 0$ such that $N$ has a resolution by $\FP_n$-flat modules, as \eqref{eqn:res}, with $Q_i = 0$ for every $i > t$. If such $t$ does not exist, we set $\FP_n\mbox{-}\fd_{R\op}(N) := \infty$.
\end{itemize}
\end{definition}

Next, we give a functorial description of $\FP_n$-injective (resp., $\FP_n$-flat) dimension.

\begin{proposition} \label{prop:FPnid_char}
Let $M$ be a module in $\Mod(R)$ and $N$ be a module in $\Mod(R\op)$.
\begin{itemize}
\item[{\rm (a)}] The following are equivalent for every $n, k \geq 0$:
\begin{itemize}
\item[{\rm (1)}] $\FP_n\mbox{-}\id_R(M) \leq k$.

\item[{\rm (2)}] Every $\FP_n$-injective $k$-th cosyzygy of $M$ is $\FP_n$-injective.

\item[{\rm (3)}] Every injective $k$-th cosyzygy of $M$ is $\FP_n$-injective.

\item[{\rm (4)}] $\Ext^{k+1}_R(L,M) = 0$ for every $L \in \mathcal{FP}_n(R)$.
\end{itemize}

\item[{\rm (b)}] Dually, the following are equivalent for every $n, t \geq 0$:
\begin{itemize}
\item[{\rm (i)}]  $\FP_n\mbox{-}\fd_{R\op}(N) \leq t$.

\item[{\rm (ii)}] Every $\FP_n$-flat $t$-th syzygy of $N$ is $\FP_n$-flat.

\item[{\rm (iii)}] Every projective $t$-th syzygy of $N$ is $\FP_n$-flat.

\item[{\rm (iv)}] $\Tor^R_{t+1}(N,L) = 0$ for every $L \in \mathcal{FP}_n(R)$.
\end{itemize}
\end{itemize}
\end{proposition}

\begin{proof}
We only prove the equivalences concerning $\FP_n$-injectivity, as those involving $\FP_n$-flatness follow similarly.
\begin{itemize}
\item[$\bullet$] (1) $\Rightarrow$ (2): Suppose $\FP_n\mbox{-}\id_R(M) \leq k$. Then, we have an exact sequence
\[
0 \to M \to E^0 \to E^1 \to \cdots \to E^{k-1} \to E^k \to 0
\]
with $E^i \in \mathcal{I}_n(R)$ for every $0 \leq i \leq k$. On the other hand, suppose we are given a $\FP_n$-injective coresolution of $M$, say $\overline{\bm{\varepsilon}} \colon 0 \to M \to \overline{E}^0 \to \overline{E}^1 \to \cdots$. By the dual version of the generalized Schanuel's Lemma \cite[Corollary 8.6.4]{EJ00}, we have an isomorphism $\Omega^{-k}_{\overline{\bm{\varepsilon}}}(M) \oplus E^{k-1} \oplus \overline{E}^{k-2} \oplus \cdots \cong E^k \oplus \overline{E}^{k-1} \oplus E^{k-2} \oplus \cdots$, and so $\Omega^{-k}_{\overline{\bm{\varepsilon}}}(M) \in \mathcal{I}_n(R)$ since the class $\mathcal{I}_n(R)$ is closed under finite direct sums and direct summands.

\item[$\bullet$] (2) $\Rightarrow$ (3): Clear.

\item[$\bullet$] (3) $\Rightarrow$ (4): Consider an injective coresolution of $M$, say $\bm{\iota}$. Then, $\Omega^{-k}_{\bm{\iota}}(M) \in \mathcal{I}_n(R)$, and by dimension shifting, we have $\Ext^{k+1}_R(L,M) \cong \Ext^1_R(L,\Omega^{-k}_{\bm{\iota}}(M)) = 0$ for every $L \in \mathcal{FP}_n(R)$.

\item[$\bullet$] (4) $\Rightarrow$ (1): Suppose that $\Ext^{k+1}_R(L,M) = 0$ for every $L \in \mathcal{FP}_n(R)$ and consider an $\FP_n$-injective coresolution $\bm{\varepsilon}$ as \eqref{eqn:cores}. In particular, we can choose each $E^i$ to be injective. Let $L \in \mathcal{FP}_n(R)$. Since $\Ext^j_R(L,E^i) = 0$ for every $j > 0$ and $0 \leq i \leq k-1$, by dimension shifting we have $\Ext^1_R(L,\Omega^{-k}_{\bm{\varepsilon}}(M)) \cong \Ext^{k+1}_R(L,M) = 0$. Hence, $\Omega^{-k}_{\bm{\varepsilon}}(M) \in \mathcal{I}_n(R)$, and so $\FP_n\mbox{-}\id_R(M) \leq k$.

\end{itemize}
\end{proof}

As a consequence of the previous result, the $\FP_n$-injective dimension of a module $M$ in $\Mod(R)$ (in the case it is finite) can also be defined as the smallest non-negative integer $k$ such that $\Ext^{k+1}_R(L,M) = 0$ for every $L \in \mathcal{FP}_n(R)$. The $\FP_n$-flat dimension of every module $N$ in $\Mod(R\op)$ has also a similar functorial description in terms of the torsion functors $\Tor^R(-,-)$. 

We conclude our study of $\FP_n$-injective and $\FP_n$-flat dimensions of modules presenting the interplay between the two via the notion of Pontrjagin dual.

\begin{proposition}\label{prop:char_dims}
The following hold for every module $M$ in $\Mod(R)$ and $N$ in $\Mod(R\op)$:
\begin{itemize}
\item[{\rm (a)}] $\FP_n\mbox{-}\fd_{R\op}(N) = \FP_n\mbox{-}\id_R(N^+)$, for every $n \geq 0$.

\item[{\rm (b)}] $\FP_n\mbox{-}\id_R(M) = \FP_n\mbox{-}\fd_{R\op}(M^+)$, for every $n \geq 2$.
\end{itemize}
\end{proposition}

\begin{proof}
 Part (a) follows by Proposition \ref{prop:FPnid_char} and the natural isomorphism $\Ext^i_R(M,N^+) \cong \Tor^R_i(N,M)^+$, for every $M$ in $\Mod(R)$ and $N$ in $\Mod(R\op)$ \cite[Theorem 3.2.1]{EJ00}. We focus on proving (b).

First, let us consider the case $\FP_n\mbox{-}\id_R(M) = \infty$. Consider an $\FP_n$-injective coresolution $\bm{\varepsilon}$ of $M$, as in \eqref{eqn:cores}. Then, we have an $\FP_n$-flat resolution $\bm{\varepsilon}^+ \colon \cdots \to (E^1)^+ \to (E^0)^+ \to M^+ \to 0$. If $\FP_n\mbox{-}\fd_{R\op}(M^+) = t < \infty$, then $\Omega^t_{\bm{\varepsilon}^+}(M^+) \simeq (\Omega^{-t}_{\bm{\varepsilon}}(M))^+$ would be $\FP_n$-flat, and so $\Omega^{-t}_{\bm{\varepsilon}}(M)$ would be $\FP_n$-injective by \cite[Proposition 3.6]{BP16} since $n \geq 2$, thus getting a contradiction. It follows $\FP_n\mbox{-}\fd_{R\op}(M^+) = \infty$.

Now suppose that $\FP_n\mbox{-}\fd_{R\op}(M^+) = \infty$ and $\FP_n\mbox{-}\id_R(M) = k < \infty$. Then, there is an exact sequence $0 \to M \to E^0 \to E^1 \to \cdots \to E^{k-1} \to E^k \to 0$ with $E^i \in \mathcal{I}_n(R)$ for every $0 \leq i \leq k$. Since $n \geq 2$ and the functor $\Hom_{\mathbb{Z}}(-,\mathbb{Q / Z})$ is exact, the previous sequence gives rise to an $\FP_n$-flat resolution of $M^+$ of length $k$, thus getting a contradiction. It follows $\FP_n\mbox{-}\id_R(M) = \infty$.

Finally, assume $\FP_n\mbox{-}\id_R(M) = k < \infty$ and $\FP_n\mbox{-}\fd_{R\op}(M^+) = t < \infty$. Using the same arguments as in the previous paragraph, we can assert that $\FP_n\mbox{-}\fd_{R\op}(M^+) \leq k$. Now consider a partial injective coresolution of $M$ of length $t$, say
\[
0 \to M \to I^0 \to I^1 \to \cdots \to I^{t-1} \to M' \to 0.
\]
Then, we have the exact sequence
\[
0 \to (M')^+ \to (I^{t-1})^+ \to \cdots \to (I^1)^+ \to (I^0)^+ \to M^+ \to 0
\]
in $\Mod(R\op)$ where $(I^j)^+ \in \mathcal{F}_n(R\op)$ for every $0 \leq j \leq t-1$ since $n \geq 2$, and so $(M')^+ \in \mathcal{F}_n(R\op)$ since $\FP_n\mbox{-}\fd_{R\op}(M^+) = t$. It follows that $M' \in \mathcal{I}_n(R)$. Hence, we have $\FP_n\mbox{-}\id_R(M) \leq t$.
\end{proof}


\subsection{$\FP_n$-injective and $\FP_n$-flat dimensions of complexes}

We present the analogous concepts of $\FP_n$-injective and $\FP_n$-flat dimensions for complexes.

We already know that for left and right $R$-modules we can always construct coresolutions by $\FP_n$-injective modules and resolutions by $\FP_n$-flat modules. In order to assert that the same happens in the category of complexes, we need the analogous of the complete cotorsion pairs $({}^\perp(\mathcal{I}_n(R)),\mathcal{I}_n(R))$ and $(\mathcal{F}_n(R\op),(\mathcal{F}_n(R\op))^\perp)$ in $\Ch(R)$ and $\Ch(R\op)$, respectively. Thanks to the works \cite{GillespieDegree} and \cite{Estrada} by Gillespie, and M. Cort\'es Izurdiaga, S. Estrada and P. A. Guil Asensio, we know methods to induce certain complete cotorsion pairs in $\Ch(R)$ from a cotorsion pair in $\Mod(R)$ cogenerated by a set. Specifically, Gillespie proved in \cite[Proposition 4.3]{GillespieDegree} that if $(\mathcal{A,B})$ is a cotorsion pair in $\Mod(R)$ cogenerated by a set, then $({}^\perp\tilde{\mathcal{B}},\tilde{\mathcal{B}})$ is a cotorsion pair in $\Ch(R)$ cogenerated by a set (and so complete), where $\tilde{\mathcal{B}}$ is defined as the class of exact complexes with cycles in $\mathcal{B}$. On the other hand, $(\tilde{\mathcal{A}},\tilde{\mathcal{A}}^\perp)$ is also a complete cotorsion pair in $\Ch(R)$ by \cite[Theorem 1.5]{Estrada}, where $\tilde{\mathcal{A}}$ is the class of exact complexes with cycles in $\mathcal{A}$. To apply these results to the context of the present paper, we know by \cite[Corollary 4.2]{BP16} that $({}^\perp(\mathcal{I}_n(R)),\mathcal{I}_n(R))$ is a cotorsion pair in $\Mod(R)$ cogenerated by a set, and that the same is true for the cotorsion pair $(\mathcal{F}_n(R\op),(\mathcal{F}_n(R\op))^\perp)$ by \cite[Theorem 4.5]{BP16}\footnote{From \cite{BP16} we can only assert that the pair $(\mathcal{F}_n(R\op),(\mathcal{F}_n(R\op))^\perp)$ is complete. The fact that it has a cogenerating set follows as in \cite[Theorem 2.9]{EnochsKaplansky}}. It follows that
\[
({}^\perp(\widetilde{\mathcal{I}_n(R)}),\widetilde{\mathcal{I}_n(R)}) \mbox{ \ and \ } (\widetilde{\mathcal{F}_n(R\op)},(\widetilde{\mathcal{F}_n(R\op)})^\perp)
\]
are complete cotorsion pairs in $\Ch(R)$ and $\Ch(R\op)$, respectively. We are ready to prove the following result.

\begin{proposition}
The following statements hold for any $n \geq 0$:
\begin{itemize}
\item[{\rm (a)}] $({}^\perp(\mathscr{I}_n(R)),\mathscr{I}_n(R))$ is a complete cotorsion pair in $\Ch(R)$.

\item[{\rm (b)}] $(\mathscr{F}_n(R\op),(\mathscr{F}_n(R\op))^\perp)$ is a perfect cotorsion pair in $\Ch(R\op)$.
\end{itemize}
\end{proposition}

\begin{proof}
Part (a) follows by Theorem \ref{tfae1} and the previous comments. In a similar way, using Theorem \ref{theo:char_flat}, we have that the cotorsion pair $(\mathscr{F}_n(R\op),(\mathscr{F}_n(R\op))^\perp)$ is complete. And since $\mathscr{F}_n(R\op)$ is closed under direct limits by Proposition \ref{prop:flatness_dl}, we gave that the previous pair is perfect by \cite[Theorem 7.2.6]{EJ00}\footnote{Although the cited result is stated in the category of modules, the arguments in its proof carry over to the category of complexes.}.
\end{proof}

Now, we can assert that for any complex $\bm{X}$ in $\Ch(R)$ and $\bm{Y}$ in $\Ch(R\op)$, we can construct exact sequences
\begin{align}
\bm{\varepsilon} \colon & \bm{0} \to \bm{X} \to \bm{E^0} \to \bm{E^1} \to \cdots \to \bm{E^{k-1}} \to \bm{E^k} \to \cdots, \label{eqn:coresCh} \\
\bm{\rho} \colon & \cdots \to \bm{Q_t} \to \bm{Q_{t-1}} \to \cdots \to \bm{Q_1} \to \bm{Q_0} \to \bm{Y} \to \bm{0}, \label{eqn:resCh}
\end{align}
with $\bm{E^k} \in \mathscr{I}_n(R)$ for every $k \geq 0$, and $\bm{Q_t} \in \mathscr{F}_n(R\op)$ for every $t \geq 0$, that is, any complex has a coresolution by $\FP_n$-injective complexes, and a resolution by $\FP_n$-flat complexes. Thus, the following definition makes sense.

\begin{definition} \
\begin{itemize}
\item[(a)] The \emph{$\FP_n$-injective dimension of a complex $\bm{X}$ in $\Ch(R)$}, denoted $\FP_n\mbox{-}\id(\bm{X})$, is defined as the smallest non-negative integer $k \geq 0$ such that $\bm{X}$ has a coresolution by $\FP_n$-injective complexes, as \eqref{eqn:coresCh}, with $\bm{E^i} = \bm{0}$ for every $i > k$. If such $k$ does not exist, we set $\FP_n\mbox{-}\id(\bm{X}) := \infty$.

\item[(b)] The \emph{$\FP_n$-flat dimension of a complex $\bm{Y}$ in $\Ch(R\op)$}, denoted $\FP_n\mbox{-}\fd(\bm{Y})$, is defined as the smallest non-negative integer $t \geq 0$ such that $\bm{Y}$ has a resolution by $\FP_n$-flat complexes, as \eqref{eqn:resCh}, with $\bm{Q_i} = \bm{0}$ for every $i > t$. If such $t$ does not exist, we set $\FP_n\mbox{-}\fd(\bm{Y}) := \infty$.
\end{itemize}
\end{definition}

J. R. Garc\'ia Rozas proved in \cite[Theorem 3.1.3]{GR99} that for any complex $\bm{X}$ in $\Ch(R)$, the injective dimension of $\bm{X}$ in $\Ch(R)$ is at most $k$ if, and only if, $\bm{X}$ is exact and the injective dimension of $Z_m(\bm{X})$ in $\Mod(R)$ is at most $k$ for any $m \in \mathbb{Z}$. Complexes with bounded flat dimension have a similar description, as proved in \cite[Lemma 5.4.1]{GR99}. Motivated by these results, we present the analogous for complexes of Proposition \ref{prop:FPnid_char}.

\begin{proposition} \label{fi-dim}
Let $\bm{X}$ be a complex in $\Ch(R)$ and $\bm{Y}$ be a complex in $\Ch(R\op)$.
\begin{itemize}
\item[{\rm (a)}] The following are equivalent for every $n, k \geq 0$:
\begin{itemize}
\item[{\rm (1)}] $\FP_n\mbox{-}\id(\bm{X}) \leq k$.

\item[{\rm (2)}] $\bm{X}$ is exact and $\FP_n\mbox{-}\id_R(Z_m(\bm{X})) \leq k$ for any $m \in \mathbb{Z}$.

\item[{\rm (3)}] Every $\FP_n$-injective $k$-th cosyzygy of $\bm{X}$ is $\FP_n$-injective.

\item[{\rm (4)}] Every injective $k$-th cosyzygy of $\bm{X}$ is $\FP_n$-injective.

\item[{\rm (5)}] $\underline{\Ext}^{k+1}(\bm{L},\bm{X}) = \bm{0}$ for every $\bm{L} \in \mathscr{FP}_n(R)$.

\item[{\rm (6)}] $\Ext^{k+1}_{\Ch}(\bm{L},\bm{X}) = 0$ for every $\bm{L} \in \mathscr{FP}_n(R)$.

\item[{\rm (7)}] $\Ext^{k+1}_{\Ch}(S^m(L),\bm{X}) = 0$ for every $m \in \mathbb{Z}$ and $L \in \mathcal{FP}_n(R)$.
\end{itemize}

\item[{\rm (b)}] Dually, the following are equivalent for every $n, t \geq 0$:
\begin{itemize}
\item[{\rm (i)}]  $\FP_n\mbox{-}\fd(\bm{Y}) \leq t$.

\item[{\rm (ii)}] $\bm{Y}$ is exact and $\FP_n\mbox{-}\fd_{R\op}(Z_m(\bm{Y})) \leq t$ for any $m \in \mathbb{Z}$.

\item[{\rm (iii)}] Every $\FP_n$-flat $t$-th syzygy of $\bm{Y}$ is $\FP_n$-flat.

\item[{\rm (iv)}] Every projective $t$-th syzygy of $\bm{Y}$ is $\FP_n$-flat.

\item[{\rm (v)}] $\overline{\Tor}_{t+1}(\bm{Y},\bm{L}) = \bm{0}$ for every $\bm{L} \in \mathscr{FP}_n(R)$.

\item[{\rm (vi)}] $\overline{\Tor}_{t+1}(\bm{Y},S^m(L)) = \bm{0}$ for every $m \in \mathbb{Z}$ and $L \in \mathcal{FP}_n(R)$. 
\end{itemize}
\end{itemize}
\end{proposition}

\begin{proof}
We only prove part (a), as (b) is dual. Note that the equivalences (1) $\Leftrightarrow$ (3) $\Leftrightarrow$ (4) $\Leftrightarrow$ (5) follow as in Proposition \ref{prop:FPnid_char}. We only focus on showing (1) $\Leftrightarrow$ (2) and (5) $\Leftrightarrow$ (6) $\Leftrightarrow$ (7). 

Suppose first that $\FP_n\mbox{-}\id(\bm{X}) \leq k$ and consider an $\FP_n$-injective coresolution $\bm{\varepsilon}$ of $\bm{X}$ as in \eqref{eqn:coresCh}, with $\bm{E^i} = \bm{0}$ for every $i > k$. By Proposition \ref{tfae1}, each $\bm{E^i}$ is an exact complex. Thus, we have that $\bm{X}$ is exact, since the class of exact complexes is thick. On the other hand, we have an exact sequence $0 \to Z_m(\bm{X}) \to Z_m(\bm{E^0}) \to Z_m(\bm{E^1}) \to \cdots \to Z_m(\bm{E^{k-1}}) \to Z_m(\bm{E^k}) \to 0$ in $\Mod(R)$ for every $m \in \mathbb{Z}$, due to \cite[Lemma 3.3.9]{PerezBook}. By Proposition \ref{tfae1} again, $Z_m(\bm{E^i}) \in \mathcal{I}_n(R)$ for every $m \in \mathbb{Z}$. Therefore, we have $\FP_n\mbox{-}\id_R(Z_m(\bm{X})) = \Gamma_R(Z_m(\bm{X})) \leq k$ for every $m \in \mathbb{Z}$.

Now assume that (2) holds. Let $\bm{0} \to \bm{X} \to \bm{E^0} \to \bm{E^1} \to \cdots \to \bm{E^{k-1}} \to \bm{X}' \to \bm{0}$ be an exact sequence in $\Ch(R)$ with $\bm{E^i} \in \mathscr{I}_n(R)$ for every $0 \leq i \leq k-1$. Being $\bm{X}'$ exact, we only need to show by Proposition \ref{tfae1} that $Z_m(\bm{X}') \in \mathcal{I}_n(R)$ for every $m \in \mathbb{Z}$, in order to assert that $\bm{X}' \in \mathscr{I}_n(R)$. This follows by using again \cite[Lemma 3.3.9]{PerezBook}.

Finally, we can note that (5) $\Leftrightarrow$ (6) is a consequence of \eqref{eqn:underext}, while (6) $\Rightarrow$ (7) is clear by the characterization of complexes of type $\FP_n$ proved in the previous section. Conversely, if we assume (7), it suffices to apply Theorem \ref{tfae1} along with induction on $k$ to conclude (6). 
\end{proof}

As a consequence of Proposition \ref{fi-dim}, the $\FP_n$-injective dimension of a complex $\bm{X}$ in $\Ch(R)$ can also be defined (in the case it is finite) as the smallest non-nogative integer $k$ such that $\underline{\Ext}^{k+1}(\bm{L},\bm{X}) = \bm{0}$ for every $\bm{L} \in \mathscr{FP}_n(R)$. We can also deduce the following.

\begin{corollary}
Let $\bm{X}$  be an exact complex in $\Ch(R)$ and $\bm{Y}$ be an exact complex in $\Ch(R\op)$. Then, the following equalities hold:
\begin{itemize}
\item[{\rm (a)}] $\FP_n\mbox{-}\id(\bm{X}) = \sup \{ \FP_n\mbox{-}\id_R(Z_m(\bm{X})) \mbox{ {\rm :} } m \in \mathbb{Z}\}$.

\item[{\rm (b)}] $\FP_n\mbox{-}\fd(\bm{Y}) = \sup \{ \FP_n\mbox{-}\fd_{R\op}(Z_m(\bm{Y})) \mbox{ {\rm :} } m \in \mathbb{Z}\}$.
\end{itemize}
\end{corollary}

\begin{proof}
We only prove part (a) in the cases where $\sup \{ \FP_n\mbox{-}\id_R(Z_m(\bm{X})) \mbox{ : } m \in \mathbb{Z}\} = \infty$ and $\FP_n\mbox{-}\id(\bm{X}) = \infty$. If we assume $\FP_n\mbox{-}\id(\bm{X}) = \infty$ and $\sup \{ \FP_n\mbox{-}\id_R(Z_m(\bm{X})) \mbox{ : } m \in \mathbb{Z}\} = k \leq \infty$, then since $\bm{X}$ is exact, we would have $\FP_n\mbox{-}\id(\bm{X}) \leq k$ by Proposition \ref{fi-dim}, and thus getting a contradiction. Similarly, if we assume $\sup \{ \FP_n\mbox{-}\id_R(Z_m(\bm{X})) \mbox{ : } m \in \mathbb{Z}\} = \infty$, we can conclude $\FP_n\mbox{-}\id(\bm{X}) = \infty$.
\end{proof}

We finish this section with the following proposition, which illustrates that, as it happens with modules, there exists a close relation between the $\FP_n$-injective and the $\FP_n$-flat dimension of complexes. We need the following preliminary result, which follows by the fact that exact functors preserve homology.

\begin{lemma}\label{lem:exact_Pontrjagin}
A complex $\bm{X}$ in $\Ch(R)$ is exact if, and only if, $\bm{X}^+$ is exact in $\Ch(R\op)$.
\end{lemma}

\begin{proposition}\label{dim-dual}
For any complex $\bm{X}$ in $\Ch(R)$ and $\bm{Y}$ in $\Ch(R\op)$, the following equalities hold true:
\begin{itemize}
\item[{\rm (a)}] $\FP_n\mbox{-}\fd(\bm{Y}) = \FP_n\mbox{-}\id(\bm{Y^+})$, for any $n \geq 0$.

\item[{\rm (b)}] $\FP_n\mbox{-}\id(\bm{X}) = \FP_n\mbox{-}\fd(\bm{X^+})$, for any $n \geq 2$.
\end{itemize}
\end{proposition}

\begin{proof}
We only prove part (b), as (a) will follow in a similar way. For the cases $\FP_n\mbox{-}\id(\bm{X}) = \infty$ and $\FP_n\mbox{-}\fd(\bm{X^+}) = \infty$ the result holds by Proposition \ref{fi-dim}. Now, let $k$ be a non-negative integer. Then:

\begin{align*}
\FP_n\mbox{-}\fd(\bm{X^+}) \leq k & \Leftrightarrow \bm{X^+} \mbox{ is exact and } \FP_n\mbox{-}\fd_{R\op}(Z_m(\bm{X^+})) \leq k \mbox{ for all } m \in \mathbb{Z} \ \ (\mbox{by \ref{fi-dim}}) \\
& \Leftrightarrow \bm{X^+} \mbox{ is exact and } \FP_n\mbox{-}\fd_{R\op}(Z_m(\bm{X})^+) \leq k \mbox{ for all } m \in \mathbb{Z} \\
& \Leftrightarrow \bm{X} \mbox{ is exact and } \FP_n\mbox{-}\id_R(Z_m(\bm{X})) \leq k \mbox{ for all } m \in \mathbb{Z} \ \ (\mbox{by \ref{prop:char_dims} and \ref{lem:exact_Pontrjagin}}) \\
& \Leftrightarrow \FP_n\mbox{-}\id(\bm{X}) \leq k \ \ (\mbox{by \ref{fi-dim}}).
\end{align*}
Hence, (b) follows.
\end{proof}


\section{\textbf{Pre-envelopes and covers by $\FP_n$-injective and $\FP_n$-flat complexes}}\label{sec:InjFlatCom}

In this section, we will investigate two classes of complexes, namely complexes of $\FP_n$-injective dimension at most $k$ and that of $\FP_n$-flat dimension at most $k$, respectively, and prove the existence of the corresponding covers and pre-envelopes. We will first investigate the same classes but in the category of modules. From them we will obtain a construction known as duality pairs, and later on we will prove some general methods to produce dual pairs of complexes from duality pairs of modules. These methods will simplify the process to obtain the covers and pre-envelopes mentioned before. We will mainly use a result of Holm and J\o rgensen \cite{duality} about duality pairs and perfect cotorsion pairs of modules, and an analogous result of Yang \cite{Ya} for chain complexes.


\subsection{Duality pairs from $\FP_n$-injective and $\FP_n$-flat dimensions of modules}

For any non-negative integer $k$, let $\mathcal{I}_{(n,k)}(R)$ denote the class of modules in $\Mod(R)$ with $\FP_n$-injective dimension at most $k$, and $\mathcal{F}_{(n,k)}(R\op)$ the class of modules in $\Mod(R\op)$ with $\FP_n$-flat dimension at most $k$. Note that $\mathcal{I}_{(n,0)}(R) = \mathcal{I}_n(R)$ and $\mathcal{F}_{(n,0)}(R\op) = \mathcal{F}_n(R\op)$, while $\mathcal{I}_{(0,k)}(R)$ is the class of modules in $\Mod(R)$ with injective dimension at most $k$, and $\mathcal{F}_{(0,k)}(R\op) = \mathcal{F}_{(1,k)}(R\op)$ is the class of modules in $\Mod(R\op)$ with flat dimension at most $k$. In what follows, we will see that $\mathcal{F}_{(n,k)}(R\op)$ is always a covering class, while the same is true for the class $\mathcal{I}_{(n,k)}(R)$ in the case where $n \geq 2$.

The notion of covers is associated to that of duality pairs, in the sense that the latter comprises enough properties to obtaining perfect cotorsion pairs.

Recall from \cite[Definition 2.1]{duality} that a \emph{duality pair over $R$} is a pair $(\mathcal{M,C})$, where $\mathcal{M}$ is a class of left (resp., right) $R$-modules, and $\mathcal{C}$ is a class of right (resp., left) $R$-modules, subject to the following conditions:
\begin{itemize}
\item[(a)] \underline{Duality property}: $M \in \mathcal{M}$ $\Leftrightarrow$ $M^+ \in \mathcal{C}$, for any $M$ in $\Mod(R)$ (resp., in $\Mod(R\op)$).

\item[(b)] $\mathcal{C}$ is closed under direct summands and finite direct sums.
\end{itemize}
A duality pair $(\mathcal{M,C})$ over $R$ is called:
\begin{itemize}
\item[$\bullet$] \emph{(co)product-closed} if the class $\mathcal{M}$ is closed under arbitrary (co)products in the category of left $R$-modules;

\item[$\bullet$] \emph{perfect} if it is coproduct-closed, $\mathcal{M}$ is closed under extensions, and $R \in \mathcal{M}$.
\end{itemize}
We construct new examples of duality pairs from $\mathcal{I}_{(n,k)}(R)$ and $\mathcal{F}_{(n,k)}(R\op)$.

\begin{theorem}\label{theo:duality_pairs}
The following statements hold true for every $k \geq 0$:
\begin{itemize}
\item[{\rm (a)}] $(\mathcal{F}_{(n,k)}(R\op),\mathcal{I}_{(n,k)}(R))$ is a perfect duality pair over $R$ for any $n \geq 0$. Moreover:
\begin{itemize}
\item[{\rm (a.1)}] For any $n \geq 2$, $(\mathcal{F}_{(n,k)}(R\op),\mathcal{I}_{(n,k)}(R))$ is product-closed.

\item[{\rm (a.2)}] $R$ is right coherent if, and only if, $(\mathcal{F}_{(1,k)}(R\op),\mathcal{I}_{(1,k)}(R))$ is product-closed.
\end{itemize}

\item[{\rm (b)}] $(\mathcal{I}_{(n,k)}(R),\mathcal{F}_{(n,k)}(R\op))$ is a (co)product-closed duality pair over $R$ for any $n \geq 2$, with the class $\mathcal{I}_{(n,k)}(R)$ closed under extensions. Moreover: 
\begin{itemize}
\item[{\rm (b.1)}] $(\mathcal{I}_{(n,k)}(R),\mathcal{F}_{(n,k)}(R\op))$ is perfect if, and only if, $\FP_n\mbox{-}\id_R(R) \leq k$.

\item[{\rm (b.2)}] $R$ is a left coherent ring if, and only if, $(\mathcal{I}_{(1,k)}(R),\mathcal{F}_{(1,k)}(R\op))$ is a coproduct-closed duality pair over $R$.

\item[{\rm (b.3)}] $R$ is a left noetherian ring if, and only if, $(\mathcal{I}_{(0,k)}(R),\mathcal{F}_{(0,k)}(R\op))$ is a coproduct-closed duality pair over $R$.
\end{itemize}
\end{itemize}
\end{theorem}

\begin{proof} \
\begin{itemize}
\item[(a)] The class $\mathcal{I}_{(n,k)}(R)$ is clearly  closed under direct summands and finite direct sums, so by Proposition \ref{prop:char_dims} (a) it follows that $(\mathcal{F}_{(n,k)}(R\op),\mathcal{I}_{(n,k)}(R))$ is a duality pair. Since $\Tor^R_{k+1}(-,M)$ preserves direct limits for every module $M$ in $\Mod(R)$, we have that $\mathcal{F}_{(n,k)}(R\op)$ is closed under coproducts. Using long exact sequences of $\Tor^R_i(-,-)$, one can easily note that $\mathcal{F}_{(n,k)}(R\op)$ is also closed under extensions. Finally, it is clear that $R \in \mathcal{F}_{(n,k)}(R\op)$. Hence, the duality pair $(\mathcal{F}_{(n,k)}(R\op),\mathcal{I}_{(n,k)}(R))$ is perfect.
\begin{itemize}
\item[(a.1)] If $n \geq 2$, the class of $\FP_n$-flat modules in $\Mod(R\op)$ is closed under products for any ring $R$, by \cite[Proposition 3.11]{BP16}. And by Proposition \ref{prop:FPnid_char} and the fact that products of exact sequences in $\Mod(R\op)$ are exact, we have that $\mathcal{F}_{(n,k)}(R\op)$ is closed under products.

\item[(a.2)] Set $n = 1$. By \cite[Theorem 2.3.2]{Glaz}, a ring is right coherent if, and only if, the class of flat modules in $\Mod(R\op)$ is closed under products. Using again the fact that the product of exact sequences in $\Mod(R\op)$ is exact, the result follows.
\end{itemize}

\item[(b)] It is clear that $\mathcal{F}_{(n,k)}(R\op)$ is closed under direct summands and finite direct sums, and by Proposition \ref{prop:char_dims} (b), we have that $(\mathcal{I}_{(n,k)}(R),\mathcal{F}_{(n,k)}(R\op))$ is a duality pair. Note also that $\mathcal{I}_{(n,k)}(R)$ is clearly closed under products and extensions. Moreover, we know from the proof of Proposition \ref{coproduct-closed} that the class of $\FP_n$-injective modules is closed under coproducts if $n \geq 1$. It follows by Proposition \ref{prop:FPnid_char} that $\mathcal{I}_{(n,k)}(R)$ is closed under coproducts.
\begin{itemize}
\item[(b.1)] Clear.

\item[(b.2)] If $(\mathcal{I}_{(1,k)}(R),\mathcal{F}_{(1,k)}(R\op))$ is a duality pair for every $k \geq 0$, in particular for $k = 0$, that is, $(\mbox{absolutely pures, flats})$ is a duality pair, and so a left $R$-module $M$ is absolutely pure if, and only if, $M^+$ is flat. This implies that $R$ is left coherent by \cite[Theorem 3.1]{DingChen}. Now if $R$ is a left coherent ring, on can deduce from \cite[Theorem 2.2]{Fe72} that the pair $(\mathcal{I}_{(1,k)}(R),\mathcal{F}_{(1,k)}(R\op))$ is a duality pair, for every $k \geq 0$.

\item[(b.3)] If $(\mathcal{I}_{(0,k)}(R),\mathcal{F}_{(0,k)}(R\op))$ is a coproduct-closed duality pair for every $k \geq 0$, in particular for $k = 0$, that is, $(\mbox{injectives, flats})$ is a coproduct-closed duality pair. It follows that $R$ is left noetherian. Now if $R$ is a left noetherian ring, we know that  the class of injective modules in $\Mod(R)$ is closed under coproducts, and by \cite[Theorem 2.2]{Fe71} we have that $(\mathcal{I}_{(0,k)}(R),\mathcal{F}_{(0,k)}(R\op))$ is a duality pair.
\end{itemize}
\end{itemize}
\end{proof}

In \cite[Theorem 3.1]{duality}, it is proven that if $(\mathcal{M,C})$ is a duality pair, then $\mathcal{M}$ is closed under pure sub-modules, pure quotients, and pure extensions. Furthermore, the following hold:
\begin{itemize}
\item[(a)] If $(\mathcal{M,C})$ is product-closed, then $\mathcal{M}$ is pre-enveloping.

\item[(b)] If $(\mathcal{M,C})$ is coproduct-closed, then $\mathcal{M}$ is covering.

\item[(c)] If $(\mathcal{M,C})$ is perfect, then $(\mathcal{M},\mathcal{M}^\perp)$ is a perfect cotorsion pair.
\end{itemize}
Combining these results with Theorem \ref{theo:duality_pairs} gives us the following.

\begin{corollary} \label{coro:more_properties}
The following statements hold true for every $k \geq 0$:
\begin{itemize}
\item[{\rm (a)}] The class $\mathcal{F}_{(n,k)}(R\op)$ is closed under pure sub-modules, pure quotients, and pure extensions for any $n \geq 0$.

\item[{\rm (b)}] The class $\mathcal{I}_{(n,k)}(R)$ is closed under pure sub-modules, pure quotients, and pure extensions for any $n \geq 2$.

\item[{\rm (c)}] For any $n \geq 0$, $(\mathcal{F}_{(n,k)}(R\op),(\mathcal{F}_{(n,k)}(R\op))^\perp)$ is a perfect cotorsion pair, and so every right $R$-module has a cover by a module with $\FP_n$-flat dimension at most $k$.

\item[{\rm (d)}] If $R$ is a right coherent ring, then every right $R$-module has a pre-envelope by a module with flat dimension at most $k$, and a cover by a module with absolutely pure dimension at most $k$.

\item[{\rm (e)}] If $R$ a left noetherian ring, then every left $R$-module has a cover by a module with injective dimension at most $k$.

\item[{\rm (f)}] For any $n \geq 2$, every right $R$-module has a pre-envelope by a module with $\FP_n$-flat dimension at most $k$, and every left $R$-module has a pre-envelope and a cover by a module with $\FP_n$-injective dimension at most $k$. In the case where $\FP_n\mbox{-}\id_R(R) \leq k$, $(\mathcal{I}_{(n,k)}(R),(\mathcal{I}_{(n,k)}(R))^\perp)$ is a perfect cotorsion pair.
\end{itemize}
\end{corollary}


\subsection{Induced dual pairs in chain complexes}

It is known that from a cotorsion pair $(\mathcal{A,B})$ cogenerated by a set in the category of modules, we can induce a series of complete cotorsion pairs in the category of complexes. This was pioneered by Gillespie in \cite{GillespieFlat} and \cite{GillespieDegree}. Following the spirit of these works, and being aware that there is a relation between duality pairs and perfect cotorsion pairs (in the categories of modules \cite{duality} and complexes \cite{Ya}), we are interested in inducing dual pairs of complexes from duality pairs of modules.

Recall from \cite[Definition 3.1]{Ya} that a \emph{dual pair over a ring $R$} is a pair $(\mathscr{M,C})$, where $\mathscr{M}$ is a class of complexes of left (resp., right) $R$-modules, and $\mathscr{C}$ is a class of complexes of right (resp., left) $R$-modules, subject to the following conditions:
\begin{itemize}
\item[(a)] \underline{Duality property}: $\bm{X} \in \mathscr{M}$ $\Leftrightarrow$ $\bm{X^+} \in \mathscr{C}$, for any complex $\bm{X}$ in $\Ch(R)$ (resp., in $\Ch(R\op)$).

\item[(b)] $\mathscr{C}$ is closed under direct summands and finite direct sums.
\end{itemize}
A dual pair $(\mathscr{M,C})$ is called:
\begin{itemize}
\item[$\bullet$] \emph{(co)product-closed} if $\mathscr{M}$ is closed under (co)products;

\item[$\bullet$] \emph{perfect} if it is coproduct-closed, $\mathscr{M}$ is closed under extensions, and $D^0(R) \in \mathscr{M}$.
\end{itemize}

Given a class $\mathcal{M}$ of left (resp., right) $R$-modules, recall from \cite{GillespieFlat,GillespieDegree} the following classes of complexes:
\begin{align*}
{\rm dw}\widetilde{\mathcal{M}} & := \{ \bm{X} \in \Ch(R) \mbox{ : } X_m \in \mathcal{M} \mbox{ for every } m \in \mathbb{Z} \}, \\
{\rm ex}\widetilde{\mathcal{M}} & := {\rm dw}\widetilde{\mathcal{M}} \cap \mathcal{E}(R),
\end{align*}
where $\mathcal{E}(R)$ denotes the class of exact complexes in $\Ch(R)$. Recall also the class $\widetilde{\mathcal{M}}$ defined previously. We have the following result.

\begin{theorem}\label{theo:induced_duality}
Let $(\mathcal{M,C})$ be a duality pair over $R$. Then, $({\rm dw}\widetilde{\mathcal{M}},{\rm dw}\widetilde{\mathcal{C}})$, $({\rm ex}\widetilde{\mathcal{M}},{\rm ex}\widetilde{C})$ and $(\widetilde{\mathcal{M}},\widetilde{\mathcal{C}})$ are dual pairs over $R$. Moreover:
\begin{itemize}
\item[$\bullet$] If $(\mathcal{M,C})$ is (co)product-closed, then so are $({\rm dw}\widetilde{\mathcal{M}},{\rm dw}\widetilde{\mathcal{C}})$, $({\rm ex}\widetilde{\mathcal{M}},{\rm ex}\widetilde{C})$ and $(\widetilde{\mathcal{M}},\widetilde{\mathcal{C}})$.

\item[$\bullet$] If $(\mathcal{M,C})$ is perfect, then so are $({\rm dw}\widetilde{\mathcal{M}},{\rm dw}\widetilde{\mathcal{C}})$, $({\rm ex}\widetilde{\mathcal{M}},{\rm ex}\widetilde{C})$ and $(\widetilde{\mathcal{M}},\widetilde{\mathcal{C}})$.
\end{itemize}
Conversely, let $\mathcal{M}$ be a class of modules in $\Mod(R)$ and $\mathcal{C}$ be a class of modules in $\Mod(R\op)$. If $({\rm dw}\widetilde{M},{\rm dw}\widetilde{\mathcal{C}})$, $({\rm ex}\widetilde{\mathcal{M}},{\rm ex}\widetilde{\mathcal{C}})$ or $(\widetilde{\mathcal{M}},\widetilde{\mathcal{C}})$ is a dual pair over $R$, then $(\mathcal{M,C})$ is a duality pair over $R$. Moreover, if any of these pairs is (co)product-closed or perfect, then so is $(\mathcal{M,C})$.
\end{theorem}

\begin{proof}
We split the proof into several parts:
\begin{itemize}
\item[$\bullet$] \underline{$({\rm dw}\widetilde{\mathcal{M}},{\rm dw}\widetilde{\mathcal{C}})$ is a dual pair}: The duality property follows by \eqref{eqn:alter_Pontrjagin}. On the other hand, since coproducts are computed component-wise, we have that ${\rm dw}\widetilde{\mathcal{C}}$ is closed under finite direct sums. Also, it is easy to see that ${\rm dw}\widetilde{\mathcal{C}}$ is also closed under direct summands. Hence, $({\rm dw}\widetilde{\mathcal{M}},{\rm dw}\widetilde{\mathcal{C}})$ is a dual pair. In the cases where $(\mathcal{M,C})$ is (co)product-closed and perfect, it is easy to see that so is $({\rm dw}\widetilde{\mathcal{M}},{\rm dw}\widetilde{\mathcal{C}})$, since the closure properties asked for ${\rm dw}\widetilde{\mathcal{M}}$ are verified component-wise.

\item[$\bullet$] \underline{$({\rm ex}\widetilde{\mathcal{M}},{\rm ex}\widetilde{\mathcal{C}})$ is a dual pair}: The duality property follows by the corresponding equivalence for $({\rm dw}\widetilde{\mathcal{M}},{\rm dw}\widetilde{\mathcal{C}})$ and by Lemma \ref{lem:exact_Pontrjagin}. In the cases where $(\mathcal{M,C})$ is (co)product-closed and perfect, the rest of the proof follows by the previous case and the facts that the class of exact complexes is closed under (co)products and extensions, and that $D^0(R) \in {\rm ex}\widetilde{\mathcal{M}}$.

\item[$\bullet$] \underline{$(\widetilde{\mathcal{M}},\widetilde{\mathcal{C}})$ is a dual pair}: By Lemma \ref{lem:exact_Pontrjagin}, we know that for any complex $\bm{X}$ in $\Ch(R)$, $\bm{X}$ is exact if, and only if, $\bm{X^+}$ is exact. On the other hand, the functor $\underline{\Hom}(-,D^0(\mathbb{Q / Z}))$ preservers cycles since it is exact, and so, $Z_m(\bm{X}) \in \mathcal{M}$ $\Leftrightarrow$ $Z_m(\bm{X^+}) \in \mathcal{C}$, for any $m \in \mathbb{Z}$. Then, the duality property follows. Now given two chain complexes $\bm{C^1},\bm{C^2} \in \widetilde{\mathcal{C}}$, we can note that $Z_m(\bm{C^1} \oplus \bm{C^2}) \simeq Z_m(\bm{C^1}) \oplus Z_m(\bm{C^2}) \in \mathcal{C}$, for any $m \in \mathbb{Z}$. It follows that $\bm{C^1} \oplus \bm{C^2} \in \widetilde{\mathcal{C}}$ since exact complexes are closed under finite direct sums. On the other hand, if $\bm{C'}$ is a direct summand of $\bm{C} \in \widetilde{\mathcal{C}}$, we have that $Z_m(\bm{C'})$ is a direct summand of $Z_m(\bm{C})$ for any $m \in \mathbb{Z}$. Also, exact complexes are closed under direct summands, and hence $\bm{C'} \in \widetilde{\mathcal{C}}$.

Given any (co)product $\prod_{i \in I} \bm{M^i}$ (resp., $\coprod_{i \in I} \bm{M^i}$) with $\bm{M^i} \in \widetilde{\mathcal{M}}$ for every $i \in I$, we have that $\prod_{i \in I} \bm{M^i}$ (resp., $\coprod_{i \in I} \bm{M^i}$) is exact since $\Mod(R)$ is a Grothendieck category. On the other hand,  $Z_m(\prod_{i \in I} \bm{M^i}) \simeq \prod_{i \in I} Z_m(\bm{M^i}) \in \mathcal{M}$ (resp., $Z_m(\coprod_{i \in I} \bm{M^i}) \simeq \coprod_{i \in I} Z_m(\bm{M^i}) \in \mathcal{M}$) in the case where $(\mathcal{M,C})$ is (co)product-closed. Finally, if $\mathcal{M}$ is closed under extensions with $R \in \mathcal{M}$, it follows that so is $\widetilde{\mathcal{M}}$ with $D^0(R) \in \widetilde{\mathcal{M}}$. Hence, $(\widetilde{\mathcal{M}},\widetilde{\mathcal{C}})$ is perfect in the case where $(\mathcal{M,C})$ is perfect.
\end{itemize}

The rest of the proof is devoted to show the converse statement. If $({\rm dw}\widetilde{\mathcal{M}},{\rm dw}\widetilde{\mathcal{C}})$ is a dual pair, we can prove the duality property for $(\mathcal{M,C})$ considering sphere complexes $S^0(M)$. Since finite direct sums and direct summands in $\Mod(R\op)$ can be thought as finite direct sums and direct summands of sphere complexes in $\Ch(R\op)$ at the same degree, we have that $\mathcal{C}$ is closed under finite direct sums and direct summands, and hence $(\mathcal{M,C})$ is a duality pair. The same argument works to show that $(\mathcal{M,C})$ is (co)product-closed or closed under extensions if so is $({\rm dw}\widetilde{\mathcal{M}},{\rm dw}\widetilde{\mathcal{C}})$. Also, it is clear that $R \in \mathcal{M}$ if $D^0(R) \in {\rm dw}\widetilde{\mathcal{M}}$. Hence, $(\mathcal{M,C})$ is perfect in the case where $({\rm dw}\widetilde{\mathcal{M}},{\rm dw}\widetilde{\mathcal{C}})$ is perfect.

On the other hand, if $({\rm ex}\widetilde{\mathcal{M}},{\rm ex}\widetilde{\mathcal{C}})$ or $(\widetilde{\mathcal{M}},\widetilde{\mathcal{C}})$ are ((co)product-closed or perfect) dual pairs, it suffices to consider disk complexes $D^0(M)$ to show that $(\mathcal{M,C})$ is a ((co)product-closed or perfect) duality pair.
\end{proof}


\subsection{Dual pairs from $\FP_n$-injective and $\FP_n$-flat dimensions of complexes}

The rest of this section will be addressed to apply the methods from Theorem \ref{theo:induced_duality} to obtain covers and pre-envelopes by the classes of complexes with $\FP_n$-injective and $\FP_n$-flat dimension at most $k$. We denote these classes by $\mathscr{I}_{(n,k)}(R)$ and $\mathscr{F}_{(n,k)}(R\op)$, respectively. Note that $\mathscr{I}_{(n,0)}(R) = \mathscr{I}_n(R)$ and $\mathscr{F}_{(n,0)}(R\op) = \mathscr{F}_n(R\op)$, while $\mathscr{I}_{0,k}(R)$ is the class of complexes in $\Ch(R)$ with injective dimension at most $k$, and $\mathscr{F}_{(0,k)}(R\op) = \mathscr{F}_{(1,k)}(R\op)$ is the class of complexes in $\Ch(R\op)$ with flat dimension at most $k$. We begin with the following result, which is a consequence Theorems \ref{theo:duality_pairs} and \ref{theo:induced_duality}, and Proposition \ref{fi-dim}.

\begin{theorem} \label{fi-dual}
The following statements hold true for every $k \geq 0$:
\begin{itemize}
\item[{\rm (a)}] The pairs
\begin{align}
& (\mathscr{F}_{(n,k)}(R\op),\mathscr{I}_{(n,k)}(R)) \label{fi-Ch} \\
& ({\rm dw}\widetilde{\mathcal{F}_{(n,k)}}(R\op),{\rm dw}\widetilde{\mathcal{I}_{(n,k)}}(R)) \label{fi-dw} \\
& ({\rm ex}\widetilde{\mathcal{F}_{(n,k)}}(R\op),{\rm ex}\widetilde{\mathcal{I}_{(n,k)}}(R)) \label{fi-ex}
\end{align}
are perfect dual pairs over $R$ for any $n \geq 0$. Moreover:
\begin{itemize}
\item[{\rm (a.1)}] For any $n \geq 2$, \eqref{fi-Ch}, \eqref{fi-dw} and \eqref{fi-ex} are product-closed.

\item[{\rm (a.2)}] In the case $n = 1$, \eqref{fi-Ch}, \eqref{fi-dw} and \eqref{fi-ex} are product-closed if, and only if, $R$ is right coherent.
\end{itemize}

\item[{\rm (b)}] The pairs
\begin{align}
& (\mathscr{I}_{(n,k)}(R),\mathscr{F}_{(n,k)}(R\op)) \label{if-Ch} \\
& ({\rm dw}\widetilde{\mathcal{I}_{(n,k)}}(R),{\rm dw}\widetilde{\mathcal{F}_{(n,k)}}(R\op)) \label{if-dw} \\
& ({\rm ex}\widetilde{\mathcal{I}_{(n,k)}}(R),{\rm ex}\widetilde{\mathcal{F}_{(n,k)}}(R\op)) \label{if-ex}
\end{align}
are (co)product-closed dual pairs over $R$ for any $n \geq 2$, with $\mathscr{I}_{(n,k)}(R)$, ${\rm dw}\widetilde{\mathcal{I}_{(n,k)}}(R)$ and ${\rm ex}\widetilde{\mathcal{I}_{(n,k)}}(R)$ closed under extensions. Moreover:
\begin{itemize}
\item[{\rm (b.1)}] \eqref{if-Ch}, \eqref{if-dw} and \eqref{if-ex} are perfect if, and only if, $\FP_n\mbox{-}\id_R(R) \leq k$.

\item[{\rm (b.2)}] In the case $n = 1$, \eqref{if-Ch}, \eqref{if-dw} or \eqref{if-ex} are dual pairs over $R$ if, and only if, $R$ is a left coherent ring.

\item[{\rm (b.3)}] In the case $n = 0$, \eqref{if-Ch}, \eqref{if-dw} or \eqref{if-ex} are coproduct-closed dual pairs over $R$ if, and only if, $R$ is a left noetherian ring.
\end{itemize}
\end{itemize}
\end{theorem}

The analogue of \cite[Theorem 3.1]{duality} is also valid in the context of complexes. This is due to Yang's \cite[Theorem 3.2]{Ya}. Namely, given a dual pair $(\mathscr{M,C})$ over $R$, then $\mathscr{M}$ is closed under pure sub-complexes, pure quotients and pure extensions. Furthermore, the following hold:
\begin{itemize}
\item[(a)] If $(\mathscr{M,C})$ is product-closed, then $\mathscr{M}$ is pre-enveloping.

\item[(b)] If $(\mathscr{M,C})$ is coproduct-closed, then $\mathscr{M}$ is covering.

\item[(c)] If $(\mathscr{M,C})$ is perfect, then $(\mathscr{M},\mathscr{M}^\perp)$ is a perfect cotorsion pair.
\end{itemize}
As it occurred with modules, the following result is a consequence of these properties combined with Theorem \ref{fi-dual}.

\begin{corollary} \label{ChCoversPreenvelopes}
The following statements hold true for every $k \geq 0$:
\begin{itemize}
\item[{\rm (a)}] The class $\mathscr{F}_{(n,k)}(R\op)$ is closed under pure sub-complexes, pure quotients, and pure extensions for any $n \geq 0$.

\item[{\rm (b)}] The class $\mathscr{I}_{(n,k)}(R)$ is closed under pure sub-complexes, pure quotients, and pure extensions for any $n \geq 2$.

\item[{\rm (c)}] For any $n \geq 0$, the pairs
\begin{align*}
& (\mathscr{F}_{(n,k)}(R\op),(\mathscr{F}_{(n,k)}(R\op))^\perp) \\
& ({\rm dw}\widetilde{\mathcal{F}_{(n,k)}(R\op)},({\rm dw}\widetilde{\mathcal{F}_{(n,k)}(R\op)})^\perp) \\
& ({\rm ex}\widetilde{\mathcal{F}_{(n,k)}(R\op)},({\rm ex}\widetilde{\mathcal{F}_{(n,k)}(R\op)})^\perp)
\end{align*}
are perfect cotorsion pairs in $\Ch(R\op)$. In particular, every complex of right $R$-modules has a cover by a complex with $\FP_n$-flat dimension at most $k$.

\item[{\rm (d)}] If $R$ is a right coherent ring, then every complex in $\Ch(R\op)$ has a pre-envelope by a complex with flat dimension at most $k$, and every complex in $\Ch(R)$ has a cover by a complex with absolutely pure dimension at most $k$.

\item[{\rm (e)}] If $R$ is a left noetherian ring, then every complex in $\Ch(R)$ has a cover by a complex with injective dimension at most $k$.

\item[{\rm (f)}] For any $n \geq 2$, every complex in $\Ch(R\op)$ has a pre-envelope by a complex with $\FP_n$-flat dimension at most $k$, and every complex in $\Ch(R)$ has a pre-envelope and a cover by a complex with $\FP_n$-injective dimension at most $k$. In the case where $\FP_n\mbox{-}\id_R(R) \leq k$, the following are perfect cotorsion pairs in $\Ch(R)$:
\begin{align*}
& (\mathscr{I}_{(n,k)}(R),(\mathscr{I}_{(n,k)}(R))^\perp) \\
& ({\rm dw}\widetilde{\mathcal{I}_{(n,k)}(R)},({\rm dw}\widetilde{\mathcal{I}_{(n,k)}(R)})^\perp) \\
& ({\rm ex}\widetilde{\mathcal{I}_{(n,k)}(R)},({\rm ex}\widetilde{\mathcal{I}_{(n,k)}(R)})^\perp).
\end{align*}
\end{itemize}
\end{corollary}

As a special case of Corollary \ref{ChCoversPreenvelopes}, we have that $\mathscr{F}_n(R\op)$ is always covering. On the other hand, if $n \geq 2$, then $\mathscr{F}_n(R\op)$ is pre-enveloping, and $\mathscr{I}_n(R)$ is covering and pre-enveloping.

The obtention of monic pre-envelopes and epic covers in $\Ch(R)$ and $\Ch(R\op)$ from the classes $\mathscr{I}_{(n,k)}(R)$ and $\mathscr{F}_{(n,k)}(R\op)$ is surprisingly related to asking a single property to the disk complex $D^0(R)$. We close this section going into the details of this, complementing Corollary \ref{ChCoversPreenvelopes}.

\begin{proposition}
The following statements are equivalent for every $n \geq 2$:
\begin{itemize}
\item[{\rm (1)}]  $D^0(R)$ has $\FP_n$-injective dimension at most $k$.

\item[{\rm (2)}] Every complex in $\Ch(R\op)$ has a monic $\mathscr{F}_{(n,k)}(R\op)$-pre-envelope.

\item[{\rm (3)}] Every complex in $\Ch(R)$ has an epic $\mathscr{I}_{(n,k)}(R)$-cover.

\item[{\rm (4)}] Every injective complex in $\Ch(R\op)$ has $\FP_n$-flat dimension at most $k$.

\item[{\rm (5)}] Every projective complex in $\Ch(R)$ has $\FP_n$-injective dimension at most $k$.

\item[{\rm (6)}] Every flat complex in $\Ch(R)$ has $\FP_n$-injective dimension at most $k$.
\end{itemize}
\end{proposition}

\begin{proof} \
\begin{itemize}
\item[$\bullet$] (1) $\Rightarrow$ (2): Let $\bm{X}$ be a complex in $\Ch(R\op)$. Then there is a $\mathscr{F}_{(n,k)}(R\op)$-pre-envelope $\bm{\varphi} \colon \bm{X} \rightarrow \bm{W}$ by Corollary \ref{ChCoversPreenvelopes}. Now consider an exact sequence
\[
\bm{0} \to  \bm{X} \to \prod_{m \in \mathbb{Z}} (D^m(R))^+.
\]
Since $D^0(R)$ has $\FP_n$-injective dimension at most $k$ by (1), then each $(D^m(R))^+$ has $\FP_n$-flat dimension at most $k$ by Proposition \ref{dim-dual}, and hence $\prod_{m \in \mathbb{Z}} (D^m(R))^+ \in \mathscr{F}_{(n,k)}(R\op)$. Now from the following commutative diagram
\[
\xymatrix@C=0.5cm{
\bm{0} \ar[d] & \\
\bm{X} \ar[d] \ar[r]^{\bm{\varphi}} & \bm{W} \ar@{.>}[ld] \\
\displaystyle\operatorname*{\prod}_{m \in \mathbb{Z}} (D^i(R))^+
}
\]
we can get that the $\mathscr{F}_{(n,k)}(R\op)$-pre-envelope $\bm{\varphi} \colon \bm{X} \rightarrow \bm{W}$ is monic.

\item[$\bullet$] (2) $\Rightarrow$ (4): Let $\bm{E}$ be an injective complex in $\Ch(R\op)$. By (2), there is an exact sequence $\bm{0} \to \bm{E} \to \bm{W} \to \bm{W/E} \to \bm{0}$ with $\bm{W} \in \mathscr{F}_{(n,k)}(R\op)$. Moreover, this sequence is split, and so $\bm{E}$ belongs to $\mathscr{F}_{(n,k)}(R\op)$ as a direct summand of $\bm{W}$.

\item[$\bullet$] (4) $\Rightarrow$ (6): Let $\bm{Q}$ be a flat complex in $\Ch(R)$. Then, $\bm{Q^+}$ is injective by \cite{Fe72}, and hence $\bm{Q^+} \in \mathscr{F}_{(n,k)}(R\op)$ by hypothesis. Finally, by Proposition \ref{dim-dual} we have  $\bm{Q} \in \mathscr{I}_{(n,k)}(R)$.

\item[$\bullet$] (1) $\Rightarrow$ (3): Let $\bm{X}$ be a complex in $\Ch(R)$. Then, there is a $\mathscr{I}_{(n,k)}(R)$-cover $\bm{\psi} \colon \bm{W} \to \bm{X}$ by Corollary \ref{ChCoversPreenvelopes}. Consider an epimorphism $\bm{f} \colon \bm{F} \to  \bm{X}$ with $\bm{F}$ free. Since $D^0(R)$ has $\FP_n$-injective dimension at most $k$ by (1), then so does $\bm{F}$. Hence, there exists a morphism $\bm{g} \colon \bm{F} \to \bm{W}$ such that $\bm{\psi} \circ \bm{g} = \bm{f}$. Since $\bm{f}$ is epic, we can get that $\bm{\psi} \colon \bm{W} \to \bm{X}$ is also epic.

\item[$\bullet$] (3) $\Rightarrow$ (5): Let $\bm{P}$ be a projective complex in $\Ch(R)$. By (3), there is an exact sequence $\bm{0} \to \bm{K} \to \bm{W} \to \bm{P} \to \bm{0}$ with $\bm{W} \in \mathscr{I}_{(n,k)}(R)$. Moreover, this sequence is split, so $\bm{P}$ belongs to $\mathscr{I}_{(n,k)}(R)$ as a direct summand of $\bm{W}$.

\item[$\bullet$] The implications (6) $\Rightarrow$ (5) and (5) $\Rightarrow$ (1) are clear.
\end{itemize}
\end{proof}


\section{\textbf{Model structures from $\FP_n$-injective and $\FP_n$-flat dimensions}}\label{sec:models}

In this last section, we construct abelian model structures on $\Ch(R)$ from the classes $\mathcal{I}_{(n,k)}(R)$ and $\mathcal{F}_{(n,k)}(R\op)$. Recall that a \emph{model structure $\mathfrak{M}$ on a bicomplete category $\mathcal{D}$}, roughly speaking, is formed by three classes of morphisms $\mathfrak{F}_{\rm ib}$, $\mathfrak{C}_{\rm of}$ and $\mathfrak{W}_{\rm eak}$ in $\mathcal{D}$ called fibrations, cofibrations and weak equivalences, respectively, satisfying a series of axioms under which it is possible to do homotopy theory on $\mathcal{D}$. We do not go into the details of the definition of model structure, but we suggest the reader to check \cite{HoveyBook}.

For the purpose of this paper, we are interested in a particular type of model structure on bicomplete abelian categories, known as \emph{abelian}. These model structure were defined by Hovey in \cite[Definition 2.1]{HoveyPaperPro}, as those model structures $\mathfrak{M} = (\mathfrak{C}_{\rm of},\mathfrak{F}_{\rm ib},\mathfrak{W}_{\rm eak})$ such that:
\begin{itemize}
\item[$\bullet$] $f \in \mathfrak{C}_{\rm of}$ if, and only if, it is monic and ${\rm CoKer}(f)$ is a cofibrant object.

\item[$\bullet$] $g \in \mathfrak{F}_{\rm ib}$ if, and only if, it is epic and ${\rm Ker}(g)$ is a fibrant object.
\end{itemize}
Trivial cofibrations (that is, cofibrations that are also weak equivalences) and trivial fibrations have a similar description. The importance of abelian model structures lies in the fact that they are in one-to-one correspondence with certain pairs of cotorsion pairs. Specifically, if we are given three classes of objects $\mathcal{A}$, $\mathcal{B}$ and $\mathcal{W}$ on an abelian category $\mathcal{D}$ such that $(\mathcal{A} \cap \mathcal{W},\mathcal{B})$ and $(\mathcal{A},\mathcal{B} \cap \mathcal{W})$ are complete cotorsion pairs, and such that $\mathcal{W}$ is thick, then there exists a unique abelian model structure on $\mathcal{D}$ such that:
\begin{align*}
\mbox{cofibrations} & = \mbox{monomorphisms with cokernel in $\mathcal{A}$}, \\
\mbox{fibrations} & = \mbox{epimorphisms with kernel in $\mathcal{B}$}, \\
\mbox{trivial cofibrations} & = \mbox{monomorphisms with cokernel in $\mathcal{A} \cap \mathcal{W}$}, \\
\mbox{trivial fibrations} & = \mbox{epimorphisms with kernel in $\mathcal{B} \cap \mathcal{W}$}.
\end{align*}
Conversely, for any abelian model structure $(\mathfrak{C}_{\rm of},\mathfrak{F}_{\rm ib},\mathfrak{W}_{\rm eak})$ on a bicomplete abelian category $\mathcal{D}$ one has that the classes $\mathcal{Q}$, $\mathcal{R}$ and $\mathcal{T}$ of cofibrant, fibrant and trivial objects, respectively, form two complete cotorsion pairs $(\mathcal{Q} \cap \mathcal{T},\mathcal{R})$ and $(\mathcal{Q},\mathcal{R} \cap \mathcal{T})$ with $\mathcal{T}$ thick. This result is known as \emph{Hovey's Correspondence}, proved by Hovey in \cite[Theorem 2.2]{HoveyPaper}, and which has turned out to be a useful method to transporting tools from algebraic topology to homological algebra.

Any two cotorsion pairs of the form $(\mathcal{A} \cap \mathcal{W},\mathcal{B})$ and $(\mathcal{A}, \mathcal{B} \cap \mathcal{W})$ are said to be \emph{compatible}. If in addition, these pairs are complete and $\mathcal{W}$ is thick, the triple $(\mathcal{A,W,B})$ is called \emph{Hovey triple}. We will denote the abelian model structure associated to a Hovey triple $(\mathcal{A,W,B})$ by
\[
\mathfrak{M} := (\mathcal{A,W,B}).
\]
In the next section, we explain how to apply Hovey's Correspondence to the context of this paper, along with some results of Gillespie to produce cotorsion pairs of complexes from cotorsion pairs of modules. All the abelian model structures constructed on $\Ch(R)$ from now on will have $\mathcal{W}$ as the class $\mathcal{E}(R)$ of exact complexes, which we know is thick, and so their classes of weak equivalences will be given by the quasi-isomorphisms.


\subsection{Construction of model structures via Hovey correspondence}\label{sec:model_sts}

We know by Corollary \ref{coro:more_properties} that, for any $n \geq 0$, the class $\mathcal{F}_{(n,k)}(R\op)$ of modules with $\FP_n$-flat dimension $\leq k$ is the left half of a perfect cotorsion pair $(\mathcal{F}_{(n,k)}(R\op),(\mathcal{F}_{(n,k)}(R\op))^\perp)$. As it happened with the case $k = 0$, this pair has also a cogenerating set, and one can notice this using the arguments from \cite[Theorem 2.9]{EnochsKaplansky}. This implies by \cite[Theorem 7.3.2]{EJ11} and \cite[Propositions 3.2, 3.3 and 4.3, and Theorem 5.5]{GillespieDegree} that we have the following complete cotorsion pairs in $\Ch(R\op)$:

\begin{align}
& (\widetilde{\mathcal{F}_{(n,k)}(R\op)},{\rm dg}\widetilde{(\mathcal{F}_{(n,k)}(R\op))^\perp}), \label{pair1} \\
& ({\rm dg}\widetilde{\mathcal{F}_{(n,k)}(R\op)},\widetilde{(\mathcal{F}_{(n,k)}(R\op))^\perp}), \label{pair2} \\
& ({\rm dw}\widetilde{\mathcal{F}_{(n,k)}(R\op)},({\rm dw}\widetilde{\mathcal{F}_{(n,k)}(R\op)})^\perp), \label{pair3} \\
& ({\rm ex}\widetilde{\mathcal{F}_{(n,k)}(R\op)},({\rm ex}\widetilde{\mathcal{F}_{(n,k)}(R\op)})^\perp), \label{pair4}
\end{align}
where \eqref{pair1}, \eqref{pair3} and \eqref{pair4} are also perfect by Corollary \ref{ChCoversPreenvelopes}, and $\widetilde{\mathcal{F}_{(n,k)}(R\op)}$ is the class of complexes with $\FP_n$-flat dimension at most $k$ by Proposition \ref{fi-dim}. The symbol ``${\rm dg}$'' stands for ``differential graded''. Recall from \cite{GillespieFlat} that if $(\mathcal{A,B})$ is a cotorsion pair in $\Mod(R)$, then
\begin{align*}
{\rm dg}\widetilde{A} & := \left\{ \begin{array}{ll} \bm{X} \in \Ch(R) \mbox{ :} & \mbox{$X_m \in \mathcal{A}$ for every $m \in \mathbb{Z}$, and $\hom(\bm{X},\bm{B})$} \\ {} & \mbox{is exact whenever $\bm{B}$ is a complex in $\widetilde{\mathcal{B}}$} \end{array} \right\} \\
{\rm dg}\widetilde{B} & := \left\{ \begin{array}{ll} \bm{Y} \in \Ch(R) \mbox{ :} & \mbox{$Y_m \in \mathcal{B}$ for every $m \in \mathbb{Z}$, and $\hom(\bm{A},\bm{Y})$} \\ {} & \mbox{is exact whenever $\bm{A}$ is a complex in $\widetilde{\mathcal{A}}$} \end{array} \right\}.
\end{align*}
As an example, if $\mathcal{P}(R\op)$ denotes the class of projective modules in $\Mod(R\op)$, then the triple $({\rm dg}\widetilde{\mathcal{P}(R\op)},\mathcal{E}(R\op),\Ch(R\op))$ is a Hovey triple in $\Ch(R\op)$. The associated model structure is known as the \emph{standard} or \emph{projective model structure} on $\Ch(R\op)$, which we will denote by $\mathfrak{M}^{\rm proj}(R\op)$. Dually, $(\Ch(R),\mathcal{E}(R),{\rm dg}\widetilde{\mathcal{I}_{(0,0)}(R)})$ is also a Hovey triple in $\Ch(R)$, and the associated model structure is known as the \emph{injective model structure} on $\Ch(R)$. See \cite[Section 2.3]{HoveyBook} for details.

We first study the possibility of obtaining model structures from the pairs \eqref{pair1} and \eqref{pair2}. In the cases $n = 0,1$, we know that $\mathcal{F}_{(n,k)}(R\op)$ is the class of modules with flat dimension at most $k$, and so the inducing cotorsion pair $(\mathcal{F}_{(n,k)}(R\op),(\mathcal{F}_{(n,k)}(R\op))^\perp)$ is hereditary, that is, the class $\mathcal{F}_{(n,k)}(R\op)$ is resolving (that is, it is closed under extensions and epi-kernels, and contains $\mathcal{P}(R\op)$). By \cite[Theorem 3.12]{GillespieFlat}, we know that if $(\mathcal{A,B})$ is a hereditary cotorsion pair in $\Mod(R\op)$ cogenerated by a set, then $\widetilde{\mathcal{A}} = {\rm dg}\widetilde{\mathcal{A}} \cap \mathcal{E}(R\op)$ and $\widetilde{\mathcal{B}} = {\rm dg}\widetilde{\mathcal{B}} \cap \mathcal{E}(R\op)$. It follows that, if $n = 0,1$, then $({\rm dg}\widetilde{\mathcal{F}_{(n,k)}(R\op)}, \mathcal{E}, {\rm dg}\widetilde{(\mathcal{F}_{(n,k)}(R\op))^\perp})$ is a Hovey triple, and so it gives rise to abelian model structures on $\Ch(R\op)$, which are the \emph{$k$-flat model structures} obtained by the second author in \cite[Theorem 6.1]{Perez16}.

Now consider $n \to \infty$. In this case, $\mathcal{F}_{(\infty,k)}(R\op)$ coincides with the class of modules with level dimension $\leq k$, and it is clear that it contains $\mathcal{P}(R\op)$ and that it is closed under extensions. On the other hand, for the case $k = 0$, it is known by \cite[Proposition 2.8]{BGH14} that the class of level modules is also closed under epi-kernels. This closure property is also true for any $k > 0$.

\begin{proposition}\label{prop:level_resolving}
Let $k \geq 0$ be a non-negative integer. Then, the class $\mathcal{F}_{(\infty,k)}(R\op)$ of modules with level dimension at most $k$ is resolving.
\end{proposition}

\begin{proof}
It is only left to show that $\mathcal{F}_{(\infty,k)}(R\op)$ is closed under epi-kernels if $k > 0$. So suppose we are given an exact sequence $0 \to A \to B \to C \to 0$ in $\Mod(R\op)$ with $B, C \in \mathcal{F}_{(\infty,k)}(R\op)$. For any $L \in \mathcal{FP}_\infty(R)$, we have an exact sequence $\Tor^R_{k+2}(C,L) \to \Tor^R_{k+1}(A,L) \to \Tor^R_{k+1}(B,L)$ where $\Tor^R_{k+1}(B,L) = 0$, $\Tor^R_{k+2}(C,L) \cong \Tor^R_{k+1}(C,L') = 0$, and $L' \in \mathcal{FP}_\infty(R)$ appearing in an exact sequence $0 \to  L' \to F \to L \to 0$ with $F$ finitely generated and free. It follows $\Tor^R_{k+1}(A,L) = 0$, and hence $\FP_\infty\mbox{-}\fd_{R\op}(A) \leq k$.
\end{proof}

Thus, being $(\mathcal{F}_{(\infty,k)}(R\op),(\mathcal{F}_{(\infty,k)}(R\op))^\perp)$ a hereditary cotorsion pair cogenerated by a set, the equalities
\[
\widetilde{\mathcal{F}_{(\infty,k)}(R\op)} = {\rm dg}\widetilde{\mathcal{F}_{(\infty,k)}(R\op)} \cap \mathcal{E}(R\op) \mbox{ \ and \ } \widetilde{(\mathcal{F}_{(\infty,k)}(R\op))^\perp} = {\rm dg}\widetilde{(\mathcal{F}_{(\infty,k)}(R\op))^\perp} \cap \mathcal{E}(R\op)
\]
hold, where $\widetilde{\mathcal{F}_{(\infty,k)}(R\op)}$ is the class of complexes with level dimension at most $k$, and so we have the following result by Hovey's Correspondence and \cite[Lemma 6.7]{HoveyPaper}.

\begin{theorem}\label{theo:level_model_structure}
Let $R$ be an arbitrary ring and $k$ be a non-negative integer. There exists a unique cofibrantly generated abelian model structure on $\Ch(R\op)$ given by
\[
\mathfrak{M}^{\rm flat}_{(\infty,k)}(R\op) := ({\rm dg}\widetilde{\mathcal{F}_{(\infty,k)}(R\op)},\mathcal{E}(R\op),{\rm dg}\widetilde{(\mathcal{F}_{(\infty,k)}(R\op))^\perp}).
\]
In the case $k = 0$, we will refer to $\mathfrak{M}^{\rm flat}_{(\infty,0)}(R\op)$ as the \emph{level model structure}.
\end{theorem}

A model structure is, roughly speaking, \emph{cofibrantly generated} if its classes of cofibrations and trivial cofibrations can be generated via \emph{transfinite compositions} from sets of morphisms, called \emph{generating cofibrations} and \emph{genereating trivial cofibrations}. We do not recall specifically the definition of a cofibrantly generated model structures, as it involves several thick abstract notions, but we refer the interested reader to \cite[Section 2.1]{HoveyBook}. However, if we work in the context of abelian model structures, cofibrantly generated model structures can be thought as the analogous of a cotorsion pair cogenerated by a set.

The flat model structure constructed by Gillespie in \cite{GillespieFlat} has the additional property that it is monoidal, with respect to the closed symmetric monoidal structure on $\Ch(R\op)$ (where $R$ is commutative) given by the usual tensor product $\otimes$. Roughly speaking, a model structure on a closed symmetric monoidal category is \emph{monoidal} if it is compatible with the monoidal structure. Checking that a model structure is monoidal involves some lengthy conditions (See \cite[Definition 4.2.6]{HoveyBook}). However, in the case of abelian model structures and thanks to Hovey's \cite[Theorem 7.2]{HoveyPaper}, we have a list of simpler conditions to check. Reading this result for the (closed symmetric) monoidal structure $(\Ch(R\op),\otimes)$, we have that an abelian model structure on $\Ch(R\op)$ is monoidal if:
\begin{itemize}
\item[(a)] Every cofibration is a pure injection in each degree.

\item[(b)] If $\bm{X}$ and $\bm{Y}$ are cofibrant objects, then so is $\bm{X} \otimes \bm{Y}$.

\item[(c)] If $\bm{X}$ and $\bm{Y}$ are cofibrant objects and any of them is trivial, then $\bm{X} \otimes \bm{Y}$ is trivial.

\item[(d)] The unit $S^0(R)$ of the monoidal category $(\Ch(R\op),\otimes)$ is cofibrant.
\end{itemize}
Level modules represent a relative version of flat modules from which one can obtain a model structure on $\Ch(R\op)$, namely $\mathfrak{M}^{\rm flat}_{(\infty,0)}(R\op)$. However, $\mathfrak{M}^{\rm flat}_{(\infty,0)}(R\op)$ does not share the property of being monoidal that its flat sibling $\mathcal{M}^{\rm flat}_{(0,0)}(R\op)$ does have. This is settled in the following result.

\begin{proposition}\label{prop:monoidal}
Let $R$ be a commutative ring. The level model structure $\mathfrak{M}^{\rm flat}_{(\infty,0)}(R\op)$ on $\Ch(R\op)$ is monoidal if, and only if, $R$ is coherent.
\end{proposition}

\begin{proof}
By \cite[Corollary 2.9]{BGH14}, $R$ is left coherent if, and only if, the classes of (right) level modules and flat modules coincide. So, if $R$ is right coherent, $\mathfrak{M}^{\rm flat}_{(\infty,0)}(R\op)$ is precisely the flat model structure, which is monoidal by \cite[Corollary 5.1]{GillespieFlat}.

Now suppose that $\mathfrak{M}^{\rm flat}_{(\infty,0)}(R\op)$ coincides with the flat model structure. Given a level module $M$ in $\Mod(R\op)$, we have by \cite[Lemma 3.4]{GillespieFlat} that $\bm{0} \to S^0(M)$ is a cofibration, and so $0 \to M$ is a pure injection, implying that $M$ must be flat. Hence, we can conclude that $R$ is left coherent.
\end{proof}

So far, with respect to the pairs \eqref{pair1} and \eqref{pair2}, we have only worked out the limit cases $n = 0,1$ and $n \to \infty$. For the cases in between, we cannot even obtain a model structure on $\Ch(R\op)$ from the inducing cotorsion pair $(\mathcal{F}_{(n,k)}(R\op),(\mathcal{F}_{(n,k)}(R\op))^\perp)$, since the cotorsion pairs \eqref{pair1} and \eqref{pair2} in $\Ch(R\op)$ are not necessarily compatible, that is, we cannot always guarantee that the equalities
\[
\widetilde{\mathcal{F}_{(n,k)}(R\op)} = {\rm dg}\widetilde{\mathcal{F}_{(n,k)}(R\op)} \cap \mathcal{E}(R\op) \mbox{ \ and \ } \widetilde{(\mathcal{F}_{(n,k)}(R\op))^\perp} = {\rm dg}\widetilde{(\mathcal{F}_{(n,k)}(R\op))^\perp} \cap \mathcal{E}(R\op)
\]
hold. Actually, this is only possible in the case where the ground ring $R$ is left $n$-coherent, due to the following result.

\begin{proposition}\label{prop:no_flat}
The following are equivalent for any ring $R$ and any $n \geq 2$:
\begin{itemize}
\item[{\rm (1)}] $R$ is left $n$-coherent.

\item[{\rm (2)}] $\widetilde{\mathcal{F}_{(n,k)}(R\op)} = {\rm dg}\widetilde{\mathcal{F}_{(n,k)}(R\op)} \cap \mathcal{E}(R\op)$.

\item[{\rm (3)}] $\widetilde{(\mathcal{F}_{(n,k)}(R\op))^\perp} = {\rm dg}\widetilde{(\mathcal{F}_{(n,k)}(R\op))^\perp} \cap \mathcal{E}(R\op)$.
\end{itemize}
\end{proposition}

\begin{proof}
The equivalence (2) $\Leftrightarrow$ (3) is a consequence of \cite[Corollary 3.13]{GillespieFlat}. On the other hand, by \cite[Theorem 5.6]{BP16} we know that $R$ is left $n$-coherent if, and only if, the cotorsion pair $(\mathcal{F}_n(R\op),(\mathcal{F}_n(R\op))^\perp)$ is hereditary. On the other hand, we can note that $\mathcal{F}_n(R\op)$ is resolving if, and only if, so is $\mathcal{F}_{(n,k)}(R\op)$ for any $k \geq 0$. For, note that $\mathcal{F}_{(n,k)}(R\op)$ is always closed under extensions and contains $\mathcal{P}(R\op)$. Now suppose that $\mathcal{F}_n(R\op)$ is closed under epi-kernels and that we are given a short exact sequence $0 \to A \to B \to C \to 0$ with $B, C \in \mathcal{F}_{(n,k)}(R\op)$. For any $L \in \mathcal{FP}_n(R)$, we have an exact sequence $\Tor^R_{k+2}(C,L) \to \Tor^R_{k+1}(A,L) \to \Tor^R_{k+1}(B,L)$ where $\Tor^R_{k+1}(B,L) = 0$ and $\Tor^R_{k+2}(C,L) \cong \Tor^R_{1}(C',L)$, and where $C'$ is a projective $(k+1)$-st syzygy of $C$. On the other hand, consider an exact sequence $0 \to C' \to P \to C'' \to 0$ with $P$ projective and $C''$ a projective $k$-th syzygy of $C$. Since $\FP_n\mbox{-}\fd_{R\op}(C) \leq k$, we have that $C'' \in \mathcal{F}_n(R\op)$. Then, it follows that $C' \in \mathcal{F}_n(R\op)$ since we are assuming $\mathcal{F}_n(R\op)$ closed under epi-kernels. Thus, we get $\Tor^R_{k+2}(C,L) \cong \Tor^R_{1}(C',L) = 0$, and so $\Tor^R_{k+1}(A,L) = 0$, that is, $A \in \mathcal{F}_{(n,k)}(R\op)$. Therefore, (1) $\Leftrightarrow$ (2) follows by \cite[Corollary 3.13]{GillespieFlat}.
\end{proof}

From the previous result, we have that there are no abelian model structures on $\Ch(R\op)$ associated to $\mathcal{F}_{(n,k)}(R\op)$ for the cases $1 < n < \infty$, unless in the case $R$ is left $n$-coherent where the model structures are those in Theorem \ref{theo:level_model_structure}. One good aspect about the pairs \eqref{pair3} and \eqref{pair4} is that we are going to have abelian model structures for any choice of $n$ and without imposing extra conditions on $R$. For the cases $n = 0, 1$, these model structures were called \emph{degree-wise $k$-flat model structures} by the second author in \cite[Theorem 6.2]{Perez16}. One important result from the previous reference is that it provides sufficient conditions to obtain a Hovey triple from \eqref{pair3} and \eqref{pair4}. On the one hand, it is clear by definition that
\[
{\rm ex}\widetilde{\mathcal{F}_{(n,k)}(R\op)} = {\rm dw}\widetilde{\mathcal{F}_{(n,k)}(R\op)} \cap \mathcal{E}(R\op).
\]
On the other hand, by \cite[Proposition 5.6 (i)]{Perez16} it is known that if the inducing cotorsion pair $(\mathcal{A,B})$ in $\Mod(R\op)$ is such that $({\rm dw}\widetilde{\mathcal{A}},({\rm dw}\widetilde{\mathcal{A}})^\perp)$ is complete, then
\[
({\rm dw}\widetilde{\mathcal{A}})^\perp = ({\rm ex}\widetilde{\mathcal{A}})^\perp \cap \mathcal{E}(R\op).
\]
Since \eqref{pair3} is complete, we have a Hovey triple
\[
({\rm dw}\widetilde{\mathcal{F}_{(n,k)}(R\op)},\mathcal{E}(R\op),({\rm ex}\widetilde{\mathcal{F}_{(n,k)}(R\op)})^\perp)
\]
and so the following result is a consequence of Hovey's Correspondence, \cite[Lemma 6.7]{HoveyPaper} and \cite[Theorem 7.2.14]{EJ11}.

\begin{theorem}
For any ring $R$ and $n, k \geq 0$, there exists a unique cofibrantly generated abelian model structure on $\Ch(R\op)$, given by
\[
\mathfrak{M}^{\rm dw\mbox{-}flat}_{(n,k)}(R\op) := ({\rm dw}\widetilde{\mathcal{F}_{(n,k)}(R\op)},\mathcal{E}(R\op),({\rm ex}\widetilde{\mathcal{F}_{(n,k)}(R\op)})^\perp).
\]
For the case $n \to \infty$ and $k = 0$, the model structure $\mathfrak{M}^{\rm dw\mbox{-}flat}_{(\infty,0)}(R\op)$ will be referred as the \emph{degree-wise level model structure on $\Ch(R\op)$}.
\end{theorem}

We know that the monoidality of the level model structure is equivalent to the coherency of the ring $R$. The same phenomenon occurs for the degree-wise level model structure, if we impose an extra condition on $R$.

\begin{proposition}
Let $R$ be a commutative ring with weak dimension at most $1$. The degree-wise level model structure $\mathfrak{M}^{\rm dw\mbox{-}flat}_{(\infty,0)}(R\op)$ is monoidal if, and only if, $R$ is coherent.
\end{proposition}

\begin{proof}
Suppose $R$ is coherent. Then, the class of level modules coincides with the class of flat modules. So $\mathfrak{M}^{\rm dw\mbox{-}flat}_{(\infty,0)}(R\op)$ is the degree-wise flat model structure, which is monoidal by \cite[Proposition 6.11]{Perez16}. The remaining implication follows as in the proof of Proposition \ref{prop:monoidal}.
\end{proof}

For the rest of this section, we study dual process of constructing model structures from the class of modules with bounded $\FP_n$-injective dimension. We know by \cite[Corollary 4.2]{BP16} that $\mathcal{I}_{(n,0)}(R)$ is the right half of a cotorsion pair $({}^\perp(\mathcal{I}_{(n,0)}(R)),\mathcal{I}_{(n,0)}(R))$ cogenerated by a set, for any $n \geq 0$. This fact will help us to prove the following result.

\begin{theorem} \label{theo:pair_inj}
For any ring $R$ and $n, k \geq 0$, $({}^\perp(\mathcal{I}_{(n,k)}(R)),\mathcal{I}_{(n,k)}(R))$ is a cotorsion pair in $\Mod(R)$ cogenerated by a set.
\end{theorem}

\begin{proof}
The pair $({}^\perp(\mathcal{I}_{(n,0)}(R)),\mathcal{I}_{(n,0)}(R))$ is cogenerated by a set $\mathcal{S}$ of representatives of modules in $\mathcal{FP}_n(R)$. Let $\mathcal{S}_k$ be a set of representatives of $k$-th projective syzygies of modules in $\mathcal{S}$. Note that $\mathcal{I}_{(n,k)}(R) = (\mathcal{S}_k)^\perp$. In fact, if $N \in \mathcal{I}_{(n,k)}(R)$ and $M \in \mathcal{S}_k$, we have that $\Ext^1_R(M,N) \cong \Ext^{k+1}_R(S,N)$, for some $S \in \mathcal{S}$. Since $\Ext^{k+1}_R(S,N) = 0$, it follows that $\mathcal{I}_{(n,k)}(R) \subseteq (\mathcal{S}_k)^\perp$. Now suppose that $N \in (\mathcal{S}_k)^\perp$ and let $L \in \mathcal{FP}_n(R)$. Since $L \in {}^\perp(\mathcal{I}_{(n,0)}(R))$ and every module in ${}^\perp(\mathcal{I}_{(n,0)}(R))$ is a direct summand of a module filtered by $\mathcal{S}$ (see \cite[Corollary 3.2.4]{Gobel}), there exists a module $L'$ in $\Mod(R)$ and an ordinal number $\lambda$ such that $L$ is a direct summand of $L'$ and $L' = \bigcup_{\alpha < \lambda} L'_\alpha$ where $L'_0 \in \mathcal{S}$ and $L'_{\alpha + 1} / L'_\alpha \in \mathcal{S}$ for any $\alpha + 1 < \lambda$. Thus, we have:
\begin{align*}
\Ext^{k+1}_R(L'_0,N) & \cong \Ext^1_R(\Omega^k(L'_0),N) = 0, \\
\Ext^{k+1}_R\left( \frac{L'_{\alpha + 1}}{L_\alpha}, N \right) & \cong \Ext^1_R\left( \Omega^{k}\left( \frac{L'_{\alpha+1}}{L'_\alpha} \right), N \right) = 0 \mbox{ for any } \alpha + 1 < \lambda.
\end{align*}
Eklof's Lemma \cite[Theorem 7.3.4]{EJ00} implies that $\Ext^{k+1}_R(L',N) = 0$, and so $\Ext^{k+1}_R(L,N) = 0$ for any $L \in \mathcal{FP}_n(R)$, that is, $N \in \mathcal{I}_{(n,k)}(R)$. Therefore, $\mathcal{I}_{(n,k)}(R) = (\mathcal{S}_k)^\perp$ and the pair $({}^\perp(\mathcal{I}_{(n,k)}(R)),\mathcal{I}_{(n,k)}(R))$ is a cotorsion pair in $\Mod(R)$ cogenerated by $\mathcal{S}_k$.
\end{proof}

The previous result, along with \cite[Propositions 4.3, 4.4 and 4.6]{GillespieDegree} and \cite[Theorem 7.3.2]{EJ11}, implies that we have the following cotorsion pairs in $\Ch(R)$ cogenerated by sets (and so complete):
\begin{align}
& ({\rm dg}\widetilde{{}^\perp(\mathcal{I}_{(n,k)}(R))},\widetilde{\mathcal{I}_{(n,k)}(R)}), \label{pair5} \\
& (\widetilde{{}^\perp(\mathcal{I}_{(n,k)}(R))},{\rm dg}\widetilde{\mathcal{I}_{(n,k)}(R)}), \label{pair6} \\
& ({}^\perp({\rm dw}\widetilde{\mathcal{I}_{(n,k)}(R)}),{\rm dw}\widetilde{\mathcal{I}_{(n,k)}(R)}), \label{pair7} \\
& ({}^\perp({\rm ex}\widetilde{\mathcal{I}_{(n,k)}(R)}),{\rm ex}\widetilde{\mathcal{I}_{(n,k)}(R)}), \label{pair8}
\end{align}
where $\widetilde{\mathcal{I}_{(n,k)}(R)}$ is by Proposition \ref{fi-dim} the class of complexes with $\FP_n$-injective dimension at most $k$. With respect to the pairs \eqref{pair7} and \eqref{pair8}, we are going to have by \cite[Proposition 5.6 (ii)]{Perez16} the equality:
\[
{}^\perp({\rm dw}\widetilde{\mathcal{I}_{(n,k)}(R)}) = {}^\perp({\rm ex}\widetilde{\mathcal{I}_{(n,k)}(R)}) \cap \mathcal{E}(R).
\]
It follows that $({}^\perp({\rm ex}\widetilde{\mathcal{I}_{(n,k)}(R)}),\mathcal{E}(R),{\rm dw}\widetilde{\mathcal{I}_{(n,k)}(R)})$ is a Hovey triple, and the following result is a consequence of Hovey's correspondence.

\begin{theorem}
Let $R$ be a ring and $n, k \geq 0$. Then, there exists a unique cofibrantly generated abelian model structure on $\Ch(R)$ given by
\[
\mathfrak{M}^{\rm dw\mbox{-}inj}_{(n,k)}(R) := ({}^\perp({\rm ex}\widetilde{\mathcal{I}_{(n,k)}(R)}),\mathcal{E}(R),{\rm dw}\widetilde{\mathcal{I}_{(n,k)}(R)}).
\]
\end{theorem}

The previous theorem is a generalization of the \emph{degree-wise $k$-injective model structures} found by the second author in \cite[Theorem 5.11]{Perez16}.

The pairs \eqref{pair5} and \eqref{pair6} have the same problem that their flat counterpart: they are not necessarily compatible since the inducing cotorsion pair $({}^\perp(\mathcal{I}_{(n,k)}(R)),\mathcal{I}_{(n,k)}(R))$ is not hereditary in general. The following result follows as Proposition \ref{prop:no_flat}, using \cite[Corollary 3.13]{GillespieFlat} and \cite[Theorem 5.5]{BP16}.

\begin{proposition} \label{prop:no_injective}
The following conditions are equivalent for any ring $R$ and $n \geq 1$.
\begin{itemize}
\item[{\rm(1)}] $R$ is left $n$-coherent.

\item[{\rm(2)}] $\widetilde{\mathcal{I}_{(n,k)}(R)} = {\rm dg}\widetilde{\mathcal{I}_{(n,k)}(R)} \cap \mathcal{E}(R)$.

\item[{\rm(3)}] $\widetilde{{}^\perp(\mathcal{I}_{(n,k)}(R))} = {\rm dg}\widetilde{{}^\perp(\mathcal{I}_{(n,k)}(R))} \cap \mathcal{E}(R)$.
\end{itemize}
\end{proposition}

However, in the case where $n \to \infty$, the class $\mathcal{I}_{(\infty,0)}(R)$ of absolutely clean modules is the right half of a hereditary cotorsion pair $({}^\perp(\mathcal{I}_{(\infty,0)}(R)),\mathcal{I}_{(\infty,0)}(R))$ cogenerated by a set (See \cite[Corollary 4.2 and Theorem 5.5]{BP16}), and as in Proposition \ref{prop:level_resolving}, we can use the fact that $\mathcal{I}_{(\infty,0)}(R)$ is coresolving to prove the following result.

\begin{proposition}\label{prop:absolutely_clean_coresolving}
For any $k \geq 0$, the class $\mathcal{I}_{(\infty,k)}(R)$ of modules with absolutely clean dimension at most $k$ is coresolving.
\end{proposition}

Theorem \ref{theo:pair_inj} is also valid in the case $n \to \infty$. It follows that $({}^\perp(\mathcal{I}_{(\infty,k)}(R)),\mathcal{I}_{(\infty,k)}(R))$ is a hereditary cotorsion pair cogenerated by a set, and hence we have the following model structure on $\Ch(R)$ from the Hovey triple $({\rm dg}\widetilde{{}^\perp(\mathcal{I}_{(\infty,k)}(R))},\mathcal{E}(R),{\rm dg}\widetilde{\mathcal{I}_{(\infty,k)}(R)})$, which is a relativization of the \emph{$k$-injective model structures} \cite[Theorem 4.9]{Perez16}.

\begin{theorem}\label{theo:model_ac}
Let $R$ be any ring and $k \geq 0$. Then, there exists a unique cofibrantly generated abelian model structure on $\Ch(R)$ given by:
\[
\mathfrak{M}^{\rm inj}_{(\infty,k)}(R) := ({\rm dg}\widetilde{{}^\perp(\mathcal{I}_{(\infty,k)}(R))},\mathcal{E}(R),{\rm dg}\widetilde{\mathcal{I}_{(\infty,k)}(R)})
\]
This model structure coincides with the $k$-injective model structure if, and only if, $R$ is left noetherian.
\end{theorem}

The last assertion in the previous theorem is a consequence of \cite[Theorem 3.17]{GillespieModels}.

It is only left to work with the case where $n \geq 2$ and $\FP_n\mbox{-}\id_R(R) \leq k$, in which $\mathcal{I}_{(n,k)}(R)$ is the left half of a perfect cotorsion pair $(\mathcal{I}_{(n,k)}(R),(\mathcal{I}_{(n,k)}(R))^\perp)$ cogenerated by a set. Then, we have the following cotorsion pairs in $\Ch(R)$ cogenerated by sets:
\begin{align}
& (\widetilde{\mathcal{I}_{(n,k)}(R)},{\rm dg}\widetilde{(\mathcal{I}_{(n,k)}(R))^\perp}) \label{pair9} \\
& ({\rm dg}\overline{\mathcal{I}_{(n,k)}(R)},\widetilde{(\mathcal{I}_{(n,k)}(R))^\perp}) \label{pair10} \\
& ({\rm dw}\widetilde{\mathcal{I}_{(n,k)}(R)},({\rm dw}\widetilde{\mathcal{I}_{(n,k)}(R)})^\perp) \label{pair11} \\
& ({\rm ex}\widetilde{\mathcal{I}_{(n,k)}(R)},({\rm ex}\widetilde{\mathcal{I}_{(n,k)}(R)})^\perp) \label{pair12}
\end{align}
where \eqref{pair9}, \eqref{pair11} and \eqref{pair12} are perfect by Corollary \ref{ChCoversPreenvelopes}. We are not aware if there are conditions under which the pairs \eqref{pair9} and \eqref{pair10}\footnote{We have used the notation ${\rm dg}\overline{\mathcal{I}_{(n,k)}(R)}$ to avoid confusion with the class ${\rm dg}\widetilde{\mathcal{I}_{(n,k)}(R)}$ associated to the cotorsion pair $({}^\perp(\mathcal{I}_{(n,k)}(R)),\mathcal{I}_{(n,k)}(R))$.} are compatible (For instance, injective modules are not resolving in general). However, if we consider the pairs \eqref{pair11} and \eqref{pair12}, we have the compatibility relations
\[
{\rm ex}\widetilde{\mathcal{I}_{(n,k)}(R)} = {\rm dw}\widetilde{\mathcal{I}_{(n,k)}(R)} \cap \mathcal{E}(R) \mbox{ \ and \ } ({\rm dw}\widetilde{\mathcal{I}_{(n,k)}(R)})^\perp = ({\rm ex}\widetilde{\mathcal{I}_{(n,k)}(R)})^\perp \cap \mathcal{E}(R),
\]
and so we obtain the following result.

\begin{theorem}
Let $n \geq 2$ and $k \geq 0$. For any ring $R$ with $\FP_n\mbox{-}\id_R(R) \leq k$, there exists a unique cofibrantly generated abelian model structure on $\Ch(R)$ given by:
\[
\mathfrak{M}^{\rm op\mbox{ }dw\mbox{-}inj}_{(n,k)}(R) := ({\rm dw}\widetilde{\mathcal{I}_{(n,k)}(R)},\mathcal{E}(R),({\rm ex}\widetilde{\mathcal{I}_{(n,k)}(R)})^\perp).
\]
\end{theorem}

Most of the model structures constructed so far may be related with each other via the notion of Quillen equivalence. We explore this point at the end of this paper, but before that, we need some preliminaries on Pontrjagin duality.


\subsection{The Pontrjagin dual of differential graded complexes}

In Section \ref{sec:InjFlatCom}, we showed how to construct from a duality pair $(\mathcal{M,C})$ over $R$, three different dual pairs over $R$, namely, $(\widetilde{\mathcal{M}},\widetilde{\mathcal{C}})$, $({\rm dw}\widetilde{\mathcal{M}},{\rm dw}\widetilde{\mathcal{C}})$ and $({\rm ex}\widetilde{\mathcal{M}},{\rm ex}\widetilde{\mathcal{C}})$. The only classes of induced complexes we did not consider were those of differential graded complexes. The problem is that we cannot even define ${\rm dg}\widetilde{\mathcal{M}}$ and ${\rm dg}\widetilde{\mathcal{C}}$, as we need $\mathcal{M}$ and $\mathcal{C}$ to be halves of cotorsion pairs. We can assume that $(\mathcal{M},\mathcal{M}^\perp)$ is a cotorsion pair in $\Mod(R)$, and that $({}^\perp\mathcal{C},\mathcal{C})$ is a cotorsion pair in $\Mod(R\op)$, but even in this case, in which we can define ${\rm dg}\widetilde{\mathcal{M}}$ and ${\rm dg}\widetilde{\mathcal{C}}$, we are not aware if $({\rm dg}\widetilde{\mathcal{M}},{\rm dg}\widetilde{\mathcal{C}})$ is a dual pair. However, we can show that the Pontrjagin duality $(-)^+ \colon \Ch(R) \longrightarrow \Ch(R\op)$ maps any complex in ${\rm dg}\widetilde{\mathcal{M}}$ to a complex in ${\rm dg}\widetilde{\mathcal{C}}$. We settle this in the following results.

\begin{lemma}\label{lem:Pontrjagin_perp}
If $(\mathcal{M,C})$ is a duality pair over $R$, then $N^+ \in \mathcal{M}^\perp$ for any $N \in {}^\perp\mathcal{C}$.
\end{lemma}

\begin{proof}
Suppose $N \in {}^\perp\mathcal{C}$, and let $M \in \mathcal{M}$. Then, $\Ext^1_R(M,N^+) \cong \Tor^R_1(N,M)^+ \cong \Ext^1_R(N,M^+)$ where $M^+ \in \mathcal{C}$ since $(\mathcal{M,C})$ is a duality pair. Hence, $\Ext^1_R(N,M^+) = 0$.
\end{proof}

\begin{proposition} \label{prop:duality_dg}
Let $(\mathcal{M,C})$ be a duality pair over $R$ such that $(\mathcal{M},\mathcal{M}^\perp)$ is a cotorsion pair in $\Mod(R)$, and $({}^\perp\mathcal{C},\mathcal{C})$ is a cotorsion pair in $\Mod(R\op)$. Then, for any complex $\bm{X}$ in ${\rm dg}\widetilde{\mathcal{M}}$, one has $\bm{X^+} \in {\rm dg}\widetilde{\mathcal{C}}$.
\end{proposition}

\begin{proof}
Let $\bm{X}$ be a complex in ${\rm dg}\widetilde{\mathcal{M}}$. Then, $X_m \in \mathcal{M}$ for any $m \in \mathbb{Z}$, and $\hom(\bm{X},\bm{Y})$ is an exact complex whenever $\bm{Y} \in \widetilde{\mathcal{M}^\perp}$. We first note that $(\bm{X}^+)_m \in \mathcal{C}$ by \eqref{eqn:alter_Pontrjagin}, for any $m \in \mathbb{Z}$. It is only left to show that $\hom(\bm{K},\bm{X^+})$ is exact whenever $\bm{K} \in \widetilde{{}^\perp\mathcal{C}}$. On the one hand, we have that $H_n(\hom(\bm{K},\bm{X^+})) \cong \Ext^1_{\rm dw}(\bm{K}[n+1],\bm{X^+})$. On the other hand, since $\bm{K}[n+1]_m \in {}^\perp\mathcal{C}$ and $(\bm{X}^+)_m \in \mathcal{C}$, it follows that $\Ext^1_{\rm dw}(\bm{K}[n+1],\bm{X^+}) = \Ext^1_{\Ch}(\bm{K}[n+1],\bm{X^+})$. Proving that $\Ext^1_{\Ch}(\bm{K}[n+1],\bm{X^+}) = 0$ is equivalent to showing that $\underline{\Ext}^1(\bm{K}[n+1],\bm{X^+}) = \bm{0}$. By \cite[Lemma 5.4.2 b)]{GR99}, we have $\underline{\Ext}^1(\bm{K}[n+1],\bm{X^+}) \cong \overline{\Tor}_1(\bm{K}[n+1],\bm{X})^+ \cong \underline{\Ext}^1(\bm{X},(\bm{K}[n+1])^+)$. Note that $\bm{K} \in \widetilde{{}^\perp\mathcal{C}}$ implies $\bm{K}[n+1] \in \widetilde{{}^\perp\mathcal{C}}$. On the other hand, recall that $(-)^+$ preserves kernels, and so $Z_m((\bm{K}[n+1])^+) \simeq (Z_m(\bm{K}[n+1]))^+ \in \mathcal{M}^\perp$ by Lemma \ref{lem:Pontrjagin_perp}. It follows that $(\bm{K}[n+1])^+ \in \widetilde{\mathcal{M}^\perp}$. Noticing that $X_m \in \mathcal{C}$ and $(\bm{K}[n+1])^+_m \in \mathcal{M}^\perp$, we have that $\Ext^1_{\Ch}(\bm{X},(\bm{K}[n+1])^+) = 0$, since it is isomorphic to $H_n(\hom(\bm{X},(\bm{K}[n+1])^+))$ and $\hom(\bm{X},(\bm{K}[n+1])^+)$ is exact. Hence, we have $\underline{\Ext}^1(\bm{X},(\bm{K}[n+1])^+) = \bm{0}$, and so $\Ext^1_{\Ch}(\bm{K}[n+1],\bm{X^+}) = 0$ for every $n \in \mathbb{Z}$, that is, $\hom(\bm{K},\bm{X^+})$ is an exact complex. Therefore, $\bm{X^+} \in {\rm dg}\widetilde{C}$.
\end{proof}


\subsection{Relation between $\FP_n$-injective and $\FP_n$-flat model structures}

We close this paper comparing the different model structures we have obtained so far. The most common way to compare two model structures is via Quillen adjunctions, which are the morphisms between model structures. Indeed, it is known by \cite{HoveyBook} that the classes of model categories, Quillen adjunctions and natural transformations form a $2$-category. Let us give a brief review of these morphisms.

Given two model categories $(\mathcal{D}_1,\mathfrak{M}_1)$ and $(\mathcal{D}_2,\mathfrak{M}_2)$, a \emph{left Quillen functor} is a left adjoint functor $F \colon \mathcal{D}_1 \longrightarrow \mathcal{D}_2$ which preserves cofibrations and trivial cofibrations. The notion of \emph{right Quillen functor} $G \colon \mathcal{D}_2 \longrightarrow \mathcal{D}_1$ is dual. Finally, a \emph{Quillen adjunction} is given by a pair $(F,G)$ such that $F$ is a left adjoint of $G$ that is a left Quillen functor, or equivalently, if $G$ is a right adjoint of $F$ that is a right Quillen functor. A Quillen adjunction $(F,G)$ is called a \emph{Quillen equivalence} if for any cofibrant object $X$ in $\mathfrak{M}_1$ and any fibrant object $Y$ in $\mathfrak{M}_2$, a morphism $f \colon F(X) \to Y$ is a weak equivalence in $\mathfrak{M}_2$ if, and only if, $\varphi(f) \colon X \to G(Y)$ is a weak equivalence in $\mathfrak{M}_1$, where $\varphi$ is the natural isomorphism $\Hom_{\mathcal{D}_2}(F(-),-) \Rightarrow \Hom_{\mathcal{D}_1}(-,G(-))$ in the adjunction $(F,G)$. The model categories $(\mathcal{D}_1,\mathfrak{M}_1)$ and $(\mathcal{D}_2,\mathfrak{M}_2)$ are said to be \emph{Quillen equivalent} if there is a Quillen equivalence between them. The reader can see these definitions in detail in \cite[Definitions 1.3.1 and 1.3.12 and Lemma 1.3.4]{HoveyBook}. As in \cite{DS}, we will say that $(\mathcal{D}_1,\mathfrak{M}_1)$ and $(\mathcal{D}_2,\mathfrak{M}_2)$ are \emph{$\ast$Quillen equivalent} if they are connected by a zig-zag of Quillen equivalences between model categories. We will write
\[
\mathfrak{M}_1 \underset{q}{\sim} \mathfrak{M}_2
\]
whenever $(\mathcal{D}_1,\mathfrak{M}_1)$ and $(\mathcal{D}_2,\mathfrak{M}_2)$ are Quillen equivalent, and
\[
\mathfrak{M}_1 \overset{\ast}{\underset{q}{\sim}} \mathfrak{M}_2
\]
whenever $(\mathcal{D}_1,\mathfrak{M}_1)$ and $(\mathcal{D}_2,\mathfrak{M}_2)$ are $\ast$Quillen equivalent.

Before establishing a comparison between the model structures in Section \ref{sec:model_sts} via Quillen equivalences, we comment some properties of the Pontrjagin duality functor in the context of abelian model structures. Namely, we show how (trivial) cofibrations of certain model structures on $\Ch(R\op)$ in Section \ref{sec:model_sts} can be determined by the (trivial) fibrations of other model structures on $\Ch(R)$. The following result is a consequence of Theorems \ref{theo:duality_pairs} and \ref{theo:induced_duality}.

\begin{proposition}
The following conditions hold true:
\begin{itemize}
\item[{\rm (a)}] For every $n, k \geq 0$, a morphism $\bm{f}$ in $\Ch(R\op)$ is a (trivial) cofibration in $\mathfrak{M}^{\rm dw\mbox{-}flat}_{(n,k)}(R\op)$ if, and only if, $\bm{f}^+$ is a (trivial) fibration in $\mathfrak{M}^{\rm dw\mbox{-}inj}_{(n,k)}(R)$. This includes the case $n \to \infty$.

\item[{\rm (b)}] For every $n \geq 2$ and $k \geq 0$, a morphism $\bm{g}$ in $\Ch(R)$ is a (trivial) fibration in $\mathfrak{M}^{\rm dw\mbox{-}inj}_{(n,k)}(R)$ if, and only if, $\bm{g}^+$ is a (trivial) cofibration in $\mathfrak{M}^{\rm dw\mbox{-}flat}_{(n,k)}(R\op)$. This includes the case $n \to \infty$.
\end{itemize}
\end{proposition}

The Pontrjagin duality functor $(-)^+ \colon \Ch(R) \longrightarrow \Ch(R\op)$ has other preservation properties due to Lemma \ref{lem:Pontrjagin_perp}, Theorem \ref{theo:induced_duality} and Proposition \ref{prop:duality_dg}.

\begin{proposition}
The following conditions hold:
\begin{itemize}
\item[{\rm (a)}] $(-)^+$ maps (trivial) cofibrations in $\mathfrak{M}^{\rm dw\mbox{-}inj}_{(n,k)}(R)$ to (trivial) fibrations in $\mathfrak{M}^{\rm dw\mbox{-}flat}_{(n,k)}(R\op)$.

\item[{\rm (b)}] $(-)^+$ maps (trivial) cofibrations in $\mathfrak{M}^{\rm flat}_{(\infty,k)}(R\op)$ to (trivial) fibrations in $\mathfrak{M}^{\rm inj}_{(\infty,k)}(R)$.
\end{itemize}
\end{proposition}

\begin{proof}
We only prove that $(-)^+$ maps cofibrations in $\mathfrak{M}^{\rm dw\mbox{-}inj}_{(n,k)}(R)$ to fibrations in $\mathfrak{M}^{\rm dw\mbox{-}flat}_{(n,k)}(R\op)$. So let $\bm{f} \colon \bm{X} \to \bm{Y}$ be a cofibration in $\mathfrak{M}^{\rm dw\mbox{-}inj}_{(n,k)}(R)$, that is, a monomorphism with cokernel $\bm{K} \in {}^\perp({\rm ex}\widetilde{\mathcal{I}_{(n,k)}(R)})$. Then, we have that $\bm{f}^+$ is an epimorphism. It remains to show that $\bm{K}^+ \in ({\rm ex}\widetilde{\mathcal{F}_{(n,k)}(R\op)})^\perp$. According to \cite[Proposition 3.3]{GillespieDegree}, this holds true if $(K_m)^+ \in (\mathcal{F}_{(n,k)}(R\op))^\perp$ and if the complex $\hom(\bm{W},\bm{K}^+)$ is exact whenever $\bm{W} \in {\rm ex}\widetilde{\mathcal{F}_{(n,k)}(R\op)}$. The former condition follows by Lemma \ref{lem:Pontrjagin_perp}. For the latter, using arguments similar to those in the proof of Proposition \ref{prop:duality_dg}, it suffices to verify that $\underline{\Ext}^1(\bm{W}[m+1],\bm{K}^+) = \bm{0}$ for every $m \in \mathbb{Z}$. This follows by the fact that $(\bm{W}[m+1])^+ \in {\rm ex}\widetilde{\mathcal{I}_{(n,k)}(R)}$ by Theorem \ref{fi-dual}, and by the natural isomorphisms $\underline{\Ext}^1(\bm{W}[m+1],\bm{K}^+) \cong \overline{\Tor}_1(\bm{W}[m+1],\bm{K})^+ \cong \underline{\Ext}^1(\bm{K},(\bm{W}[m+1])^+) = \bm{0}$.
\end{proof}

Although the functor $(-)^+$ has some nice properties when it comes to relating model structures, it fails to be a Quillen equivalence (or even a left or right Quillen functor). In what remains of this section, we will study some conditions under which it is possible to establish Quillen equivalences between the model structures associated to $\FP_n$-injective modules, and those associated to $\FP_n$-flat modules. The functors to be studied here as candidates for Quillen functors will be the identity functor ${\rm id}\colon \Ch(R) \to \Ch(R)$ and the functor $\overline{R} \otimes_R - \colon \Ch(R) \to \Ch(\overline{R})$ related to the \emph{change of ring} construction induced by a ring homomorphism $\varphi \colon R \to \overline{R}$. The contents presented below are motivated by Hovey's work \cite{HoveyBook} on projective and injective model structures.

In the last part of \cite[Section 2.3]{HoveyBook}, it is claimed that the identity $\id \colon \Ch(R) \longrightarrow \Ch(R)$ is a Quillen equivalence between the standard and injective model structures on $\Ch(R)$. This functor is going to be a source for several Quillen equivalences between the model structures in Section \ref{sec:model_sts}. We can start to specify this claim by noticing that $\id$ maps (trivial) cofibrations in $\mathfrak{M}^{\rm inj}_{(\infty,0)}(R)$ into (trivial) cofibrations in $\mathfrak{M}^{\rm inj}_{(0,0)}(R)$. It follows that
\begin{align}\label{equiv1}
\mathfrak{M}^{\rm inj}_{(\infty,0)}(R) & \underset{q}{\sim} \mathfrak{M}^{\rm inj}_{(0,0)}(R).
\end{align}
This equivalence also holds for any injective dimension, that is,
\begin{align}\label{equiv2}
\mathfrak{M}^{\rm inj}_{(\infty,k)}(R) & \underset{q}{\sim} \mathfrak{M}^{\rm inj}_{(0,k)}(R)
\end{align}
for any $k \geq 0$. Recall by Theorem \ref{theo:model_ac} that the previous equivalence becomes an equality if, and only if, $R$ is a left noetherian ring. We can extend the previous equivalence to a $\ast$Quillen equivalence between model structures associated to $\FP_n$-injective modules when we vary the finiteness parameter ``n''. If $n \geq m \geq 0$, then every module in $\Mod(R)$ of type $\FP_n$ is of type $\FP_m$. It follows that $\mathcal{I}_{(m,k)}(R) \subseteq \mathcal{I}_{(n,k)}(R)$ for any $k \geq 0$. On the other hand, if we want to vary the dimension parameter by assuming that $k \geq t \geq 0$, then $\mathcal{I}_{(n,k)}(R) \subseteq \mathcal{I}_{(n,t)}(R)$. From these inclusions, we have that:
\begin{align}
\mathfrak{M}^{\rm dw\mbox{-}inj}_{(n,k)}(R) & \overset{\ast}{\underset{q}{\sim}} \mathfrak{M}^{\rm dw\mbox{-}inj}_{(m,t)}(R), \label{equiv3} \\
\mathfrak{M}^{\rm op\mbox{ }dw\mbox{-}inj}_{(n,k)}(R) & \overset{\ast}{\underset{q}{\sim}} \mathfrak{M}^{\rm op\mbox{ }dw\mbox{-}inj}_{(m,t)}(R). \label{equiv4}
\end{align}
The $\ast$Quillen equivalence \eqref{equiv3} becomes an equality if, and only if, $R$ is left $m$-coherent and $k = t$. In a similar way, we have that
\begin{align}\label{equiv5}
\mathfrak{M}^{\rm dw\mbox{-}flat}_{(n,k)}(R\op) \overset{\ast}{\underset{q}{\sim}} \mathfrak{M}^{\rm dw\mbox{-}flat}_{(m,t)}(R\op)
\end{align}
for every $n,m \geq 0$ and $k,t \geq 0$.

Now we compare the absolutely clean and level model structures on $\Ch(R)$ and $\Ch(R\op)$. We study some conditions under which these two model structures are $\ast$Quillen equivalent. One way to do this is comparing the homotopy categories of each of the model structures. Recall from \cite{DS} that the \emph{homotopy category} of a model category $(\mathcal{D},\mathfrak{M})$, denoted by ${\rm Ho}(\mathcal{D})$ is obtained by formally inverting the weak equivalences to obtain the category-theoretic localization $\mathfrak{W}_{\rm eak}^{-1}\mathcal{D}$. If we choose any of the model structures on $\Ch(R)$ obtained in this paper, its homotopy category is equivalent to the derived category $\bm{{\rm D}}(R)$ of the ring $R$, since its class of weak equivalences is given by the quasi-isomorphisms. So one may think of comparing abelian model categories on $\Ch(R)$ and $\Ch(\overline{R})$ by checking if the rings $R$ and $\overline{R}$ are derived equivalent. This is related to a non trivial result due to D. Dugger and B. Shipley \cite[Theorem 2.6]{DS}. Specifically, they proved that two rings $R$ and $\overline{R}$ are derived equivalent if, and only if, their associated standard model structures on $\Ch(R)$ and $\Ch(\overline{R})$ are $\ast$Quillen equivalent. We will use this result as a way to compare model structures on $\Ch(R)$ and $\Ch(R\op)$.

Consider the absolutely clean model structure $\mathfrak{M}^{\rm inj}_{(\infty,0)}(R)$ on $\Ch(R)$ and the level model structure $\mathfrak{M}^{\rm flat}_{(\infty,0)}(R\op)$ on $\Ch(R\op)$. On the one hand, we already know that
\[
\mathfrak{M}^{\rm inj}_{(\infty,0)}(R) \underset{q}{\sim} \mathfrak{M}^{\rm inj}_{(0,0)}(R) \mbox{ \ and \ } \mathfrak{M}^{\rm inj}_{(0,0)}(R) \underset{q}{\sim} \mathfrak{M}^{\rm proj}(R).
\]
Thus,
\[
\mathfrak{M}^{\rm inj}_{(\infty,0)}(R) \overset{\ast}{\underset{q}{\sim}} \mathfrak{M}^{\rm proj}(R).
\]
On the other hand, it is easy to note that every dg-projective complex in $\Ch(R\op)$ is dg-level (that is, ${\rm dg}\widetilde{\mathcal{P}(R\op)} \subseteq {\rm dg}\widetilde{\mathcal{F}_{(\infty,0)}(R\op)}$), and so $\id$ maps (trivial) cofibrations in $\mathfrak{M}^{\rm proj}(R\op)$ to (trivial) cofibrations in $\mathfrak{M}^{\rm flat}_{(\infty,0)}(R\op)$. It follows that
\[
\mathfrak{M}^{\rm flat}_{(\infty,0)}(R\op) \underset{q}{\sim} \mathfrak{M}^{\rm proj}(R\op).
\]
By \cite[Theorem 2.6]{DS}, we conclude the following result.

\begin{proposition}\label{prop:ac_level_eq}
For any ring $R$, the following conditions are equivalent:
\begin{itemize}
\item[{\rm (1)}] $R$ and $R\op$ are derived equivalent.

\item[{\rm (2)}] The absolutely clean model structure on $\Ch(R)$ is $\ast$Quillen equivalent to the level model structure on $\Ch(R\op)$.
\end{itemize}
\end{proposition}
In a similar way, we can conclude that $R$ and $R\op$ are derived equivalent if, and only if, the injective model structure on $\Ch(R)$ and the flat model structure on $\Ch(R\op)$ are $\ast$Quillen equivalent. And more generally, combining Proposition \ref{prop:ac_level_eq} with \eqref{equiv3} and \eqref{equiv5}, along with the $\ast$Quillen equivalences
\[
\mathfrak{M}^{\rm dw\mbox{-}flat}_{(n,k)}(R\op) \overset{\ast}{\underset{q}{\sim}} \mathfrak{M}^{\rm proj}(R\op) \mbox{ \ and \ } \mathfrak{M}^{\rm dw\mbox{-}inj}_{(n,k)}(R) \overset{\ast}{\underset{q}{\sim}} \mathfrak{M}^{\rm proj}(R),
\]
we obtain the following result.

\begin{proposition}
For any ring $R$, the following conditions are equivalent for any finiteness parameters $n, m \geq 0$ and any dimension parameters $k, t \geq 0$:
\begin{itemize}
\item[{\rm (1)}] $R$ and $R\op$ are derived equivalent.

\item[{\rm (2)}] $\mathfrak{M}^{\rm dw\mbox{-}inj}_{(n,k)}(R) \overset{\ast}{\underset{q}{\sim}} \mathfrak{M}^{\rm dw\mbox{-}flat}_{(m,t)}(R\op)$.
\end{itemize}
\end{proposition}

In general, a ring $R$ is not derived equivalent to its opposite. There are cases where, however, such an equivalence occurs. For instance, if $R$ is commutative, then $R = R\op$. In the case where $R$ is not commutative, we cannot even assert that $R$ and $R\op$ are Morita equivalent. Rings which are Morita equivalent to its opposite were characterized by U. A. First in \cite{First}. For these rings, we have that $\Mod(R)$ and $\Mod(R\op)$ are (category-theoretic) equivalent. It follows that the derived categories $\bm{{\rm D}}(R)$ and $\bm{{\rm D}}(R\op)$ of $\Mod(R)$ and $\Mod(R\op)$ are equivalent, that is, $R$ and $R\op$ are derived equivalent. So in this case, the absolutely clean model structure on $\Ch(R)$ and the level model structure on $\Ch(R\op)$ are $\ast$Quillen equivalent, and the corresponding homotopy categories are triangle equivalent. This is one of the cases where a triangle equivalence between homotopy categories comes from a Quillen equivalence, although this is not true in general (See \cite[Theorem 8.5.23]{Hir}). Notice that the homotopy categories considered here are triangulated since their model structures are pointed (See \cite[Chapter 7]{HoveyBook}).

Although the absolutely clean and level model structures are not always $\ast$Quillen equivalent when compared between $\Ch(R)$ and $\Ch(R\op)$, they will be indeed if they are considered on the same category, say $\Ch(R)$. We can complement the equivalence \eqref{equiv1} with the fact that the standard and Gillespie's flat model structure \cite{GillespieFlat} are Quillen equivalent, and hence
\begin{align}
\mathfrak{M}^{\rm inj}_{(\infty,0)}(R) \overset{\ast}{\underset{q}{\sim}} \mathfrak{M}^{\rm flat}_{(\infty,0)}(R).
\end{align}
Moreover, by the previous equivalence, along with \eqref{equiv3} and \eqref{equiv5}, we have
\begin{align}
\mathfrak{M}^{\rm dw\mbox{-}inj}_{(n,k)}(R) \overset{\ast}{\underset{q}{\sim}} \mathfrak{M}^{\rm dw\mbox{-}flat}_{(m,t)}(R)
\end{align}
for every $n,m \geq 0$ and $k,t \geq 0$.

In what remains of this paper, we study the possibility that a ring homomorphism $\varphi \colon R \to \overline{R}$ induces, in the form of the change of ring functor, a Quillen adjunction between the $\FP_n$-injective and the $\FP_n$-flat model structures of Section \ref{sec:model_sts}. Any ring homomorphism $\varphi$ induces an adjoint pair $(\overline{R} \otimes_R -, U) \colon \Ch(R) \to \Ch(\overline{R})$. On the one hand, for any left $R$-module $M$, one has that $\overline{R} \otimes_R M$ is a left $\overline{R}$-module. On the other hand, any left $\overline{R}$-module $N$ can be given a left $R$-module structure via $\varphi$ as follows: $r \cdot y = \varphi(r) \cdot y$ for every $r \in R$ and $y \in N$. We denote by $U(N)$ the left $\overline{R}$-module $N$ thought as a left $R$-module. These two constructions yield functors $\overline{R} \otimes_R - \colon \Ch(R) \longrightarrow \Ch(\overline{R})$ and $U \colon \Ch(\overline{R}) \longrightarrow \Ch(R)$, which form an adjoint pair $(\overline{R} \otimes_R -, U)$. Similarly, we also get an adjunction $(\overline{R} \otimes_R -, U) \colon \Mod(R) \longleftrightarrow \Mod(\overline{R})$, which we denote the same way by abuse of notation. The left adjoint is known as the \emph{change of ring} or the \emph{induction} functor, while the right adjoint is known as the \emph{restriction} or the \emph{forgetful functor}. According to \cite[Section 2.3]{HoveyBook}, the induction is a left Quillen adjunction between the standard model structures on $\Ch(R)$ and $\Ch(\overline{R})$, which turns out to be a Quillen equivalence if, and only if, $\varphi$ is an isomorphism. Note that this result cannot be applied if we set $\overline{R} := R\op$, since $R$ and $R\op$ are not necessarily isomorphic (See \cite[Section 2.8]{Jacobson1} for a counter-example).

The fact that $(\overline{R} \otimes_R -, U)$ is a Quillen adjunction between the standard model structures is only claimed but not proved in \cite{HoveyBook}, but it is important that we prove it by ourselves in order to study $\overline{R} \otimes_R -$ and $U$ as left and right Quillen functors between the model structures of Section \ref{sec:model_sts}. We also extend Hovey's assertions to Gillespie's flat model structures.

\begin{lemma}\label{lem:proj_flat_Quillen}
The induction $- \otimes_R \overline{R} \colon \Ch(R\op) \to \Ch(\overline{R}\op)$ is a left Quillen functor from the standard model structure on $\Ch(R\op)$ to the standard model structure on $\Ch(\overline{R}\op)$. It is also a left Quillen functor from the flat model structure on $\Ch(R\op)$ to the flat model structure on $\Ch(\overline{R}\op)$. In both cases, it is a Quillen equivalence if, and only if, $\varphi$ is an isomorphism.
\end{lemma}

\begin{proof}
Suppose that $\bm{f} \colon \bm{X} \to \bm{Y}$ is a cofibration in $\mathfrak{M}^{\rm proj}(R\op)$, that is, a monomorphism with cokernel $\bm{K}$ dg-projective over $R$. Since each $K_m$ is projective, and so flat, we have that each $f_m \otimes_R \overline{R}$ is a monomorphism, and so $\bm{f} \otimes_R \overline{R}$ is a monomorphism in $\Ch(\overline{R}\op)$. We show that $\bm{K} \otimes_R \overline{R}$ is dg-projective over $\overline{R}$. For any exact complex $\bm{E}$ in $\Ch(\overline{R}\op)$, we have by \cite[Proposition 4.4.11]{PerezBook} a natural isomorphism $\hom(\bm{K} \otimes_R \overline{R},\bm{E}) = \hom(\bm{K} \otimes S^0(\overline{R}), \bm{E}) \cong \hom(\bm{K},\hom(S^0(\overline{R}),\bm{E}))$ where $\hom(S^0(\overline{R}),\bm{E})$ is exact as a complex in $\Ch(R)$ since $S^0(\overline{R})$ is dg-projective over $\overline{R}$, and so the resulting complex $\hom(\bm{K},\hom(S^0(\overline{R}),\bm{E}))$ is exact since $S^0(\overline{R})$ is dg-projective over $R$. Then, $\hom(\bm{K} \otimes_R \overline{R},\bm{E})$ is exact. On the other hand, note that each $K_m \otimes_R \overline{R}$ is a projective module in $\Mod(\overline{R}\op)$, due to the natural isomorphism $\Hom_{\overline{R}\op}(K_m \otimes_R \overline{R},-) \cong \Hom_{R\op}(K_m,U(-))$ and to the fact that the forgetful functor $U \colon \Ch(\overline{R}\op) \longrightarrow \Ch(R\op)$ is exact by \cite[Proposition 8.33]{Rotman}. Hence, we conclude that $\overline{R} \otimes_R \bm{K}$ is dg-projective over $\overline{R}$. One can also check that projective complexes in $\Ch(R\op)$ remain exact after tensoring with $\overline{R}$. It follows that $- \otimes_R \overline{R}$ is a left Quillen functor.

Now suppose that $\bm{f}$ as above is a cofibration in the flat model structure on $\Ch(R\op)$. Then, one can note that $\hom(\bm{K} \otimes_R \overline{R},\bm{E})$ is exact whenever $\bm{E}$ is a cotorsion complex in $\Ch(\overline{R}\op)$, that is, $\bm{E}$ is exact with cycles in $(\mathcal{F}_{(0,0)}(\overline{R}\op))^\perp$. On the other hand, each $K_m \otimes_R \overline{R}$ is flat in $\Mod(\overline{R}\op)$. Since $K_m$ is flat in $\Mod(R\op)$, by Lazard's Theorem we can write $K_m \simeq \varinjlim K^i_m$ where each $K^i_m$ is projective, that is, $K_m$ is a direct limit of projective modules in $\Mod(R\op)$. Now using the fact that $- \otimes_R \overline{R}$ preserves direct limits, we have $K_m \otimes_R \overline{R} \simeq \varinjlim K^i_m \otimes_R \overline{R}$, where each $K^i_m \otimes_R \overline{R}$ is projective in $\Mod(\overline{R}\op)$, and hence, $K_m \otimes_R \overline{R}$ is flat in $\Mod(\overline{R}\op)$. Hence, $\bm{f} \otimes_R \overline{R}$ is a cofibration in $\mathfrak{M}^{\rm flat}_{(0,0)}(\overline{R}\op)$. Also, $- \otimes_R \overline{R}$ preserves the exactness of exact complexes with flat cycles, and hence $- \otimes_R \overline{R}$ maps trivial cofibrations in $\mathfrak{M}^{\rm flat}_{(0,0)}(R\op)$ to trivial cofibrations in $\mathfrak{M}^{\rm flat}_{(0,0)}(\overline{R}\op)$.
\end{proof}

The arguments applied in the previous lemma cannot apply to the model structures involving the class $\mathcal{F}_{(n,k)}(R\op)$ with $n > 1$. Specifically, we do not have a version of Lazard's Theorem for $\FP_n$-flat modules. However, we can settle this inconvenience by imposing some extra conditions on $R$ and $\overline{R}$. We first study the preservation of modules of type $\FP_n$ under $\overline{R} \otimes_R -$ and $U$. This will have to do with a particular type of flat modules. Recall that a left $R$-module $M$ is \emph{faithfully flat} if for every sequence $\bm{\eta} \colon 0 \to A \to B \to C \to 0$ in $\Mod(R\op)$, one has that $\bm{\eta}$ is exact if, and only if, $\bm{\eta} \otimes_R M$ is exact.

\begin{proposition}\label{prop:change_FPn}
Let $\varphi \colon R \to \overline{R}$ be a ring homomorphism. The following conditions hold:
\begin{itemize}
\item[{\rm (a)}] If $\varphi$ makes $\overline{R}$ a faithfully flat right $R$-module, then the conditions $M \in \mathcal{FP}_n(R)$ and $\overline{R} \otimes_R M \in \mathcal{FP}_n(\overline{R})$ are equivalent.

\item[{\rm (b)}] If $\varphi$ makes $\overline{R}$ a finitely generated projective left $R$-module, then $U(N) \in \mathcal{FP}_n(R)$  whenever $N \in \mathcal{FP}_n(\overline{R})$. If in addition $\varphi$ is an isomorphism and $\varphi$ makes $\overline{R}$ a faithfully flat right $R$-module, then $N \in \mathcal{FP}_n(\overline{R})$ whenever $U(N) \in \mathcal{FP}_n(R)$.
\end{itemize}
\end{proposition}

\begin{proof} \
\begin{itemize}
\item[(a)] The cases $n = 0, 1$ follow by \cite[Theorem 2.1.9]{Glaz}. Now let $M \in \mathcal{FP}_n(R)$, and suppose that the result is true for every module in $\mathcal{FP}_{n-1}(R)$. We have a short exact sequence $\bm{\eta} \colon 0 \to M' \to F \to M \to 0$ in $\Mod(R)$ such that $F$ is finitely generated and free and $M' \in \mathcal{FP}_{n-1}(R)$. Since $\overline{R}$ is a faithfully flat module in $\Mod(R\op)$, we have that $\overline{R} \otimes \bm{\eta} \colon 0 \to \overline{R} \otimes_R M' \to \overline{R} \otimes_R F \to \overline{R} \otimes_R M \to 0$ is a short exact sequence in $\Mod(\overline{R})$, where $\overline{R} \otimes_R M' \in \mathcal{FP}_{n-1}(\overline{R})$. On the other hand, we can write $F \simeq R^{(I)}$, where $I$ is a finite set and $R^{(I)}$ is a coproduct of copies of $R$ indexed by $I$. Then, $\overline{R} \otimes_R F \simeq (\overline{R} \otimes_R R)^{(I)} \simeq \overline{R}^{(I)}$, that is, $\overline{R} \otimes_R F$ is a finitely generated free module in $\Mod(\overline{R})$. It follows that $\overline{R} \otimes_R M \in \mathcal{FP}_n(\overline{R})$.

Now suppose that $\overline{R} \otimes_R M \in \mathcal{FP}_n(\overline{R})$. Then, we know by the case $n = 0$ that $M$ is finitely generated. Then, we can consider a short exact sequence as $\bm{\eta}$ with $F$ finitely generated and free in $\Mod(R)$. Then, we obtain a short exact sequence $\overline{R} \otimes_R \bm{\eta}$, where $\overline{R} \otimes_R F$ is finitely generated and free in $\Mod(\overline{R})$ and $\overline{R} \otimes_R M \in \mathcal{FP}_n(\overline{R})$. It follows by \cite[Theorem 2.1.2]{Glaz} that $\overline{R} \otimes_R M' \in \mathcal{FP}_{n-1}(\overline{R})$, and by the induction hypothesis, we conclude that $M' \in \mathcal{FP}_{n-1}(R)$. Hence, $M \in \mathcal{FP}_n(R)$.

\item[(b)] Let $N \in \mathcal{FP}_n(\overline{R})$. First of all, since $N$ is a finitely generated module in $\Mod(\overline{R})$, we have an epimorphism $h \colon \overline{R}^{(J)} \to N$ where $J$ is a finite set. Since the forgetful functor $U$ preserves epimorphisms and finite direct sums in $\Mod(\overline{R})$, we have an epimorphism $U(h) \colon U(\overline{R})^{(J)} \to U(N)$ in $\Mod(R)$, where each $U(\overline{R})$ is a non-zero finitely generated projective left $R$-module, and thus so is $U(\overline{R})^{(J)}$. It follows that $U(N)$ is finitely generated. In the same way, one can show that $U(N) \in \mathcal{FP}_n(R)$.

Now suppose that $\varphi$ is an isomorphism and that $U(N) \in \mathcal{FP}_n(R)$. On the one hand, the adjoint pair $(\overline{R} \otimes_R -, U) \colon \Mod(R) \longleftrightarrow \Mod(\overline{R})$ is in this case an adjoint equivalence, and so the counit $\varepsilon \colon \overline{R} \otimes_R U(-) \Rightarrow \id_{\Mod(\overline{R})}$ is a natural isomorphism. Thus, we have $\overline{R} \otimes_R U(N) \simeq N$. By part (a), we have $N \in \mathcal{FP}_n(\overline{R})$.
\end{itemize}
\end{proof}

\begin{proposition}\label{prop:change_FPn-flat}
Let $\varphi \colon R \to \overline{R}$ be a ring homomorphism and $M$ be a right $R$-module.
\begin{itemize}
\item[{\rm (a)}] If $\varphi$ makes $\overline{R}$ a finitely generated projective module in $\Mod(R)$, then $M \otimes_R \overline{R} \in \mathcal{F}_{(n,k)}(\overline{R}\op)$ whenever $M \in \mathcal{F}_{(n,k)}(R\op)$.

\item[{\rm (b)}] If $R$ and $\overline{R}$ are commutative, and $\varphi$ makes $\overline{R}$ a (left and right) faithfully flat $R$-module, then $M \in \mathcal{F}_{(n,k)}(R\op)$ whenever $M \otimes_R \overline{R} \in \mathcal{F}_{(n,k)}(\overline{R}\op)$.
\end{itemize}
\end{proposition}

\begin{proof}
For part (a), let $M \in \mathcal{F}_{(n,k)}(R\op)$ and $L \in \mathcal{FP}_n(\overline{R})$. By \cite[Corollary 10.61]{Rotman}, we have $\Tor^{\overline{R}}_{k+1}(M \otimes_R \overline{R},L) \cong \Tor^R_{k+1}(M, U(\overline{R} \otimes_{\overline{R}} L)) \cong \Tor^R_{k+1}(M,U(L))$. And by Proposition \ref{prop:change_FPn}, we have that $U(L) \in \mathcal{FP}_n(R)$, and so $\Tor^R_{k+1}(M,U(L)) = 0$. It follows that $\Tor^{\overline{R}}_{k+1}(M \otimes_R \overline{R},L) = 0$. Hence, $M \otimes_R \overline{R} \in \mathcal{F}_{(n,k)}(\overline{R}\op)$.

Now for part (b), suppose $M \otimes_R \overline{R} \in \mathcal{F}_{(n,k)}(\overline{R}\op)$ and $L \in \mathcal{FP}_n(R)$. We want to show $\Tor^R_{k+1}(M,L) = 0$. By \cite[Proposition 1, page 27]{Bourbaki}, this is equivalent to showing that $\Tor^R_{k+1}(M,L) \otimes_R \overline{R} = 0$, since $\overline{R}$ is faithfully flat over $R$. By \cite[Theorem 2.1.11]{EJ00}, we have that $\Tor^R_{k+1}(M,L) \otimes_R S \cong \Tor^{\overline{R}}_{k+1}(M \otimes_R \overline{R}, L \otimes_R \overline{R})$. On the other hand, by Proposition \ref{prop:change_FPn}, we know that $L \otimes_R \overline{R} \in \mathcal{FP}_n(\overline{R})$. Then, $\Tor^{\overline{R}}_{k+1}(M \otimes_R \overline{R}, L \otimes_R \overline{R}) = 0$. Therefore, $\Tor^R_{k+1}(M,L) \otimes_R \overline{R} = 0$, that is, $\Tor^R_{k+1}(M,L) = 0$ and so $M \in \mathcal{F}_{(n,k)}(R\op)$.
\end{proof}

Given a functor $F \colon \mathcal{D}_1 \longrightarrow \mathcal{D}_2$ between model categories, recall that $F$ is said to \emph{reflect} a property of morphisms if, given a morphism $f$ in $\mathcal{D}_1$, if $F(f)$ has the property so does $f$. The following result is a consequence of the previous proposition and the techniques from Lemma \ref{lem:proj_flat_Quillen}.

\begin{theorem}\label{theo:final1}
Let $\varphi \colon R \to \overline{R}$ be a ring homomorphism making $\overline{R}$ a finitely generated projective module in $\Mod(R)$. The following statements hold true for every $n \geq 2$ and $k \geq 0$:
\begin{itemize}
\item[{\rm (a)}] The induction $- \otimes_R \overline{R} \colon \Ch(R\op) \longrightarrow \Ch(\overline{R}\op)$ is a left Quillen functor from $\mathfrak{M}^{\rm dw\mbox{-}flat}_{(n,k)}(R\op)$ to $\mathfrak{M}^{\rm dw\mbox{-}flat}_{(n,k)}(\overline{R}\op)$, which is a Quillen equivalence if, and only if, $\varphi$ is an isomorphism of rings. If in addition $R$ and $\overline{R}$ are commutative, then $- \otimes_R \overline{R}$ reflects cofibrations and trivial cofibrations between $\mathfrak{M}^{\rm dw\mbox{-}flat}_{(n,k)}(R\op)$ and $\mathfrak{M}^{\rm dw\mbox{-}flat}_{(n,k)}(\overline{R}\op)$. This also applies to the case where $n \to \infty$.

\item[{\rm (b)}] The induction $- \otimes_R \overline{R} \colon \Ch(R\op) \longrightarrow \Ch(\overline{R}\op)$ is a left Quillen functor from $\mathfrak{M}^{\rm flat}_{(\infty,k)}(R\op)$ to $\mathfrak{M}^{\rm flat}_{(\infty,k)}(\overline{R}\op)$, which is a Quillen equivalence if, and only if, $\varphi$ is an isomorphism of rings. If in addition $R$ and $\overline{R}$ are commutative, then $- \otimes_R \overline{R}$ reflects cofibrations and trivial cofibrations between $\mathfrak{M}^{\rm flat}_{(\infty,k)}(R\op)$ and $\mathfrak{M}^{\rm flat}_{(\infty,k)}(\overline{R}\op)$.
\end{itemize}
\end{theorem}

We are also interested in presenting the analogous of Theorem \ref{theo:final1} for $\FP_n$-injective dimensions. This interest is motivated by the fact that if $\varphi \colon R \to \overline{R}$ is a ring homomorphism and $\Ch(R)$ and $\Ch(\overline{R})$ are equipped with the injective model structures, then the induction will be a left Quillen functor if, and only if, $\varphi$ makes $\overline{R}$ into a flat left $R$-module, and again, this will be a Quillen equivalence if, and only if, $\varphi$ is an isomorphism (See \cite[Section 2.3]{HoveyBook}). Note that, in this case, if $I$ is an injective module in $\Mod(\overline{R})$, then we have that $\Hom_R(-,U(I)) \cong \Hom_{\overline{R}}(\overline{R} \otimes_R -,I)$ is an exact functor since $\overline{R}$ is flat over $R$. We generalize this fact in the following result.

\begin{proposition}\label{prop:change_FPn-injective}
Let $\varphi \colon R \to \overline{R}$ be a ring homomorphism. The following statements hold:
\begin{itemize}
\item[{\rm (a)}] If $\varphi$ makes $\overline{R}$ a faithfully flat right $R$-module and $N \in \mathcal{I}_{(n,k)}(\overline{R})$, then $U(N) \in \mathcal{I}_{(n,k)}(R)$.

\item[{\rm (b)}] If $\varphi$ is an isomorphism that makes $\overline{R}$ a finitely generated projective left $R$-module and a faithfully flat right $R$-module, then $N \in \mathcal{I}_{(n,k)}(\overline{R})$ whenever $U(N) \in \mathcal{I}_{(n,k)}(R)$.
\end{itemize}
\end{proposition}

\begin{proof}
For (a) and (b), we only prove the case where $k = 0$. Let us first start with (a). Suppose $N \in \mathcal{I}_n(\overline{R})$ and $L \in \mathcal{FP}_n(R)$. Then, we have an exact sequence $0 \to L' \to F \to L \to 0$ in $\Mod(R)$ with $F$ finitely generated and free, and $L' \in \mathcal{FP}_n(R)$. Using the adjunction $(\overline{R} \otimes_R -, U)$, along with the fact that the functor $\overline{R} \otimes_R - \colon \Mod(R) \longrightarrow \Mod(\overline{R})$ is exact, we can obtain the following commutative diagram with exact rows

\[
\begin{tikzpicture}
\matrix (m) [matrix of math nodes, row sep=3em, column sep=0.75em]
{
0 & \Hom_R(L,U(N)) & \Hom_R(F,U(N)) & \Hom_R(L',U(N)) & \Ext^1_R(L,U(N)) & 0 \\
0 & \Hom_{\overline{R}}(\overline{R} \otimes_R L, N) & \Hom_{\overline{R}}(\overline{R} \otimes_R F, N) & \Hom_{\overline{R}}(\overline{R} \otimes_R L', N) & \Ext^1_{\overline{R}}(\overline{R} \otimes_R L, N) & 0 \\
};
\path[->]
(m-1-1) edge (m-1-2) (m-1-2) edge (m-1-3) (m-1-3) edge (m-1-4) (m-1-4) edge (m-1-5) (m-1-5) edge (m-1-6)
(m-2-1) edge (m-2-2) (m-2-2) edge (m-2-3) (m-2-3) edge (m-2-4) (m-2-4) edge (m-2-5) (m-2-5) edge (m-2-6)
(m-1-2) edge node[right] {\footnotesize$\cong$} (m-2-2) (m-1-3) edge node[right] {\footnotesize$\cong$} (m-2-3) (m-1-4) edge node[right] {\footnotesize$\cong$} (m-2-4) (m-1-5) edge (m-2-5)
;
\end{tikzpicture}
\]
where $\Ext^1_R(F,U(N)) = 0$ since $F$ is projective, and $\Ext^1_R(\overline{R} \otimes_R F, N) = 0$ by Proposition \ref{prop:change_FPn}. It follows that $\Ext^1_R(L,U(N)) \cong \Ext^1_{\overline{R}}(\overline{R} \otimes_R L, N) = 0$, that is, $U(N) \in \mathcal{I}_n(R)$.

For (b), suppose that $U(N) \in \mathcal{I}_n(R)$ and $L \in \mathcal{FP}_n(\overline{R})$. Since $\varphi$ is an isomorphism, the pair $(\overline{R} \otimes_R -, U) \colon \Mod(R) \longleftrightarrow \Mod(\overline{R})$ is an adjoint equivalence, and so $\overline{R} \otimes_R U(L) \cong L$. So, it suffices to show $\Ext^1_{\overline{R}}(\overline{R} \otimes_R U(L), N) = 0$, but this follows by the previous diagram and the fact that $U(L) \in \mathcal{FP}_n(R)$ by Proposition \ref{prop:change_FPn}.
\end{proof}

\begin{theorem}\label{theo:final2}
Let $\varphi \colon R \to \overline{R}$ be a ring homomorphism. The following statements hold true for every $\infty \geq n \geq 1$ and $k \geq 0$:
\begin{itemize}
\item[{\rm (a)}] If $\varphi$ makes $\overline{R}$ a faithfully flat right $R$-module, then $U \colon \Ch(\overline{R}) \longrightarrow \Ch(R)$ is:
\begin{itemize}
\item[$\bullet$] A right Quillen functor from $\mathfrak{M}^{\rm dw\mbox{-}inj}_{(n,k)}(\overline{R})$ to $\mathfrak{M}^{\rm dw\mbox{-}inj}_{(n,k)}(R)$, which is a Quillen equivalence if, and only if, $\varphi$ is an isomorphism.

\item[$\bullet$] A right Quillen functor from $\mathfrak{M}^{\rm inj}_{(\infty,k)}(\overline{R})$ to $\mathfrak{M}^{\rm inj}_{(\infty,k)}(R)$, which is a Quillen equivalence if, and only if, $\varphi$ is an isomorphism of rings.
\end{itemize}

\item[{\rm (b)}] If $\varphi$ is an isomorphism that makes $\overline{R}$ a finitely generated projective left $R$-module and a faithfully flat right $R$-module, then the forgetful functor reflects:
\begin{itemize}
\item[$\bullet$] Fibrations and trivial fibrations between $\mathfrak{M}^{\rm dw\mbox{-}inj}_{(n,k)}(\overline{R})$ and $\mathfrak{M}^{\rm dw\mbox{-}inj}_{(n,k)}(R)$.

\item[$\bullet$] Fibrations and trivial fibrations between $\mathfrak{M}^{\rm inj}_{(\infty,k)}(\overline{R})$ and $\mathfrak{M}^{\rm inj}_{(\infty,k)}(R)$.
\end{itemize}
\end{itemize}
\end{theorem}


\section*{\textbf{Acknowledgements}}
The first author is supported by  the National Natural Science Foundation of China  (No. 11571164), the University Postgraduate Research and Innovation Project of Jiangsu Province 2016 (No. KYZZ16\_0034), and Nanjing University Innovation and Creative Program for PhD candidate (No. 2016011).  The second author is supported by a DGAPA-UNAM (Direcci\'on General de Asuntos del Personal Acad\'emico - Universidad Nacional Aut\'onoma de M\'exico) postdoctoral fellowship. The authors would like to thank Professor Zhaoyong Huang for his careful guidance and helpful suggestions.

\bibliographystyle{alpha}
\bibliography{RefRelcom}

\begin{thebibliography}{AEGRO01}

\bibitem[AEGRO01]{AEGO01}
S.~Tempest Aldrich, Edgar~E. Enochs, J.~R. Garc\'{\i}a~Rozas, and Luis
  Oyonarte.
\newblock Covers and envelopes in {G}rothendieck categories: flat covers of
  complexes with applications.
\newblock {\em J. Algebra}, 243(2):615--630, 2001.

\bibitem[AF91]{AF91}
Luchezar~L. Avramov and Hans-Bj{\o}rn Foxby.
\newblock Homological dimensions of unbounded complexes.
\newblock {\em J. Pure Appl. Algebra}, 71(2-3):129--155, 1991.

\bibitem[BG16]{BG16}
Daniel Bravo and James Gillespie.
\newblock Absolutely clean, level, and {G}orenstein {AC}-injective complexes.
\newblock {\em Comm. Algebra}, 44(5):2213--2233, 2016.

\bibitem[BGH14]{BGH14}
Daniel Bravo, James Gillespie, and Mark Hovey.
\newblock The stable module category of a general ring.
\newblock 2014.

\bibitem[Bou89]{Bourbaki}
Nicolas Bourbaki.
\newblock {\em Commutative algebra. {C}hapters 1--7}.
\newblock Elements of Mathematics (Berlin). Springer-Verlag, Berlin, 1989.
\newblock Translated from the French, Reprint of the 1972 edition.

\bibitem[BP17]{BP16}
Daniel Bravo and Marco~A. P\'erez.
\newblock Finiteness conditions and cotorsion pairs.
\newblock {\em J. Pure Appl. Algebra}, 221(6):1249--1267, 2017.

\bibitem[Bro75]{Brown}
Kenneth~S. Brown.
\newblock Homological criteria for finiteness.
\newblock {\em Comment. Math. Helv.}, 50:129--135, 1975.

\bibitem[CD96]{DingChen}
Jianlong Chen and Nanqing Ding.
\newblock On {$n$}-coherent rings.
\newblock {\em Comm. Algebra}, 24(10):3211--3216, 1996.

\bibitem[Cos94]{Co94}
D.~L. Costa.
\newblock Parameterizing families of non-{N}oetherian rings.
\newblock {\em Comm. Algebra}, 22(10):3997--4011, 1994.

\bibitem[DS04]{DS}
Daniel Dugger and Brooke Shipley.
\newblock {$K$}-theory and derived equivalences.
\newblock {\em Duke Math. J.}, 124(3):587--617, 2004.

\bibitem[EGR98]{EG98}
Edgar~E. Enochs and J.~R. Garc\'{\i}a~Rozas.
\newblock Flat covers of complexes.
\newblock {\em J. Algebra}, 210(1):86--102, 1998.

\bibitem[EJ00]{EJ00}
Edgar~E. Enochs and Overtoun M.~G. Jenda.
\newblock {\em Relative homological algebra}, volume~30 of {\em De Gruyter
  Expositions in Mathematics}.
\newblock Walter de Gruyter \& Co., Berlin, 2000.

\bibitem[EJ11]{EJ11}
Edgar~E. Enochs and Overtoun M.~G. Jenda.
\newblock {\em Relative homological algebra. {V}olume 2}, volume~54 of {\em De
  Gruyter Expositions in Mathematics}.
\newblock Walter de Gruyter GmbH \& Co. KG, Berlin, 2011.

\bibitem[ELR02]{EnochsKaplansky}
Edgar~E. Enochs and J.~A. L\'opez-Ramos.
\newblock Kaplansky classes.
\newblock {\em Rend. Sem. Mat. Univ. Padova}, 107:67--79, 2002.

\bibitem[Eno81]{En81}
Edgar~E. Enochs.
\newblock Injective and flat covers, envelopes and resolvents.
\newblock {\em Israel J. Math.}, 39(3):189--209, 1981.

\bibitem[ER97]{EG97}
Edgar~E. Enochs and J.~R.~Garc\'{\i}a Rozas.
\newblock Tensor products of complexes.
\newblock {\em Math. J. Okayama Univ.}, 39:17--39 (1999), 1997.

\bibitem[ET01]{EkTr}
Paul~C. Eklof and Jan Trlifaj.
\newblock How to make {E}xt vanish.
\newblock {\em Bull. London Math. Soc.}, 33(1):41--51, 2001.

\bibitem[Fie71]{Fe71}
D.~J. Fieldhouse.
\newblock Character modules.
\newblock {\em Comment. Math. Helv.}, 46:274--276, 1971.

\bibitem[Fie72]{Fe72}
David~J. Fieldhouse.
\newblock Character modules, dimension and purity.
\newblock {\em Glasgow Math. J.}, 13:144--146, 1972.

\bibitem[Fir15]{First}
Uriya~A. First.
\newblock Rings that are morita equivalent to their opposites.
\newblock {\em J. Algebra}, 430:26--61, 2015.

\bibitem[GH16]{GH16}
Zenghui Gao and Zhaoyong Huang.
\newblock Weak injective and weak flat complexes.
\newblock {\em Glasgow Math. J.}, 58(3):539--557, 2016.

\bibitem[Gil04]{GillespieFlat}
James Gillespie.
\newblock The flat model structure on {${\rm Ch}(R)$}.
\newblock {\em Trans. Amer. Math. Soc.}, 356(8):3369--3390, 2004.

\bibitem[Gil08]{GillespieDegree}
James Gillespie.
\newblock Cotorsion pairs and degreewise homological model structures.
\newblock {\em Homology Homotopy Appl.}, 10(1):283--304, 2008.

\bibitem[Gil17]{GillespieModels}
James Gillespie.
\newblock Models for homotopy categories of injectives and {G}orenstein
  injectives.
\newblock {\em Comm. Algebra}, 45(6):2520--2545, 2017.

\bibitem[Gla89]{Glaz}
Sarah Glaz.
\newblock {\em Commutative coherent rings}, volume 1371 of {\em Lecture Notes
  in Mathematics}.
\newblock Springer-Verlag, Berlin, 1989.

\bibitem[GR99]{GR99}
J.~R. Garc\'{\i}a~Rozas.
\newblock {\em Covers and envelopes in the category of complexes of modules},
  volume 407 of {\em Chapman \& Hall/CRC Research Notes in Mathematics}.
\newblock Chapman \& Hall/CRC, Boca Raton, FL, 1999.

\bibitem[GT06]{Gobel}
R\"udiger G\"obel and Jan Trlifaj.
\newblock {\em Approximations and endomorphism algebras of modules}, volume~41
  of {\em De Gruyter Expositions in Mathematics}.
\newblock Walter de Gruyter GmbH \& Co. KG, Berlin, 2006.

\bibitem[GW15]{GW15}
Zenghui Gao and Fanggui Wang.
\newblock Weak injective and weak flat modules.
\newblock {\em Comm. Algebra}, 43(9):3857--3868, 2015.

\bibitem[Hir03]{Hir}
Philip~S. Hirschhorn.
\newblock {\em Model categories and their localizations}, volume~99 of {\em
  Mathematical Surveys and Monographs}.
\newblock American Mathematical Society, Providence, RI, 2003.

\bibitem[HJ09]{duality}
Henrik Holm and Peter J{\o}rgensen.
\newblock Cotorsion pairs induced by duality pairs.
\newblock {\em J. Commut. Algebra}, 1(4):621--633, 2009.

\bibitem[Hov99]{HoveyBook}
Mark Hovey.
\newblock {\em Model categories}, volume~63 of {\em Mathematical Surveys and
  Monographs}.
\newblock American Mathematical Society, Providence, RI, 1999.

\bibitem[Hov02]{HoveyPaper}
Mark Hovey.
\newblock Cotorsion pairs, model category structures, and representation
  theory.
\newblock {\em Math. Z.}, 241(3):553--592, 2002.

\bibitem[Hov07]{HoveyPaperPro}
Mark Hovey.
\newblock Cotorsion pairs and model categories.
\newblock In {\em Interactions between homotopy theory and algebra}, volume 436
  of {\em Contemp. Math.}, pages 277--296. Amer. Math. Soc., Providence, RI,
  2007.

\bibitem[IEA12]{Estrada}
M.~Cort\'es Izurdiaga, S.~Estrada, and P.~A.~Guil Asensio.
\newblock A model structure approach to the finitistic dimension conjectures.
\newblock {\em Math. Nachr.}, 285(7):821--833, 2012.

\bibitem[Jac85]{Jacobson1}
Nathan Jacobson.
\newblock {\em Basic algebra. {I}}.
\newblock W. H. Freeman and Company, New York, second edition, 1985.

\bibitem[P{\'e}r16a]{Perez16}
Marco~A. P{\'e}rez.
\newblock Homological dimensions and abelian model structures on chain
  complexes.
\newblock {\em Rocky Mountain J. Math.}, 46(3):951--1010, 2016.

\bibitem[P{\'e}r16b]{PerezBook}
Marco~A. P{\'e}rez.
\newblock {\em {Introduction to abelian model structures and Gorenstein
  homological dimensions.}}
\newblock Boca Raton, FL: CRC Press, 2016.

\bibitem[Rot09]{Rotman}
Joseph~J. Rotman.
\newblock {\em An introduction to homological algebra}.
\newblock Universitext. Springer, New York, second edition, 2009.

\bibitem[Ste70]{St70}
Bo~Stenstr\"om.
\newblock Coherent rings and {$F\,P$}-injective modules.
\newblock {\em J. London Math. Soc. (2)}, 2:323--329, 1970.

\bibitem[Ste75]{StBook}
Bo~Stenstr\"om.
\newblock {\em Rings of quotients: An introduction to methods of ring theory}.
\newblock Springer-Verlag, New York-Heidelberg, 1975.
\newblock Die Grundlehren der Mathematischen Wissenschaften, Band 217.

\bibitem[Wei94]{We94}
Charles~A. Weibel.
\newblock {\em An introduction to homological algebra}, volume~38 of {\em
  Cambridge Studies in Advanced Mathematics}.
\newblock Cambridge University Press, Cambridge, 1994.

\bibitem[WL11]{WL11}
Zhanping Wang and Zhongkui Liu.
\newblock F{P}-injective complexes and {FP}-injective dimension of complexes.
\newblock {\em J. Aust. Math. Soc.}, 91(2):163--187, 2011.

\bibitem[Yan12]{Ya}
Xiaoyan Yang.
\newblock Cotorsion pairs of complexes.
\newblock In {\em Proceedings of the {I}nternational {C}onference on {A}lgebra
  2010}, pages 697--703. World Sci. Publ., Hackensack, NJ, 2012.

\bibitem[YL10]{YL10}
Xiaoyan Yang and Zhongkui Liu.
\newblock F{P}-injective complexes.
\newblock {\em Comm. Algebra}, 38(1):131--142, 2010.

\end{thebibliography}

\end{document}